\documentclass[twoside,11pt]{article}
\usepackage{jmlr2e_arxiv}
\usepackage{amsmath}
\usepackage{amsfonts}
\usepackage{amssymb}
\usepackage{algorithm,algorithmic}
\usepackage{subfig}
\usepackage{multirow}
\usepackage{color}
\usepackage{pslatex}
\usepackage{pgf}
\usepackage{url}
\usepackage{tikz}
\usepackage{textcomp}
\usetikzlibrary{arrows,shapes,snakes,automata,backgrounds,petri}
\graphicspath{{./images/}}

\def\x{{\mathbf x}}
\def\z{{\mathbf z}}
\def\1{{\mathbf 1}}

\def\v{{\mathbf v}}
\def\X{{\mathbf X}}

\def\Cmat{{\mathbf C}}
\def\Rmat{{\mathbf R}}

\def\gammab{{\boldsymbol\gamma}}

\def\varepsilonb{{\boldsymbol\varepsilon}}

\def\y{{\mathbf y}}
\def\w{{\mathbf w}}
\def\tildew{{\mathbf{\tilde{w}}}}

\def\Y{{\mathbf Y}}

\def\Q{{\mathbf Q}}

\def\W{{\mathbf W}}

\def\N{{\mathcal N}}

\def\C{{\mathcal C}}
\def\GG{{\mathcal G}}

\def\L{{\mathcal L}}

\def\H{{\mathcal H}}

\def\e{{\mathbf e}}

\def\u{{\mathbf u}}

\def\Real{{\mathbb R}}

\def\U{{\mathbf U}}
\def\u{{\mathbf u}}
\def\A{{\mathbf A}}

\def\V{{\mathbf V}}

\def\I{{\mathbf I}}

\def\argmin{\operatornamewithlimits{arg\,min}}

\def\supp{\operatorname{Supp}}
\def\sign{\operatorname{sign}}
\def\st{~~\text{s.t.}~~}
\def\defin{\triangleq}

\newcommand{\R}[1]{\mathbb{R}^{#1}}
\newcommand{\RR}[2]{\mathbb{R}^{#1 \times #2}}
\newcommand{\G}{\mathcal{G}}

\newcommand{\NormDeux}[1]{\left\|#1\right\|_2}
\newcommand{\NormInf}[1]{\left\|#1\right\|_{\infty}}
\newcommand{\NormFro}[1]{\left\|#1\right\|_{\text{F}}}

\newcommand{\IntSet}[1]{ \{1,\ldots, #1\}}
\newcommand{\InSet}[1]{ \{1,\ldots, #1\}}

\def \xib{{\boldsymbol\xi}}
\def \xibbar{{\boldsymbol{\overline{\xi}}}}

\def \kappab {{\boldsymbol\kappa}}

\long\def\symbolfootnote[#1]#2{\begingroup\def\thefootnote{\fnsymbol{footnote}}\footnote[#1]{#2}\endgroup} 
\renewcommand{\cite}{\citep}

\ShortHeadings{Convex and Network Flow Optimization for Structured Sparsity}{Mairal, Jenatton, Obozinski and Bach}
\firstpageno{1}

\begin{document}

\title{Convex and Network Flow Optimization for Structured Sparsity}

\author{\name Julien Mairal\thanks{These authors contributed equally.}~\hspace*{0.08cm}\thanks{When most of this work was conducted, all authors were affiliated to INRIA, WILLOW Project-Team.} \email julien@stat.berkeley.edu \\
        \addr Department of Statistics \\
University of California\\
Berkeley, CA 94720-1776, USA \\
   \name Rodolphe Jenatton$^*$$^\dagger$ \email rodolphe.jenatton@inria.fr \\
        \name Guillaume Obozinski$^\dagger$ \email guillaume.obozinski@inria.fr \\
        \name Francis Bach$^\dagger$ \email francis.bach@inria.fr \\
        \addr INRIA - SIERRA Project-Team\\
         Laboratoire d'Informatique de l'Ecole Normale Sup\'erieure (INRIA/ENS/CNRS UMR 8548)\\
         23, avenue d'Italie 75214 Paris CEDEX 13, France.
       }


\maketitle

\vspace*{0.2cm}
\begin{abstract}
We consider a class of learning problems regularized by a structured
sparsity-inducing norm defined as the sum of $\ell_2$- or $\ell_\infty$-norms
over groups of variables. Whereas much effort has been put in developing
fast optimization techniques when the groups are disjoint or embedded in a
hierarchy, we address here the case of general overlapping groups.  To this
end, we present two different strategies: On the one hand, we show that the
proximal operator associated with a sum of $\ell_\infty$-norms can be computed
exactly in polynomial time by solving a \emph{quadratic min-cost flow problem},
allowing the use of accelerated proximal gradient methods.  On the other hand,
we use proximal splitting techniques, and address an equivalent formulation
with non-overlapping groups, but in higher dimension and with
additional constraints.  We propose efficient and scalable algorithms exploiting these two
strategies, which are significantly
faster than alternative approaches.  We illustrate these
methods with several problems such as CUR matrix factorization, multi-task
learning of tree-structured dictionaries, background subtraction in video
sequences, image denoising with wavelets, and topographic dictionary learning
of natural image patches.
\end{abstract}

\begin{keywords}
Convex optimization, proximal methods, sparse coding, structured sparsity, matrix factorization, network flow optimization, alternating direction method of multipliers.
\end{keywords}
\vspace*{0.2cm}

\section{Introduction}
Sparse linear models have become a popular framework for dealing with various
unsupervised and supervised tasks in machine learning and signal processing.
In such models, linear combinations of small sets of variables are selected to
describe the data.  Regularization by the $\ell_1$-norm has emerged as a
powerful tool for addressing this variable selection problem,
relying on both a well-developed theory~\citep[see][and references
therein]{tibshirani,chen,mallat,tsybakov,wainwright} and efficient
algorithms~\citep{efron,nesterov,beck,needell,combette}.

The $\ell_1$-norm primarily encourages sparse solutions, regardless of the
potential structural relationships (e.g., spatial, temporal or hierarchical)
existing between the variables.  Much effort has recently been devoted to
designing sparsity-inducing regularizations capable of encoding higher-order
information about the patterns of non-zero
coefficients~\citep{cehver,jenatton,jacob,zhao,he,huang,baraniuk,Micchelli2010}, with successful
applications in bioinformatics~\citep{jacob,kim3}, topic
modeling~\citep{jenatton3,jenatton4} and computer vision~\citep{cehver,huang, Jenatton2010}.  By
considering sums of norms of appropriate subsets, or \textit{groups}, of
variables, these regularizations control the sparsity patterns of the
solutions.  The underlying optimization is usually difficult, in part because
it involves nonsmooth components. 

Our first strategy uses proximal gradient methods, which have proven
to be effective in this context, essentially because of their fast convergence
rates and their ability to deal with large problems~\cite{nesterov,beck}.  They
can handle differentiable loss functions with Lipschitz-continuous gradient, and we
show in this paper how to use them with a regularization term composed of 
a sum of $\ell_\infty$-norms. 
The second strategy we consider exploits proximal splitting methods~\citep[see][and references
therein]{combettes2,combette,goldfarb,tomioka,qin,boyd2}, which builds upon an equivalent
formulation with non-overlapping groups, but in a higher dimensional space and with additional
constraints.\footnote{The idea of using this class of algorithms for solving structured sparse
problems was first suggested to us by Jean-Christophe
Pesquet and Patrick-Louis Combettes. It was also suggested to us later by Ryota
Tomioka, who briefly mentioned this possibility in \citep{tomioka}. It can also
briefly be found in~\citep{boyd2}, and in details in the work of~\citet{qin} which 
was conducted as the same time as ours. It was also used in a related context by~\citet{sprechmann} for
solving optimization problems with hierarchical norms.
} 
More precisely, we make four main contributions:
\begin{itemize}
\item We show that the \emph{proximal operator} associated with the sum of $\ell_\infty$-norms
with overlapping groups can be computed efficiently and exactly by solving a \emph{quadratic min-cost flow}
problem, thereby establishing a connection with the network flow optimization
literature.\footnote{Interestingly, this is not the first time that network
flow optimization tools have been used to solve sparse regularized problems
with proximal methods.  Such a connection was recently established by
\citet{chambolle} in the context of total variation regularization, and similarly by~\citet{Hoefling} for the fused Lasso. One can also find the use of maximum flow problems for non-convex penalties in the work of~\citet{cehver} which combines Markov random fields and sparsity.} This is the main contribution of the paper, which allows us to 
use proximal gradient methods in the context of structured sparsity.

\item We prove that the dual norm of the sum of $\ell_\infty$-norms can
also be evaluated efficiently, which enables us to compute duality gaps for the
corresponding optimization problems.

\item We present proximal splitting methods for solving 
structured sparse regularized problems.

\item We demonstrate that our methods are relevant for various applications
whose practical success is made possible by our algorithmic tools and efficient implementations.
First, we introduce a new CUR matrix factorization technique exploiting structured sparse regularization, 
built upon the links drawn by~\citet{bien} between CUR decomposition~\citep{mahoney2009cur} and sparse regularization.
Then, we illustrate our algorithms with different tasks: video background subtraction, estimation of hierarchical structures for
dictionary learning of natural image patches~\citep{jenatton3,jenatton4}, wavelet image denoising with a structured sparse prior, and topographic dictionary learning of natural image patches~\citep{hyvarinen2,kavukcuoglu2,garrigues}.
\end{itemize}
Note that this paper extends a shorter version published in Advances in Neural Information Processing Systems~\citep{mairal10}, by adding new experiments (CUR matrix factorization, wavelet image denoising and topographic dictionary learning), presenting the proximal splitting methods, providing the full proofs of the optimization results, and adding numerous discussions.
\subsection{Notation}
Vectors are denoted by bold lower case letters and matrices by upper case ones.
We define for $q \geq 1$ the \mbox{$\ell_q$-norm} of a vector~$\x$ in~$\Real^m$ as
$\|\x\|_q \defin (\sum_{i=1}^m |\x_i|^q)^{{1}/{q}}$, where~$\x_i$ denotes the
$i$-th coordinate of~$\x$, and $\|\x\|_\infty \defin \max_{i=1,\ldots,m} |\x_i|
= \lim_{q \to \infty} \|\x\|_q$.  We also define the $\ell_0$-pseudo-norm as
the number of nonzero elements in a vector:\footnote{Note that it would
be more proper to write $\|\x\|_0^0$ instead of $\|\x\|_0$ to be consistent with the traditional notation $\|\x\|_q$.
However, for the sake of simplicity, we will keep this
notation unchanged in the rest of the paper.}
$\|\x\|_0 \defin \#\{i \st \x_i
\neq 0  \} = \lim_{q \to 0^+}  (\sum_{i=1}^m |\x_i|^q)$.  We consider the
Frobenius norm of a matrix~$\X$ in~$\Real^{m \times n}$: $\NormFro{\X} \defin
(\sum_{i=1}^m \sum_{j=1}^n \X_{ij}^2)^{{1}/{2}}$, where $\X_{ij}$ denotes the entry of~$\X$ at row $i$ and column $j$. Finally, for a scalar $y$, we denote $(y)_+ \defin \max(y,0)$.
For an integer $p > 0$, we denote by $2^{\{1,\dots,p\}}$ the powerset composed of the $2^p$ subsets of $\{1,\ldots,p\}$.

The rest of this paper is organized as follows: Section \ref{sec:related}
presents structured sparse models and related work.  Section \ref{sec:optim_prox} 
is devoted to proximal gradient algorithms, and 
 Section \ref{sec:optim_prox2}
to proximal splitting methods.
Section~\ref{sec:exp} presents several experiments and applications demonstrating the
effectiveness of our approach and Section~\ref{sec:ccl} concludes the paper.

\section{Structured Sparse Models} \label{sec:related}
We are interested in machine learning problems where the solution is not only known beforehand
to be sparse---that is, the solution has only a few non-zero coefficients, but also
to form non-zero patterns with a specific structure.
It is indeed possible to encode
additional knowledge in the regularization other than just sparsity. 
For instance, one may want the
non-zero patterns to be structured in the form of non-overlapping
groups~\citep{turlach,yuan,Stojnic2009,obozinski}, in a tree~\citep{zhao,bach4,jenatton3,jenatton4},
or in overlapping groups~\citep{jenatton,jacob,huang,baraniuk,cehver,he}, which
is the setting we are interested in here.

As for classical non-structured sparse models, there are basically two lines of
research, that either (A) deal with nonconvex and combinatorial formulations
that are in general computationally intractable and addressed with greedy
algorithms or (B) concentrate on convex relaxations solved with convex
programming methods.

\subsection{Nonconvex Approaches}
A first approach introduced by~\citet{baraniuk} consists in imposing that the
sparsity pattern of a solution (i.e., its set of non-zero coefficients) is
in a predefined subset of groups of variables $\GG \subseteq 2^{\{1,\ldots,p\}}$. 
Given this a priori knowledge, a greedy
algorithm~\citep{needell} is used to address the following nonconvex structured sparse
decomposition problem
\begin{displaymath}
 \min_{\w \in\Real^{p}} \frac{1}{2} \|\y-\X\w\|_2^2 \st \supp(\w) \in \GG ~~\text{and}~~ \|\w\|_0 \leq s,
\end{displaymath}
where $s$ is a specified sparsity level (number of nonzeros coefficients), $\y$ in $\Real^m$ is an observed signal, $\X$ is a design matrix in $\Real^{m \times p}$ and $\supp(\w)$ is the support of $\w$ (set of non-zero entries).

In a different approach motivated by the minimum description length
principle~\citep[see][]{barron}, \citet{huang} consider a collection of groups
$\GG \subseteq 2^{\{1,\ldots,p\}}$, and define a ``coding length'' for every
group in $\GG$, which in turn is used to define a coding length for every
pattern in $2^{\{1,\ldots,p\}}$. Using this tool, they propose a
regularization function $\text{cl}: \Real^p \to \Real$ such that for a vector $\w$ in
$\Real^p$, $\text{cl}(\w)$ represents the number of bits that are used for
encoding~$\w$.
The corresponding optimization problem is also addressed with a greedy procedure:
\begin{displaymath}
 \min_{\w \in\Real^{p}} \frac{1}{2} \|\y-\X\w\|_2^2 \st ~ \text{cl}(\w) \leq s,
\end{displaymath}
Intuitively, this formulation encourages solutions~$\w$ whose sparsity patterns have a
small coding length, meaning in practice that they can be represented by a union of a
small number of groups.  Even though they are related, this model is different
from the one of~\citet{baraniuk}.

These two approaches are encoding a priori knowledge on the shape of non-zero
patterns that the solution of a regularized problem should have. A different
point of view consists of modelling the zero patterns of the solution---that is,
define groups of variables that should be encouraged to be set to zero together.
After defining a set $\GG \subseteq 2^{\{1,\ldots,p\}}$ of such groups of variables, 
the following penalty can naturally be used as a regularization to induce the desired property
\begin{equation}
\psi(\w) \defin \sum_{g\in\GG} \eta_g \delta^g(\w),\ \text{with}\
\delta^g(\w) \defin
\begin{cases}
  1 & \text{if there exists}\ j\in g\ \text{such that}\ \w_j\neq 0,\\
  0 & \text{otherwise},
\end{cases}  \label{eq:nonconvex}
\end{equation}
where the $\eta_g$'s are positive weights.
This penalty was considered by \citet{bach6}, who showed that 
the convex envelope of such nonconvex functions (more precisely strictly
positive, non-increasing submodular functions of
$\supp(\w)$,~\citealp[see][]{fujishige2}) when restricted on the unit $\ell_\infty$-ball, are in fact types of structured
sparsity-inducing norms which are the topic of the next section.

\subsection{Convex Approaches with Sparsity-Inducing Norms}
In this paper, we are interested in convex regularizations which induce structured sparsity.
Generally, we consider the following optimization problem
\begin{equation}
   \min_{\w \in \Real^p} f(\w) + \lambda \Omega(\w), \label{eq:formulation}
\end{equation}
where $f: \Real^p \to \Real$ is a convex function (usually an empirical risk in machine
learning and a data-fitting term in signal processing), and $\Omega: \Real^p \to
\Real$ is a structured sparsity-inducing norm, defined~as
\begin{equation}
   \Omega(\w) \, \defin\, \sum_{g \in \GG} \eta_g \|\w_g\|, \label{eq:omega}
\end{equation}
where $\GG \subseteq 2^{\{1,\ldots,p\}}$ is a set of groups of variables, 
the vector~$\w_g$ in $\R{|g|}$ represents the coefficients of $\w$ indexed by
$g$ in $\GG$, the scalars~$\eta_g$ are positive weights, and $\|.\|$ denotes
the $\ell_2$- or $\ell_\infty$-norm. We now consider different cases:
\begin{itemize}
\item When $\GG$ is the set of singletons---that is $\GG \defin \{ \{1\},\{2\},\ldots,\{p\} \}$,
and all the $\eta_g$ are equal to one, $\Omega$ is the $\ell_1$-norm, 
which is well known to induce sparsity. This leads for instance to the Lasso~\citep{tibshirani} or equivalently to basis
pursuit~\citep{chen}. 
\item If $\GG$ is a partition of $\IntSet{p}$, i.e. the groups do not overlap,
variables are selected in groups rather than individually.  When the
coefficients of the solution are known to be organized in such a way,
explicitly encoding the a priori group structure in the regularization can
improve the prediction performance and/or interpretability of the learned
models~\citep{turlach,yuan,roth2,Stojnic2009,huang2,obozinski}. Such a penalty is commonly
called group-Lasso penalty.
\item When the groups overlap, $\Omega$ is still a norm and sets groups of variables
to zero together~\cite{jenatton}.  The latter setting has first been considered
for hierarchies~\cite{zhao,kim3,bach4,jenatton3,jenatton4}, and then extended to general
group structures~\cite{jenatton}.  Solving Eq.~(\ref{eq:formulation}) in this
context is a challenging problem which is the topic of this paper.
\end{itemize}
Note that other types of structured-sparsity inducing norms have also been
introduced, notably the approach of \citet{jacob}, which penalizes the
following quantity 
\begin{displaymath}
   \Omega'(\w) \, \defin\, \min_{ \xib=(\xib^g)_{g \in \GG} \in \Real^{p \times |\GG|}} \sum_{g \in \GG} \eta_g \|\xib^g\| \st  \w = \sum_{g \in \GG} \xib^g ~~\text{and}~~ \forall g,~ \supp(\xib^g) \subseteq g.
\end{displaymath}
This penalty, which is also a norm, can be seen as a convex relaxation of the
regularization introduced by~\citet{huang}, and encourages the sparsity pattern of the solution
to be a union of a small number of groups. Even though both~$\Omega$
and~$\Omega'$ appear under the terminology of ``structured sparsity with
overlapping groups'', they have in fact significantly different purposes and algorithmic treatments.
For example,~\citet{jacob} consider the problem of selecting genes in a gene
network which can be represented as the union of a few predefined pathways in
the graph (groups of genes), which overlap.  In this case, it is natural to use
the norm~$\Omega'$ instead of~$\Omega$.  On the other hand, we present a matrix
factorization task in Section~\ref{subsec:cur}, where the set of zero-patterns
should be a union of groups, naturally leading to the use of~$\Omega$.  Dealing
with~$\Omega'$ is therefore relevant, but out of the scope of this~paper.

\subsection{Convex Optimization Methods Proposed in the Literature}
Generic approaches to solve Eq.~(\ref{eq:formulation}) mostly rely on
subgradient descent schemes~\citep[see][]{bertsekas}, and interior-point
methods~\cite{boyd}.  These generic tools do not scale well to large problems
and/or do not naturally handle sparsity (the solutions they return may have
small values but no ``true'' zeros).  These two points prompt the need for
dedicated methods.

To the best of our knowledge, only a few recent papers have addressed problem Eq.~(\ref{eq:formulation})
with dedicated optimization procedures, and in fact, only when $\Omega$ is a linear combination
of $\ell_2$-norms.
In this setting, a first line of work deals with the non-smoothness of $\Omega$ by 
expressing the norm as the minimum over a set of smooth functions. 
At the cost of adding new variables (to describe the set of smooth functions), 
the problem becomes more amenable to optimization.
In particular, reweighted-$\ell_2$ schemes consist of approximating the norm $\Omega$ by successive quadratic upper bounds~\cite{argyriou, Rakotomamonjy2008, Jenatton2010, Micchelli2010}.
It is possible to show for instance that
$$
\Omega(\w) = \min_{(z_g)_{g \in \GG} \in \Real^{|\GG|}_+} \frac{1}{2} 
\bigg\{ 
\sum_{g\in\GG} \frac{\eta_g^2 \|\w_g\|_2^2}{z_g} + z_g 
\bigg\}.
$$
Plugging the previous relationship into Eq.~(\ref{eq:formulation}),
the optimization can then be performed by alternating between 
the updates of $\w$ and the additional variables $(z_g)_{g \in \GG}$.\footnote{Note
that such a scheme is interesting only if the optimization with respect to~$\w$ is simple, 
which is typically the case with the square loss function~\citep{bach5}.
Moreover, for this alternating scheme to be provably convergent,  
the variables $(z_g)_{g \in \GG}$ have to be bounded away from zero, 
resulting in solutions whose entries may have small values, but not ``true'' zeros.}
When the norm $\Omega$ is defined as a linear combination of $\ell_\infty$-norms, 
we are not aware of the existence of such variational formulations.

Problem~(\ref{eq:formulation}) has also been addressed with working-set algorithms~\citep{bach4, jenatton, Schmidt2010}.
The main idea of these methods is to solve a sequence of increasingly larger subproblems of~(\ref{eq:formulation}). 
Each subproblem consists of an instance of Eq.~(\ref{eq:formulation})
reduced to a specific subset of variables known as the \textit{working set}.
As long as some predefined optimality conditions are not satisfied, 
the working set is augmented with selected inactive variables~\cite[for more details, see][]{bach5}.

The last approach we would like to mention is that of ~\citet{chen2},
who used a smoothing technique introduced by~\citet{nesterov3}.  A smooth
approximation $\Omega_\mu$ of $\Omega$ is used, when $\Omega$ is a sum of
$\ell_2$-norms, and $\mu$ is a parameter controlling the trade-off between
smoothness of $\Omega_\mu$ and quality of the approximation.  Then, 
Eq.~(\ref{eq:formulation}) is solved with accelerated gradient
techniques~\citep{beck,nesterov} but $\Omega_\mu$ is substituted to the
regularization~$\Omega$. Depending on the required precision for solving the
original problem, this method provides a natural choice for the parameter
$\mu$, with a known convergence rate.
A drawback is that it requires to choose the precision of the optimization beforehand.
Moreover, since a $\ell_1$-norm is added to the smoothed~$\Omega_\mu$, 
the solutions returned by the algorithm might be sparse but possibly without respecting the structure encoded by~$\Omega$.
This should be contrasted with other smoothing techniques, e.g., the reweighted-$\ell_2$ scheme
we mentioned above, where the solutions are only approximately sparse.

\section{Optimization with Proximal Gradient Methods} \label{sec:optim_prox}
We address in this section the problem of solving Eq.~(\ref{eq:formulation}) 
under the following assumptions:
\begin{itemize}
 \item \emph{$f$ is differentiable with Lipschitz-continuous gradient.} For
machine learning problems, this hypothesis holds when $f$ is for example the square, logistic or
multi-class logistic loss~\citep[see][]{shawe}.
 \item \emph{$\Omega$ is a sum of $\ell_\infty$-norms.} 
 Even though the $\ell_2$-norm is sometimes used in the
literature~\citep{jenatton}, and is in fact used later in Section~\ref{sec:optim_prox2},
the $\ell_\infty$-norm is piecewise linear, and we take advantage of this property in this work.
\end{itemize}
To the best of our knowledge, no dedicated optimization method has been
developed for this setting.  Following~\citet{jenatton3,jenatton4} who tackled the
particular case of hierarchical norms, we propose to use proximal gradient
methods, which we now introduce.

\subsection{Proximal Gradient Methods}
Proximal methods have drawn increasing attention in the signal processing
\citep[e.g.,][and numerous references therein]{wright,
combette} and the machine learning communities \citep[e.g.,][and
references therein]{bach5}, especially because of their convergence rates
(optimal for the class of first-order techniques) and their ability to deal
with large nonsmooth convex problems~\citep[e.g.,][]{nesterov,beck}.

These methods are iterative procedures that can be seen as an
extension of gradient-based techniques when the objective function to
minimize has a nonsmooth part.
The simplest version of this class of methods linearizes at each iteration the
function $f$ around the current estimate $\tildew$, and this estimate is updated
as the (unique by strong convexity) solution of the \textit{proximal} problem,
defined~as:
\begin{displaymath}
    \min_{\w \in \R{p}} f(\tildew) + (\w - \tildew)^\top \nabla f(\tildew) + \lambda \Omega(\w) + {\displaystyle \frac{L}{2}}\|\w - \tildew\|_2^2.
\end{displaymath}
The quadratic term keeps the update in a neighborhood where $f$ is close to its
linear approximation, and $L\!>\!0$ is a parameter which is a upper bound on
the Lipschitz constant of $\nabla f$.  This problem can be equivalently
rewritten~as:
\begin{displaymath}
   \min_{\w \in \R{p}} {\displaystyle \frac{1}{2}} \big\|\tildew - \frac{1}{L} \nabla f(\tildew)-\w\big\|_2^2 + \frac{\lambda}{L} \Omega(\w),
\end{displaymath}
Solving \textit{efficiently} and exactly this problem allows to
attain the fast convergence rates of proximal methods, i.e., reaching a precision of $O(\frac{L}{k^2})$ in $k$ iterations.\footnote{Note, however, that fast convergence rates can also be achieved while solving approximately the 
proximal problem~\citep[see][for more details]{schmidt}.}
In addition, when the nonsmooth term~$\Omega$ is not present, the previous
proximal problem exactly leads to the standard gradient update rule.
More generally, we define the \textit{proximal operator}: 
\begin{definition}[Proximal Operator]~\newline
The proximal operator associated with our regularization term $\lambda\Omega$, which we denote by $\text{Prox}_{\lambda \Omega}$, is the function that maps a vector $\u \in \R{p}$ to the unique solution of
\begin{equation}\label{eq:prox_problem}
   \min_{\w \in \R{p}} \frac{1}{2} \NormDeux{\u-\w}^2 + \lambda  \Omega(\w).
\end{equation}
\end{definition}
This operator was initially introduced by~\citet{moreau} to generalize the 
projection operator onto a convex set. What makes proximal methods appealing to solve sparse
decomposition problems is that this operator can often be computed in closed form. For instance,
\begin{itemize}
   \item When $\Omega$ is the $\ell_1$-norm---that is $\Omega(\w)=\|\w\|_1$---
      the proximal operator is the well-known elementwise soft-thresholding
      operator,
      \begin{displaymath}
         \forall j \in \InSet{p},~~\u_j \mapsto \sign(\u_j)(|\u_j|-\lambda)_+~~~ = \begin{cases}
            0 & \text{if}~~ |\u_j| \leq \lambda \\
            \sign(\u_j)(|\u_j|-\lambda) & \text{otherwise}. 
         \end{cases}
      \end{displaymath}
   \item When $\Omega$ is a group-Lasso penalty with $\ell_2$-norms---that is,
      $\Omega(\u)=\sum_{g\in \GG}\|\u_g\|_2$, with $\GG$ being a partition of
      $\InSet{p}$, the proximal problem is \emph{separable} in every group, 
      and the solution is a generalization of the soft-thresholding operator to
      groups of variables:
      \begin{displaymath}
         \forall g\in \G~~, \u_g \mapsto \u_g-\Pi_{\|.\|_2 \leq \lambda}[\u_g] = \begin{cases}
            0 & \text{if}~~ \|\u_g\|_2 \leq \lambda \\
	    \frac{\|\u_g\|_2-\lambda}{\|\u_g\|_2}\u_g & \text{otherwise}, \\
         \end{cases} 
      \end{displaymath}
      where $\Pi_{\|.\|_2 \leq \lambda}$ denotes the orthogonal projection onto the ball of 
      the $\ell_2$-norm of radius $\lambda$.
   \item When $\Omega$ is a group-Lasso penalty with $\ell_\infty$-norms---that is,
      $\Omega(\u)=\sum_{g\in \G}\|\u_g\|_\infty$, with $\GG$ being a partition of
      $\InSet{p}$, the solution is a different group-thresholding
      operator:
      \begin{displaymath}
         \forall g\in \G,~~\u_g \mapsto \u_g - \Pi_{\|.\|_1 \leq \lambda}[\u_g],
      \end{displaymath}
      where $\Pi_{\|.\|_1 \leq \lambda}$ denotes the orthogonal projection onto
      the $\ell_1$-ball of radius $\lambda$, which can be solved in $O(p)$
      operations~\cite{brucker,maculan}. Note that when $\|\u_g\|_1 \leq \lambda$, we
      have a group-thresholding effect, with $\u_g - \Pi_{\|.\|_1 \leq \lambda}[\u_g]=0$.
  \item When $\Omega$ is a tree-structured sum of $\ell_2$- or $\ell_\infty$-norms as introduced by~\citet{zhao}---meaning that two groups are either disjoint or one is included in the other, the solution admits a closed form. 
Let $\preceq$ be a total order on $\G$ such that for $g_1, g_2$ in $\G$, $g_1 \preceq g_2$ if and only if either $g_1 \subset g_2$ or $g_1 \cap g_2=\emptyset$.\footnote{For a tree-structured set $\G$, such an order exists.} Then, if $g_1 \preceq \ldots \preceq g_{|\G|}$, and if we define $\text{Prox}^g$ as (a) the proximal operator $\u_g \mapsto \text{Prox}_{\lambda \eta_g \|\cdot\|}(\u_g)$ on the subspace corresponding to group $g$ and (b) the identity on the orthogonal, \citet{jenatton3,jenatton4} showed that:
\begin{equation}
\text{Prox}_{\lambda\Omega}=\text{Prox}^{g_m} \circ \ldots \circ \text{Prox}^{g_1},
\end{equation}
which can be computed in $O(p)$ operations. It also includes the sparse group Lasso (sum of group-Lasso penalty and $\ell_1$-norm) of~\citet{friedman2} and \citet{sprechmann}.
\end{itemize}
The first contribution of our paper is to address the case of general overlapping groups with~$\ell_\infty$-norm.

\subsection{Dual of the Proximal Operator}
We now show that, for a set ~$\G$ of general overlapping groups, a convex dual of the proximal problem~(\ref{eq:prox_problem})  can be reformulated as a \emph{quadratic min-cost flow problem}. We then propose an efficient algorithm to solve
it exactly, as well as a related algorithm to compute the dual norm of~$\Omega$.
We start by considering the dual formulation to problem~(\ref{eq:prox_problem}) introduced by~\citet{jenatton3,jenatton4}:
\begin{lemma}
   [Dual of the proximal problem,~\citealp{jenatton3,jenatton4}]
\label{lem:dual}~\newline
   Given $\u$ in $\R{p}$, consider the problem
\begin{equation}
   \min_{ \xib \in \RR{p}{|\G|} } \frac{1}{2} \| \u - \sum_{g \in \G}
   \xib^g  \|^2_2 ~~\mbox{ s.t. }~~\forall g\in\G,\
   \|\xib^g\|_1 \leq \lambda \eta_g ~~~\mbox{ and }~~~ \, \xib^g_j = 0 \, \mbox{ if }
   \, j \notin g ,\label{eq:dual_problem}
\end{equation}
where $\xib\! =\! (\xib^g)_{g \in \G}$ is in $\RR{p}{|\G|}$, and $\xib^g_j$ denotes the $j$-th coordinate
of the vector $\xib^g$. Then, every solution $\xib^\star\! =\! (\xib^{\star g})_{g \in \G}$ of Eq.~(\ref{eq:dual_problem}) satisfies $\w^\star\! =\! \u\!-\!\sum_{g \in \G}\xib^{\star g}$, where $\w^\star$ is the solution of Eq.~(\ref{eq:prox_problem}) when~$\Omega$ is a weighted sum of~$\ell_\infty$-norms.
\end{lemma}
Without loss of generality,\footnote{
Let $\xib^\star$ denote a solution of Eq.~(\ref{eq:dual_problem}).
Optimality conditions of Eq.~(\ref{eq:dual_problem}) derived in~\citet{jenatton3,jenatton4}
show that for all $j$ in $\IntSet{p}$, the signs of the
non-zero coefficients $\xib_j^{\star g}$ for $g$ in $\GG$ are the same as the signs of the entries $\u_j$.  
To solve Eq.~(\ref{eq:dual_problem}), one can therefore flip
the signs of the negative variables $\u_j$, then solve the modified dual formulation
(with non-negative variables), which gives the magnitude of the entries
$\xib_j^{\star g}$ (the signs of these being known).}
we assume from now on that the scalars $\u_j$
are all non-negative, and we constrain the entries of~$\xib$ to be so.
Such a formulation introduces $p|\G|$ dual variables which can be much greater than $p$, the number of primal variables,
but it removes the issue of overlapping regularization.  We now associate a graph
with problem~(\ref{eq:dual_problem}), on which the variables $\xib_j^g$, for $g$ in $\GG$ and~$j$ in $g$,
can be interpreted as measuring the components of a flow. 

\subsection{Graph Model} \label{subsec:graph}
Let $G$ be a directed graph $G=(V,E,s,t)$, where $V$ is a set of vertices,
$E\subseteq V\times V$ a set of arcs, $s$ a source, and $t$ a sink. 
For all arcs in $E$, we define a non-negative capacity constant,
and as done classically in the network flow literature~\citep{ahuja,bertsekas2}, we
define a \emph{flow} as a non-negative function on arcs that satisfies capacity
constraints on all arcs (the value of the flow on an arc is less than or equal
to the arc capacity) and conservation constraints on all vertices (the sum of
incoming flows at a vertex is equal to the sum of outgoing flows) except for
the source and the sink.
For every arc $e$ in $E$, we also define a real-valued cost function, which depends on the value of the flow
on~$e$.
We now introduce the \emph{canonical} graph $G$ associated with our optimization problem:
\begin{definition}[Canonical Graph]~\newline
Let $\GG \subseteq \InSet{p}$ be a set of groups, and $(\eta_g)_{g \in \GG}$ be positive weights. 
The canonical graph $G=(V,E,s,t)$
is the unique graph defined as follows:
\begin{enumerate}
\item $V = V_u \cup V_{gr}$, where $V_u$ is a vertex set of size $p$, one vertex being associated to 
each index $j$ in $\IntSet{p}$, and $V_{gr}$ is a vertex set of size $|\GG|$, one vertex per group~$g$ in~$\GG$. 
We thus have $|V|=|\GG|+p$. For simplicity, we identify groups $g$ in $\GG$ and indices $j$ in $\InSet{p}$ with
vertices of the graph, such that one can from now on refer to ``vertex $j$'' or ``vertex $g$''.
\item For every group $g$ in $\GG$, $E$ contains an arc $(s,g)$.
These arcs have capacity $\lambda\eta_g$ and zero cost. 
\item For every group $g$ in $\GG$, and every index $j$ in $g$, $E$
contains an arc $(g,j)$ with zero cost and infinite capacity. We
denote by $\xib_j^g$ the flow on this arc. 
\item For every index $j$ in $\IntSet{p}$, $E$ contains an arc $(j,t)$
with infinite capacity and a cost \mbox{$\frac{1}{2}(\u_j-\xibbar_j)^2$}, where $\xibbar_j$ is the flow on $(j,t)$.
\end{enumerate}
\end{definition}
Examples of canonical graphs are given in Figures
\ref{subfig:grapha}-\subref*{subfig:graphc} for three simple group structures.  
The flows $\xib_j^g$ associated
with $G$ can now be identified with the variables of problem
(\ref{eq:dual_problem}). Since we have assumed the entries of $\u$ to be non-negative, we can now reformulate Eq.~(\ref{eq:dual_problem}) as
\begin{equation}
   \min_{ \xib \in \Real_+^{p \times |\GG|}, \xibbar \in \Real^p} \sum_{j=1}^p \frac{1}{2} (\u_j - \xibbar_j)^2 \st \xibbar = \sum_{g \in \GG} \xib^g ~~\text{and}~~\forall g\in\GG,\
   \left\{ \sum_{j \in g} \xib_j^g \leq \lambda \eta_g ~~\text{and}~~ \supp(\xib^g) \subseteq g \right\}. \label{eq:dual_problem2}
\end{equation}
 Indeed,
\begin{itemize}
\item the only arcs with a cost are those leading to the sink, which have the form $(j,t)$, where $j$ is the index of a variable in $\InSet{p}$.
The sum of these costs is $\sum_{j=1}^p \frac{1}{2} (\u_j - \xibbar_j)^2$, which is the objective function minimized in Eq.~(\ref{eq:dual_problem2});
\item by flow conservation, we necessarily have $\xibbar_j=\sum_{g \in \GG}{\xib_j^g}$ in the canonical graph;
\item the only arcs with a capacity constraints are those coming out of the source, which have the form $(s,g)$, where $g$ is a group in $\GG$. 
By flow conservation, the flow on an arc $(s,g)$ is $\sum_{j \in g} \xib_j^g$ which should be less than $\lambda \eta_g$ by capacity constraints;
\item all other arcs have the form $(g,j)$, where $g$ is in $\GG$ and $j$ is in $g$. Thus, $\supp(\xib^g) \subseteq g$. 
\end{itemize}
Therefore we have shown that finding a flow \emph{minimizing the sum of the
costs} on such a graph is equivalent to solving
problem~(\ref{eq:dual_problem}). 
When some groups are included in others, the canonical graph can be simplified
to yield a graph with a smaller number of edges. Specifically, if $h$ and $g$
are groups with $h \subset g$, the edges $(g,j)$ for $j \in h$ carrying a flow
$\xib^g_j$ can be removed and replaced by a single edge $(g,h)$ of infinite
capacity and zero cost, carrying the flow $\sum_{j \in h} \xib^g_j$. This
simplification is illustrated in Figure~\ref{subfig:graphd}, with a graph
equivalent to the one of Figure~\ref{subfig:graphc}.  This does not change the
optimal value of~$\xibbar^\star$, which is the quantity of interest for
computing the optimal primal variable~$\w^\star$.  We present in
Appendix~\ref{appendix:equivalent} a formal definition of equivalent graphs.
These simplifications are useful in practice, since they reduce the number of
edges in the graph and improve the speed of our algorithms.
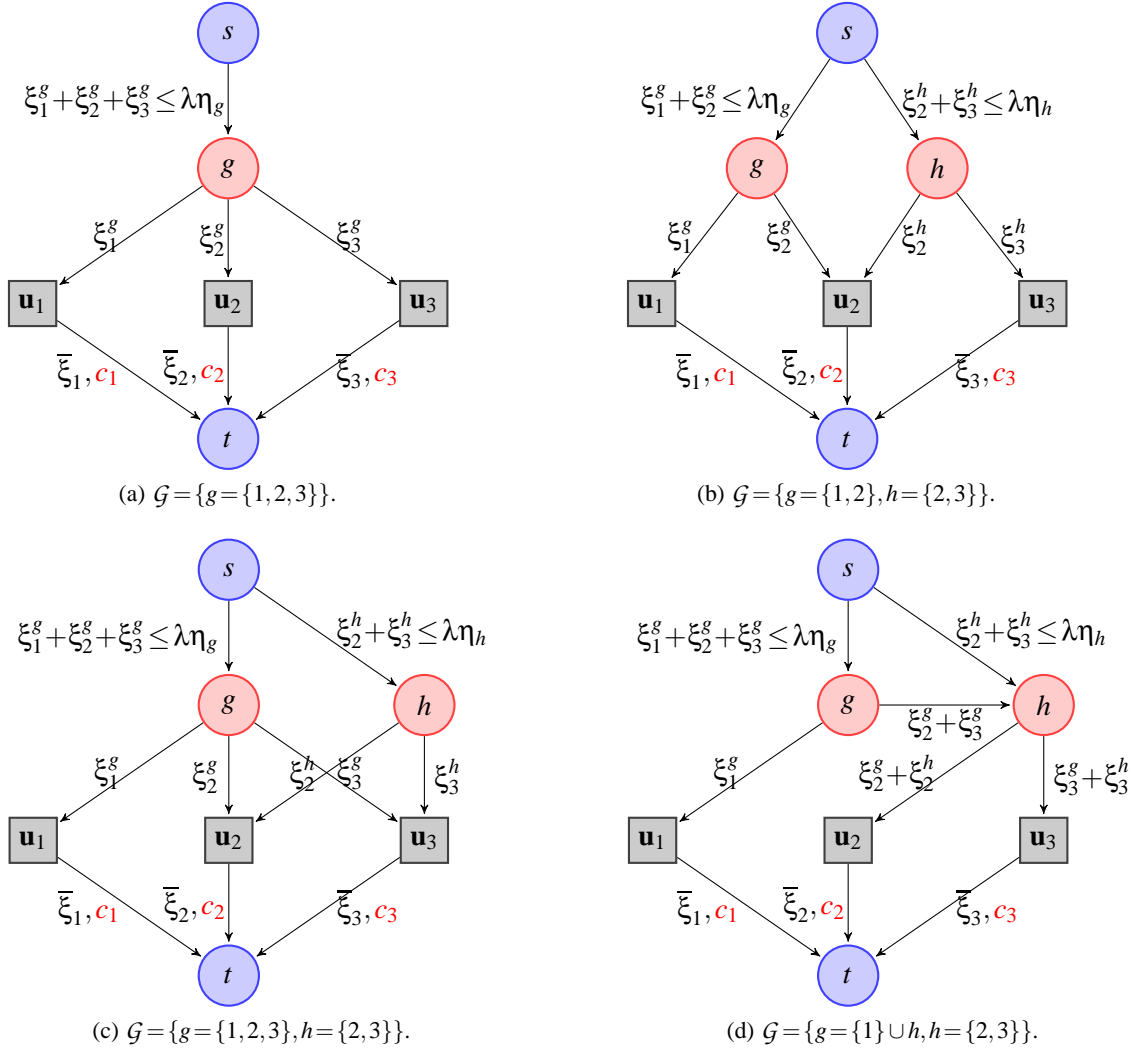
\begin{figure}[hbtp!]
\tikzstyle{source}=[circle,thick,draw=blue!75,fill=blue!20,minimum size=8mm]
\tikzstyle{sink}=[circle,thick,draw=blue!75,fill=blue!20,minimum size=8mm]
\tikzstyle{group}=[place,thick,draw=red!75,fill=red!20, minimum size=8mm]
\tikzstyle{var}=[rectangle,thick,draw=black!75,fill=black!20,minimum size=6mm]
\def\distnode{1.8cm}
\def\distnodex{2.6cm}
\tikzstyle{every label}=[red]
   \begin{center}
      \subfloat[$\GG\!=\!\{ g\!=\!\{1,2,3\} \}$.]{
      \begin{tikzpicture}[node distance=\distnode,>=stealth',bend angle=45,auto]
         \begin{scope}
            \node [source]   (s)                                    {$s$};
            \node [group]    (g1)  [below of=s]                      {$g$}
            edge  [pre] node[left,xshift=1mm] {$\xib^g_1 \!+\! \xib^g_2 \!+\! \xib^g_3 \!\leq\! \lambda \eta_g$} (s);
            \node [var] (u2) [below of=g1]                    {$\u_2$}
            edge  [pre] node[above, left,xshift=1mm] {$\xib^{g}_2$} (g1);
            \node [var] (u1)  [left of=u2, node distance=\distnodex] {$\u_1$}
            edge  [pre] node[above, left] {$\xib^{g}_1$} (g1);
            \node [var] (u3) [right of=u2, node distance=\distnodex] {$\u_3$}
            edge  [pre] node[above, right] {$\xib^{g}_3$} (g1);
            \node [sink] (si) [below of=u2] {$t$}
            edge [pre] node[above,left] {$\xibbar_1, \color{red} c_1$} (u1)
            edge [pre] node[above,left,xshift=1mm] {$\xibbar_2,\color{red} c_2$} (u2)
            edge [pre] node[above,right] {$\xibbar_3,\color{red} c_3$} (u3);
         \end{scope}
      \end{tikzpicture}\label{subfig:grapha}
      } \hfill 
      \subfloat[$\GG\!=\!\{ g\!=\!\{1,2\},h\!=\!\{2,3\} \}$.]{
      \begin{tikzpicture}[node distance=\distnode,>=stealth',bend angle=45,auto]
         \begin{scope}
            \node [source]   (s)                                    {$s$};
            \node [group]    (g)  [below of=s,xshift=-12mm]                      {$g$}
            edge  [pre] node[left] {$\xib^g_1 \!+ \!\xib^g_2 \!\leq \!\lambda \eta_g$} (s);
            \node [group]    (h)  [below of=s,xshift=12mm]                      {$h$}
            edge  [pre] node[right] {$\xib^h_2 \!+ \!\xib^h_3 \!\leq \!\lambda \eta_h$} (s);
            \node [var] (u2) [below of=g,xshift=12mm]                    {$\u_2$}
            edge  [pre] node[above, right] {$\xib^{h}_2$} (h)
            edge  [pre] node[above, left] {$\xib^{g}_2$} (g);
            \node [var] (u1)  [left of=u2, node distance=\distnodex] {$\u_1$}
            edge  [pre] node[above, left] {$\xib^{g}_1$} (g);
            \node [var] (u3) [right of=u2, node distance=\distnodex] {$\u_3$}
            edge  [pre] node[above, right] {$\xib^{h}_3$} (h);
            \node [sink] (si) [below of=u2] {$t$}
            edge [pre] node[above,left] {$\xibbar_1, \color{red} c_1$} (u1)
            edge [pre] node[above,left,xshift=1mm] {$\xibbar_2,\color{red} c_2$} (u2)
            edge [pre] node[above,right] {$\xibbar_3,\color{red} c_3$} (u3);
         \end{scope}
      \end{tikzpicture}\label{subfig:graphb}
      } ~~~~~~~~~~~\\
      \subfloat[$\GG\!=\!\{ g\!=\!\{1,2,3\},h\!=\!\{2,3\} \}$.]{
      \begin{tikzpicture}[node distance=\distnode,>=stealth',bend angle=45,auto]
         \begin{scope}
            \node [source]   (s)                                    {$s$};
            \node [group]    (g)  [below of=s]                      {$g$}
            edge  [pre] node[left] {$\xib^g_1 \!+ \!\xib^g_2 \!+ \!\xib^g_3 \!\leq \!\lambda \eta_g$} (s);
            \node [group]    (h)  [right of=g,node distance=\distnodex]                      {$h$}
            edge  [pre] node[right,yshift=1mm] {$\xib^h_2 \!+ \!\xib^h_3 \!\leq \!\lambda \eta_h$} (s);
            \node [var] (u2) [below of=g]                    {$\u_2$}
            edge  [pre] node[above, left] {$\xib^{h}_2$} (h)
            edge  [pre] node[above, left] {$\xib^{g}_2$} (g);
            \node [var] (u1)  [left of=u2, node distance=\distnodex] {$\u_1$}
            edge  [pre] node[above, left] {$\xib^{g}_1$} (g);
            \node [var] (u3) [right of=u2, node distance=\distnodex] {$\u_3$}
            edge  [pre] node[above, right] {$\xib^{g}_3$} (g)
            edge  [pre] node[above, right] {$\xib^{h}_3$} (h);
            \node [sink] (si) [below of=u2] {$t$}
            edge [pre] node[above,left] {$\xibbar_1, \color{red} c_1$} (u1)
            edge [pre] node[above,left,xshift=1mm] {$\xibbar_2,\color{red} c_2$} (u2)
            edge [pre] node[above,right] {$\xibbar_3,\color{red} c_3$} (u3);
         \end{scope}
      \end{tikzpicture}\label{subfig:graphc}
      } \hfill
      \subfloat[$\GG\!=\!\{ g\!=\!\{1\}\cup h,h\!=\!\{2,3\} \}$.]{
      \begin{tikzpicture}[node distance=\distnode,>=stealth',bend angle=45,auto]
         \begin{scope}
            \node [source]   (s)                                    {$s$};
            \node [group]    (g)  [below of=s]                      {$g$}
            edge  [pre] node[left] {$\xib^g_1 \!+ \!\xib^g_2 \!+ \!\xib^g_3 \!\leq \!\lambda \eta_g$} (s);
            \node [group]    (h)  [right of=g, node distance=\distnodex]                      {$h$}
            edge  [pre] node[right,yshift=1mm] {$\xib^h_2 \!+ \!\xib^h_3 \!\leq \!\lambda \eta_h$} (s)
            edge  [pre] node[below,yshift=1mm] {$\xib^g_2 \!+ \!\xib^g_3$} (g);
            \node [var] (u2) [below of=g]                    {$\u_2$}
            edge  [pre] node[above,left] {$\xib^g_2\!+\!\xib^{h}_2$} (h);
            \node [var] (u1)  [left of=u2, node distance=\distnodex] {$\u_1$}
            edge  [pre] node[above, left] {$\xib^{g}_1$} (g);
            \node [var] (u3) [right of=u2, node distance=\distnodex] {$\u_3$}
            edge  [pre] node[above, right] {$\xib^g_3\!+\!\xib^{h}_3$} (h);
            \node [sink] (si) [below of=u2] {$t$}
            edge [pre] node[above,left] {$\xibbar_1, \color{red} c_1$} (u1)
            edge [pre] node[above,left,xshift=1mm] {$\xibbar_2,\color{red} c_2$} (u2)
            edge [pre] node[above,right] {$\xibbar_3,\color{red} c_3$} (u3);
         \end{scope}
      \end{tikzpicture} \label{subfig:graphd}
      } 
   \end{center}
   \caption{Graph representation of simple proximal problems with different
   group structures $\GG$. The three indices $1,2,3$ 
   are represented as grey squares, and the groups
   $g,h$ in $\GG$ as red discs. The source is
   linked to every group $g,h$ with respective maximum capacity
   $\lambda\eta_g,\lambda\eta_h$ and zero cost. Each variable $\u_j$ is linked to the sink
   $t$, with an infinite capacity, and with a cost
   $c_j\!\defin\!\frac{1}{2}(\u_j-\xibbar_j)^2$. All other arcs in the graph have
   zero cost and infinite capacity. They represent inclusion relations in-between groups, and between groups and variables.
   The graphs
   \subref{subfig:graphc} and \subref{subfig:graphd} correspond to a special case of
   tree-structured hierarchy in the sense of \citet{jenatton3}. Their min-cost
   flow problems are equivalent.} \label{fig:graphs}
\end{figure}
\subsection{Computation of the Proximal Operator}\label{subsec:prox}
Quadratic min-cost flow problems have been well studied in the operations
research literature \citep{hochbaum}. One of the simplest cases, where~$\GG$
contains a single group as in Figure~\ref{subfig:grapha}, is solved by
an orthogonal projection on the $\ell_1$-ball of radius~$\lambda\eta_g$.  It
has been shown, both in machine learning~\citep{duchi} and operations
research~\citep{hochbaum,brucker}, that such a projection can be computed in $O(p)$
operations.  When the group structure is a tree as in
Figure~\ref{subfig:graphd}, 
strategies developed in the two communities are
also similar~\citep{jenatton3,hochbaum},\footnote{Note however that, 
while~\citet{hochbaum} only consider a tree-structured sum of $\ell_\infty$-norms, 
the results from~\citet{jenatton3} also apply for a sum of $\ell_2$-norms.} and solve the problem in $O(p d)$
operations, where $d$ is the depth of the tree.

The general case of overlapping groups is more difficult. \citet{hochbaum} have shown that \emph{quadratic min-cost flow problems} can be reduced to a specific
\emph{parametric max-flow} problem, for which an efficient algorithm
exists~\citep{gallo}.\footnote{By definition, a parametric max-flow problem
consists in solving, for every value of a parameter, a max-flow problem on a
graph whose arc capacities depend on this parameter.}
While this generic approach could be used to solve Eq.~(\ref{eq:dual_problem}),
we propose to use Algorithm~\ref{algo:prox} that also exploits
the fact that our graphs have non-zero costs only on edges leading to the sink.
As shown in Appendix~\ref{sec:speed_comp}, it it has a significantly better
performance in practice.  This algorithm clearly shares some similarities with
existing approaches in network flow optimization
such as the simplified version of~\citet{gallo} presented by~\citet{babenko} 
that uses a divide and conquer strategy. 
Moreover, an equivalent algorithm exists for minimizing convex functions over
polymatroid sets~\citep{groenevelt}. This equivalence, a priori non trivial, is uncovered through a representation of structured sparsity-inducing norms via
submodular functions, which was recently proposed by~\citet{bach6}. 

\begin{algorithm}[hbtp]
\caption{\hspace*{0.1cm}Computation of the proximal operator for overlapping groups.}\label{algo:prox}
\begin{algorithmic}[1]
\INPUT $\u \in \R{p}$, a set of groups $\GG$, positive weights
$(\eta_g)_{g\in\GG}$, and $\lambda$ (regularization parameter).
\STATE Build the initial graph $G_0=(V_0,E_0,s,t)$ as explained in Section \ref{subsec:prox}.
\STATE Compute the optimal flow: $\xibbar \leftarrow \text{\texttt{computeFlow}}(V_0,E_0)$.
\STATE {\bf{Return:}} $\w = \u-\xibbar$ (optimal solution of the proximal problem).
\end{algorithmic}
\vspace*{0.1cm}
{\bf Function} \texttt{computeFlow}($V=V_u \cup V_{gr},E$)
\begin{algorithmic}[1]
\STATE Projection step: $\gammab \leftarrow
\argmin_\gammab \sum_{j \in V_u}\frac{1}{2}(\u_j -\gammab_j)^2 \st  \sum_{j \in V_u} \gammab_j \leq \lambda\sum_{g \in V_{gr}}\eta_g.$
\STATE For all nodes $j$ in $V_u$, set $\gammab_j$ to be the capacity of the arc $(j,t)$.
\STATE Max-flow step: Update $(\xibbar_j)_{j \in V_u}$ by computing a max-flow on the graph $(V,E,s,t)$.
\IF{ $\exists~j \in V_u \st \xibbar_j \neq \gammab_j$}
\STATE Denote by $(s,V^+)$ and $(V^-,t)$ the two disjoint subsets of $(V,s,t)$ separated by
the minimum $(s,t)$-cut of the graph, and remove the arcs between $V^+$ and $V^-$. 
Call $E^+$ and $E^-$ the two remaining disjoint subsets of~$E$ corresponding to $V^+$ and $V^-$.
\STATE $(\xibbar_j)_{j \in V_u^+} \leftarrow \text{\texttt{computeFlow}}(V^+,E^+)$.
\STATE $(\xibbar_j)_{j \in V_u^-} \leftarrow \text{\texttt{computeFlow}}(V^-,E^-)$.
\ENDIF
\STATE {\bf{Return:}} $(\xibbar_j)_{j \in V_u}$.
\end{algorithmic}
\end{algorithm}

The intuition behind our algorithm, \texttt{computeFlow} (see Algorithm~\ref{algo:prox}), is the following: since $\xibbar=\sum_{g
\in \G} \xib^g$ is the only value of interest to compute the solution of the
proximal operator $\w = \u-\xibbar$, the first step looks for a candidate value $\gammab$ for 
$\xibbar$ by solving the following relaxed version of problem~(\ref{eq:dual_problem2}):
\begin{equation}
\argmin_{\gammab \in \Real^p} \sum_{j \in V_u}\frac{1}{2}(\u_j -\gammab_j)^2 \st  \sum_{j \in V_u} \gammab_j \leq \lambda\sum_{g \in V_{gr}}\eta_g. \label{eq:relaxed}
\end{equation}
The cost function here is the same as in problem~(\ref{eq:dual_problem2}), but
the constraints are weaker: Any feasible point of
problem~(\ref{eq:dual_problem2}) is also feasible for
problem~(\ref{eq:relaxed}). This problem can be solved in linear
time~\cite{brucker}.
Its solution, which we denote $\gammab$ for simplicity, provides the lower bound
$\|\u-\gammab\|_2^2/2$ for the optimal cost of problem~(\ref{eq:dual_problem2}).

The second step tries to construct a feasible flow $(\xib,\xibbar)$, satisfying additional capacity constraints equal to $\gamma_j$ on arc $(j,t)$, and
whose cost matches this lower bound; this latter problem can be cast as a max-flow problem~\citep{goldberg}. 
If such a flow exists, the algorithm returns $\xibbar=\gammab$, the cost of the flow reaches the lower bound, and is therefore optimal. 
If such a flow does not exist, we have $\xibbar \neq \gammab$, the lower bound is not achievable, 
and we build a minimum $(s,t)$-cut of the graph~\cite{ford} defining
two disjoints sets of nodes $V^+$ and~$V^-$; $V^+$ is the part of the graph 
which is reachable from the source (for every node $j$ in~$\V^+$, there exists a non-saturated path from $s$ to~$j$),
whereas all paths going from~$s$ to nodes in $V^-$ are saturated. More details
about these properties can be found at the beginning of Appendix~\ref{appendix:convergence}.
At this point, it is possible to show that the value of the optimal min-cost flow 
on all arcs between~$V^+$ and~$V^-$ is necessary zero.
Thus, removing them yields an equivalent optimization problem, 
which can be decomposed into two independent problems of smaller sizes and
solved recursively by the calls to \texttt{computeFlow}$(V^+,E^+)$ and
\texttt{computeFlow}$(V^-,E^-)$.  A formal proof of correctness of
Algorithm~\ref{algo:prox} and further details are relegated to
Appendix~\ref{appendix:convergence}.

The approach of~\citet{hochbaum,gallo} which recasts the quadratic min-cost
flow problem as a parametric max-flow is guaranteed to have the same worst-case
complexity as a single max-flow algorithm. However, we have experimentally
observed a significant discrepancy
between the worst case and empirical complexities for these flow problems,
essentially because the empirical cost of each max-flow is significantly smaller
than its theoretical cost.  Despite the fact that the worst-case guarantees for our algorithm 
is weaker than theirs (up to a factor $|V|$), 
it is more adapted to the
structure of our graphs and has proven to be much faster in our experiments
(see Appendix~\ref{sec:speed_comp}).\footnote{The best theoretical worst-case complexity of a max-flow is achieved by~\citet{goldberg} and is $O\big(|V||E|\log(|V|^2 / |E|)\big)$.
Our algorithm achieves the same worst-case complexity when the cuts are well balanced---that is~$|V^+| \approx |V^-| \approx |V|/2$, but we lose a factor~$|V|$ when it is not the case. The practical speed of such algorithms is however significantly different than their theoretical worst-case complexities~\citep[see][]{boykov}.}
Some implementation details are also crucial to the efficiency of the algorithm: 
\begin{itemize}
   \item \textbf{Exploiting maximal connected components}: When there exists no
arc between two subsets of $V$, the solution can be obtained by solving
two smaller optimization problems corresponding to the two disjoint subgraphs.
It is indeed possible to process them independently to
solve the global min-cost flow problem.  To that effect, before calling the
function \texttt{computeFlow}($V,E$), we look for maximal connected components
$(V_1,E_1),\ldots,(V_N,E_N)$ and call sequentially the procedure
\texttt{computeFlow}($V_i,E_i$) for $i$ in $\IntSet{N}$. 
\item  \textbf{Efficient max-flow algorithm}: We have implemented the
``push-relabel'' algorithm of~\citet{goldberg} to solve our max-flow
problems, using classical heuristics that significantly speed it up in practice;
see \citet{goldberg} and~\citet{cherkassky}.
 We use the so-called ``highest-active vertex
selection rule, global and gap heuristics''~\citep{goldberg,cherkassky},
which has a worst-case complexity of $O(|V|^2 |E|^{1/2})$ for a graph
$(V,E,s,t)$.
This algorithm leverages the concept of \emph{pre-flow} that relaxes the
definition of flow and allows vertices to have a positive excess.
\item \textbf{Using flow warm-restarts}: 
   The max-flow steps in our algorithm can be initialized with any valid pre-flow, enabling warm-restarts.
  This is also a key concept in the parametric max-flow algorithm of~\citet{gallo}.
 \item \textbf{Improved projection step}:
The first line of the procedure \texttt{computeFlow} can be replaced by 
$\gammab \leftarrow \argmin_\gammab \sum_{j \in V_u}\frac{1}{2}(\u_j
-\gammab_j)^2 \st  \sum_{j \in V_u} \gammab_j \leq \lambda\sum_{g \in V_{gr}}\eta_g ~\text{and}~
|\gammab_j| \leq \lambda\sum_{g \ni j}\eta_g.$
The idea is to build a relaxation of Eq.~(\ref{eq:dual_problem2}) which is
closer to the original problem than the one of Eq.~(\ref{eq:relaxed}),
but that still can be solved in linear time.
The structure of the graph will indeed not allow~$\xibbar_j$ to be greater
than $\lambda\sum_{g \ni j}\eta_g$ after the max-flow step. 
This modified
projection step can still be computed in linear time~\citep{brucker},
and leads to better performance.
\end{itemize}

\subsection{Computation of the Dual Norm}\label{subsec:dual}
The dual norm $\Omega^*$ of $\Omega$, defined for any vector $\kappab$ in $\R{p}$ by 
$$
\Omega^*(\kappab)\defin\max_{\Omega(\z)\leq 1}\z^\top\kappab,
$$
is a key quantity to study sparsity-inducing regularizations in many respects.
For instance, dual norms are central in working-set algorithms~\citep{jenatton, bach5}, 
and arise as well when proving theoretical estimation or prediction guarantees~\citep{negahban2009unified}.

In our context, we use it to monitor the convergence of the proximal method through a duality gap,
hence defining a proper optimality criterion for problem~(\ref{eq:formulation}).
As a brief reminder, the duality gap of a minimization problem is defined as the difference between
the primal and dual objective functions, 
evaluated for a feasible pair of primal/dual variables~\citep[see Section 5.5,][]{boyd}. 
This gap serves as a certificate of (sub)optimality: if it
is equal to zero, then the optimum is reached, and provided that strong duality holds, the converse
is true as well~\citep[see Section 5.5,][]{boyd}.
A description of the algorithm we use in the experiments~\citep{beck} 
along with the integration of the computation of the duality gap is given in Appendix~\ref{appendix:fista}.

We now denote by $f^*$ the Fenchel conjugate of $f$~\citep{borwein}, defined by $f^*(\kappab)\defin\sup_{\z} [\z^\top\kappab - f(\z)]$.
The duality gap for problem~(\ref{eq:formulation}) can be derived from standard Fenchel duality arguments \citep{borwein} and it is equal to
$$
 f(\w)+ \lambda \Omega(\w) + f^*(-\kappab)\ \text{for}\ \w,\kappab\ \text{in}\ \R{p}\ \text{with}\
\Omega^*(\kappab)\leq \lambda.
$$
Therefore, evaluating the duality gap requires to compute efficiently $\Omega^*$
in order to find a feasible dual variable~$\kappab$ 
(the gap is otherwise equal to $+\infty$ and becomes non-informative).
This is equivalent to solving another network flow problem, based on the following variational formulation:
\begin{equation}
\Omega^*(\kappab) = \!\!\! \min_{\xib\in\RR{p}{|\G|}} \!\!\! \tau \quad \text{s.t.}\quad \sum_{g\in\G}\xib^g=\kappab,\ \text{and}\ \forall g\in\G,\ \|\xib^g\|_1 \leq \tau\eta_g ~~\text{with}~~\
\xib_j^g=0\ \text{if}\ j\notin g. \label{eq:dual_norm}
\end{equation}
In the network problem associated with~(\ref{eq:dual_norm}), the capacities on the arcs $(s,g)$,
$g \in \GG$, are set to $\tau\eta_g$, and the capacities on the arcs $(j,t)$, $j$ in
$\IntSet{p}$, are fixed to $\kappab_j$.
Solving problem~(\ref{eq:dual_norm}) amounts to finding the smallest value of $\tau$, such
that there exists a flow saturating all the capacities $\kappab_j$ on the arcs
leading to the sink~$t$.
Equation~(\ref{eq:dual_norm}) and Algorithm~\ref{alg:dualnorm}
are proven to be correct in Appendix~\ref{appendix:convergence}.

\begin{algorithm}[hbtp]
\caption{\hspace*{0.1cm}Computation of the dual norm.}\label{algo:dual_norm}
\begin{algorithmic}[1]
\INPUT $\kappab \in \R{p}$, a set of groups $\GG$, positive weights
$(\eta_g)_{g\in\GG}$.
\STATE Build the initial graph $G_0=(V_0,E_0,s,t)$ as explained in Section~\ref{subsec:dual}.
\STATE $\tau \leftarrow \text{\texttt{dualNorm}}(V_0,E_0)$.
\STATE {\bf{Return:}} $\tau$ (value of the dual norm).
\end{algorithmic}
\vspace*{0.1cm}
{\bf Function} \texttt{dualNorm}($V = V_u \cup V_{gr},E$)
\begin{algorithmic}[1]
\STATE $\tau\! \leftarrow \!(\sum_{j \in V_u} \kappab_j) / (\sum_{g \in V_{gr}} \eta_g)$ and set the capacities of arcs $(s,g)$ to $\tau\eta_g$ for all $g$ in~$V_{gr}$.
\STATE Max-flow step: Update $(\xibbar_j)_{j \in V_u}$ by computing a max-flow on the graph $(V,E,s,t)$.
\IF{ $\exists~j \in V_u \st \xibbar_j \neq \kappab_j$}
\STATE Define $(V^+,E^+)$ and $(V^-,E^-)$ as in Algorithm~\ref{algo:prox}, and set $\tau \leftarrow \text{\texttt{dualNorm}}(V^-,E^-)$.
\ENDIF
\STATE {\bf {Return:} } $\tau$.
\end{algorithmic}\label{alg:dualnorm}
\end{algorithm}

\section{Optimization with Proximal Splitting Methods}\label{sec:optim_prox2}
We now present proximal splitting algorithms~\citep[see][and references
therein]{combettes2,combette,tomioka,boyd2} for solving Eq.~(\ref{eq:formulation}). Differentiability of $f$ is not required here and the regularization function can either
be a sum of $\ell_2$- or $\ell_\infty$-norms. However, we assume that:
\begin{itemize}
\item[\textbf{(A)}] either $f$ can be written $f(\w) = \sum_{i=1}^n \tilde{f}_i(\w)$, where the functions $\tilde{f}_i$ are such that $\text{prox}_{\gamma\tilde{f}_i}$ can be obtained in closed form for all $\gamma > 0$ and all $i$---that is, for all $\u$ in $\Real^m$, the following problems admit closed form solutions:
$
   \min_{\v \in \Real^m}  \frac{1}{2}\|\u-\v\|_2^2 + \gamma\tilde{f}_i(\v). 
$

    \item[\textbf{(B)}] or $f$ can be written $f(\w) = \tilde{f}(\X\w)$ for all $\w$ in $\Real^p$, where $\X$ in $\Real^{n \times p}$ is a design matrix, and one knows how to efficiently compute $\text{prox}_{\gamma\tilde{f}}$ for all $\gamma > 0$.
\end{itemize}
It is easy to show that this condition is satisfied for the square
and hinge loss functions, making it possible to build linear SVMs with
a structured sparse regularization.
These assumptions are not the same as
the ones of Section~\ref{sec:optim_prox}, and the scope of the problems addressed
is therefore slightly different.
Proximal splitting methods seem indeed to offer more flexibility regarding
the regularization function, since they can deal with sums of
$\ell_2$-norms.\footnote{We are not aware of any efficient algorithm
providing the exact solution of the proximal operator associated to a sum of
$\ell_2$-norms, which would be necessary for using (accelerated) proximal gradient methods. 
An iterative algorithm could possibly be used to compute
it approximately~\citep[e.g., see][]{jenatton3,jenatton4}, 
but such a procedure would be computationally expensive and would require to be able to 
deal with approximate computations of the proximal operators~\citep[e.g., see][and discussions therein]{combette,schmidt}.  We have
chosen not to consider this possibility in this paper.}
However, proximal gradient methods, as presented in
Section~\ref{sec:optim_prox}, enjoy a few advantages over proximal splitting
methods, namely: automatic parameter tuning with line-search
schemes~\citep{nesterov}, known convergence rates~\citep{nesterov,beck}, and
 ability to provide
sparse solutions (approximate solutions obtained with proximal splitting methods often have small values, but not ``true'' zeros).

\subsection{Algorithms}
We consider a class of algorithms which leverage the concept of variable
splitting~\citep[see][]{combette,bertsekas3,tomioka}.
The key is to introduce additional variables $\z^g$ in $\Real^{|g|}$, one for every group $g$ in $\GG$,
and equivalently reformulate Eq.~(\ref{eq:formulation}) as
\begin{equation}
   \min_{\substack{ \w \in \Real^p\\ \z^g \in \Real^{|g|}\ \text{for}\ g \in \GG} } f(\w) + \lambda
\sum_{g \in \GG} \eta_g \|\z^g\|   \st  \forall g \in \GG,~~ \z^g = \w_g, \label{eq:varsplit}
\end{equation}
The issue of overlapping groups is removed, but new constraints are added, and as in Section~\ref{sec:optim_prox},
the method introduces additional variables which induce a memory cost of $O(\sum_{g\in\GG}|g|)$.

To solve this problem, it is possible to use the so-called alternating direction method of
multipliers (ADMM)~\citep[see][]{combette,bertsekas3,tomioka,boyd2}.\footnote{This
method is used by~\citet{sprechmann} for computing the
proximal operator associated to hierarchical norms, and independently in the same context as ours
by~\citet{boyd2} and~\citet{qin}.}
It introduces dual variables $\nu^g$ in $\Real^{|g|}$ for all $g$ in $\GG$, and defines the
augmented Lagrangian:
\begin{displaymath}
   \L\big(\w,(\z^g)_{g\in\G},(\nu^g)_{g\in\G}\big) \defin f(\w) + \sum_{g \in \GG} 
   \big[\lambda\eta_g \|\z^g\|  + \nu^{g\top}(\z^g-\w_g) + \frac{\gamma}{2} \|\z^g - \w_g \|_2^2\big],
\end{displaymath}
where $\gamma > 0$ is a parameter.
It is easy to show that solving Eq.~(\ref{eq:varsplit}) amounts to finding a
saddle-point of the augmented Lagrangian.\footnote{The augmented Lagrangian is
in fact the classical Lagrangian~\citep[see][]{boyd} of the following
optimization problem which is equivalent to Eq.~(\ref{eq:varsplit}):
\begin{displaymath}
   \min_{\w \in \Real^p, (\z^g \in \Real^{|g|})_{g \in \GG} } f(\w) + \lambda
\sum_{g \in \GG} \eta_g \|\z^g\| + \frac{\gamma}{2}\|\z^g-\w_g\|_2^2  \st  \forall g \in \GG,~~ \z^g = \w_g.
\end{displaymath}
}
The ADMM algorithm finds such a saddle-point by iterating between the minimization of $\L$ with respect to each primal variable, keeping the other ones fixed,
and gradient ascent steps with respect to the dual variables. More precisely, it can be summarized~as: 
\begin{enumerate}
   \item Minimize $\L$ with respect to $\w$, keeping the other variables fixed. \label{admm:step1}
   \item Minimize $\L$ with respect to the $\z^g$'s, keeping the other variables fixed. The solution can be obtained in closed form: for all $g$ in $\G$, $\z^g \leftarrow \text{prox}_{\frac{\lambda\eta_g}{\gamma}\|.\|}[\w_g-\frac{1}{\gamma}\nu^g]$.
   \item Take a gradient ascent step on $\L$ with respect to the $\nu^g$'s: $\nu^g \leftarrow \nu^g + \gamma( \z^g - \w_g )$.
  \item Go back to step~\ref{admm:step1}.
\end{enumerate}
Such a procedure is guaranteed to converge to the desired solution for all
value of $\gamma  > 0$ (however, tuning $\gamma$ can greatly influence the
convergence speed), but solving efficiently step~\ref{admm:step1} can be difficult.
To cope with this issue, we propose two variations exploiting assumptions \textbf{(A)} and~\textbf{(B)}.
\subsubsection{Splitting the Loss Function $f$} \label{subsec:loss_split}
We assume condition \textbf{(A)}---that is, we have
$f(\w)=\sum_{i=1}^n \tilde{f}_i(\w)$. 
For example, when $f$ is the square loss function $f(\w)=\frac{1}{2}\|\y-\X\w\|_2^2$, 
where $\X$ in $\Real^{n \times p}$ is a design matrix and $\y$ is in $\Real^n$,
we would define for all $i$ in $\{1,\ldots,n\}$ the functions $\tilde{f}_i: \Real \to \Real$
such that $\tilde{f}_i(\w)\defin \frac{1}{2}(\y_i-\x_i^\top\w)^2$, where $\x_i$ is the $i$-th row of $\X$.

We now introduce new variables $\v^i$ in
$\Real^{p}$ for $i=1,\ldots,n$, and replace $f(\w)$ in Eq.~(\ref{eq:varsplit})
by  $\sum_{i=1}^n \tilde{f}_i(\v^i)$, with the additional constraints that
$\v^i=\w$. The resulting equivalent optimization problem can now be tackled using the
ADMM algorithm, following the same methodology presented above. It is easy to
show that every step can be obtained efficiently, as long as one knows how to 
compute the proximal operator associated to the functions ${\tilde{f}_i}$ in closed form.
This is in fact the case for the square and hinge loss functions, where~$n$ is
the number of training points. The main problem of this strategy is the
possible high memory usage it requires when $n$ is large.
\subsubsection{Dealing with the Design Matrix}\label{subsec:lin_split}
If we assume condition \textbf{(B)}, another possibility consists of introducing a
new variable~$\v$ in $\Real^n$, such that one can replace the function
$f(\w)=\tilde{f}(\X\w)$ by $\tilde{f}(\v)$ in Eq.~(\ref{eq:varsplit}) with the
additional constraint $\v=\X\w$.  Using directly the ADMM algorithm to solve
the corresponding problem implies adding a term
$\kappab^\top(\v-\X\w)+\frac{\gamma}{2}\|\v-\X\w\|_2^2$ to the augmented
Lagrangian~$\L$, where $\kappab$ is a new dual variable.
The minimization of~$\L$ with respect to $\v$ is now obtained by $\v \leftarrow
\text{prox}_{\frac{1}{\gamma}\tilde{f}}[\X\w-\kappab]$, which is 
easy to compute according to \textbf{(B)}.
However, the design matrix~$\X$ in the quadratic term makes the minimization of
$\L$ with respect to $\w$ more difficult.  
To overcome this issue, we adopt a
strategy presented by~\citet{zhang2}, which replaces at iteration $k$ the quadratic term
$\frac{\gamma}{2}\|\v-\X\w\|_2^2$ in the augmented Lagrangian by an additional proximity term: $\frac{\gamma}{2}\|\v-\X\w\|_2^2+\frac{\gamma}{2}\|\w-\w^k\|_\Q^2$,
where $\w^k$ is the current estimate of $\w$, and $\|\w-\w^k\|_\Q^2=(\w-\w^k)^\top \Q(\w-\w^k)$,
where $\Q$ is a symmetric positive definite matrix. By choosing $\Q \defin \delta \I - \X^\top\X$, with $\delta$ large enough,
minimizing $\L$ with respect to~$\w$ becomes simple, while convergence to the solution is still ensured.
More details can be found in~\citet{zhang2}.

\section{Applications and Experiments} \label{sec:exp}
In this section, we present various experiments demonstrating the applicability
and the benefits of our methods for solving large-scale sparse and structured
regularized problems.

\subsection{Speed Benchmark}
We consider a structured sparse decomposition problem with overlapping groups
of $\ell_\infty$-norms, and compare the proximal gradient algorithm
FISTA~\citep{beck} with our proximal operator presented in
Section~\ref{sec:optim_prox} (referred to as ProxFlow), two variants of proximal splitting methods,
(ADMM) and (Lin-ADMM) respectively presented in Section~\ref{subsec:loss_split}
and \ref{subsec:lin_split}, and two generic optimization techniques,
namely a subgradient descent (SG) and an interior point method,\footnote{In our
simulations, we use the commercial software \texttt{Mosek},
\texttt{http://www.mosek.com/}} on a regularized linear regression problem.
SG, ProxFlow, ADMM and Lin-ADMM are implemented in \texttt{C++}.\footnote{Our implementation of ProxFlow is available at \url{http://www.di.ens.fr/willow/SPAMS/.} }  Experiments are run on a
single-core $2.8$ GHz CPU.  We consider a design matrix $\X$ in $\Real^{n
\times p}$ built from overcomplete dictionaries
of discrete cosine transforms (DCT), which are naturally organized on one-
or two-dimensional grids and display local correlations.  The following
families of groups $\G$ using this spatial information are thus
considered: (1) every contiguous sequence of length $3$ for the
one-dimensional case, and (2) every $3\!\times\!3$-square in the
two-dimensional setting.  We generate vectors~$\y$ in $\R{n}$ according to
the linear model $\y = \X\w_0 + \varepsilonb$, 
where $\varepsilonb \sim \N(0,0.01\|\X\w_0\|_2^2)$.  The vector $\w_0$ has
about $20\%$ percent nonzero components, randomly selected, while
respecting the structure of~$\G$, and uniformly generated in
$[-1,1]$.

In our experiments, the regularization parameter $\lambda$ is chosen to achieve
the same level of sparsity ($20\%$).  For SG, ADMM and Lin-ADMM, some parameters are
optimized to provide the lowest value of the objective function after $1\,000$
iterations of the respective algorithms.  For SG, we take the step size to be equal to $a/(k+b)$, where $k$
is the iteration number, and $(a,b)$ are the pair of parameters selected in
$\{10^{-3},\dots,10\}\!\times\!\{10^2,10^3,10^4\}$. Note that a step size of the form~$a/(\sqrt{t}+b)$ is 
also commonly used in subgradient descent algorithms. In the context of
hierarchical norms, both choices have led to similar results~\citep{jenatton4}.
The parameter~$\gamma$ for ADMM is selected in $\{10^{-2},\ldots,10^{2}\}$.  The parameters
$(\gamma,\delta)$ for Lin-ADMM are selected in $\{10^{-2},\ldots,10^{2}\}
\times \{10^{-1},\ldots,10^8\}$.  For interior point methods, since
problem~(\ref{eq:formulation}) can be cast either as a quadratic (QP) or as a
conic program (CP), we show in Figure~\ref{fig:speed_cmp} the results for both
formulations.
On three problems of different sizes, with $(n,p)\in\{(100,10^3),(1024,10^4),(1024,10^5)\}$, our algorithms ProxFlow, ADMM and Lin-ADMM compare favorably
with the other methods, 
 (see
Figure~\ref{fig:speed_cmp}), except for ADMM in the large-scale setting which
yields an objective function value  similar to that of SG after~$10^4$ seconds.  Among
ProxFlow, ADMM and Lin-ADMM, ProxFlow is consistently better than Lin-ADMM,
which is itself better than ADMM. Note that for the small scale problem,
the performance of ProxFlow and Lin-ADMM is similar.
In addition, note that QP, CP, SG, ADMM and
Lin-ADMM do not obtain sparse solutions, whereas ProxFlow does.\footnote{To reduce the computational cost of this experiment, the curves reported are the results of one single run. 
Similar types of experiments with several runs have shown very small variability~\citep{bach5}.}
\begin{figure}[hbtp]
    \centering
   \includegraphics[width=0.34\textwidth]{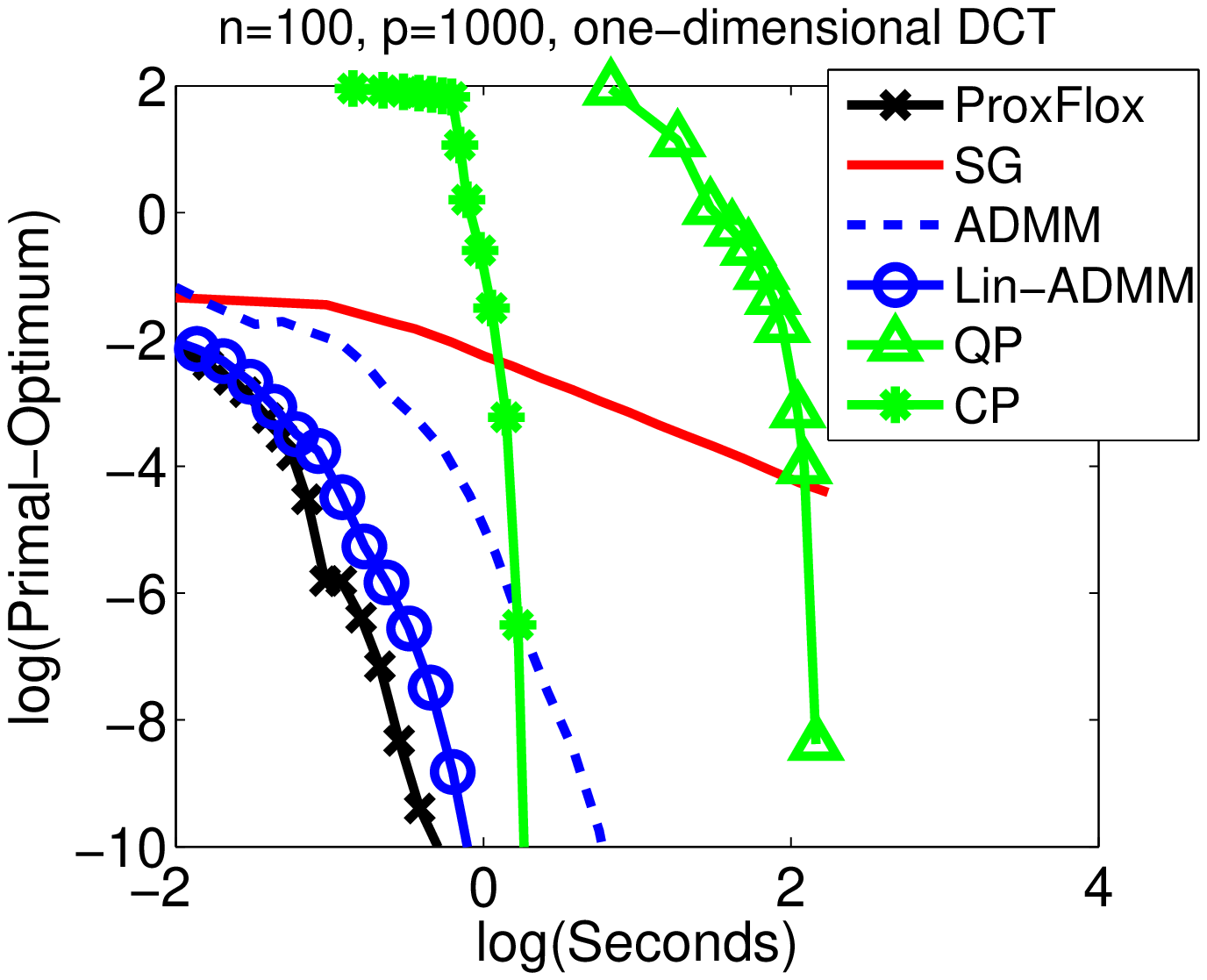} \hfill
   \includegraphics[width=0.32\textwidth]{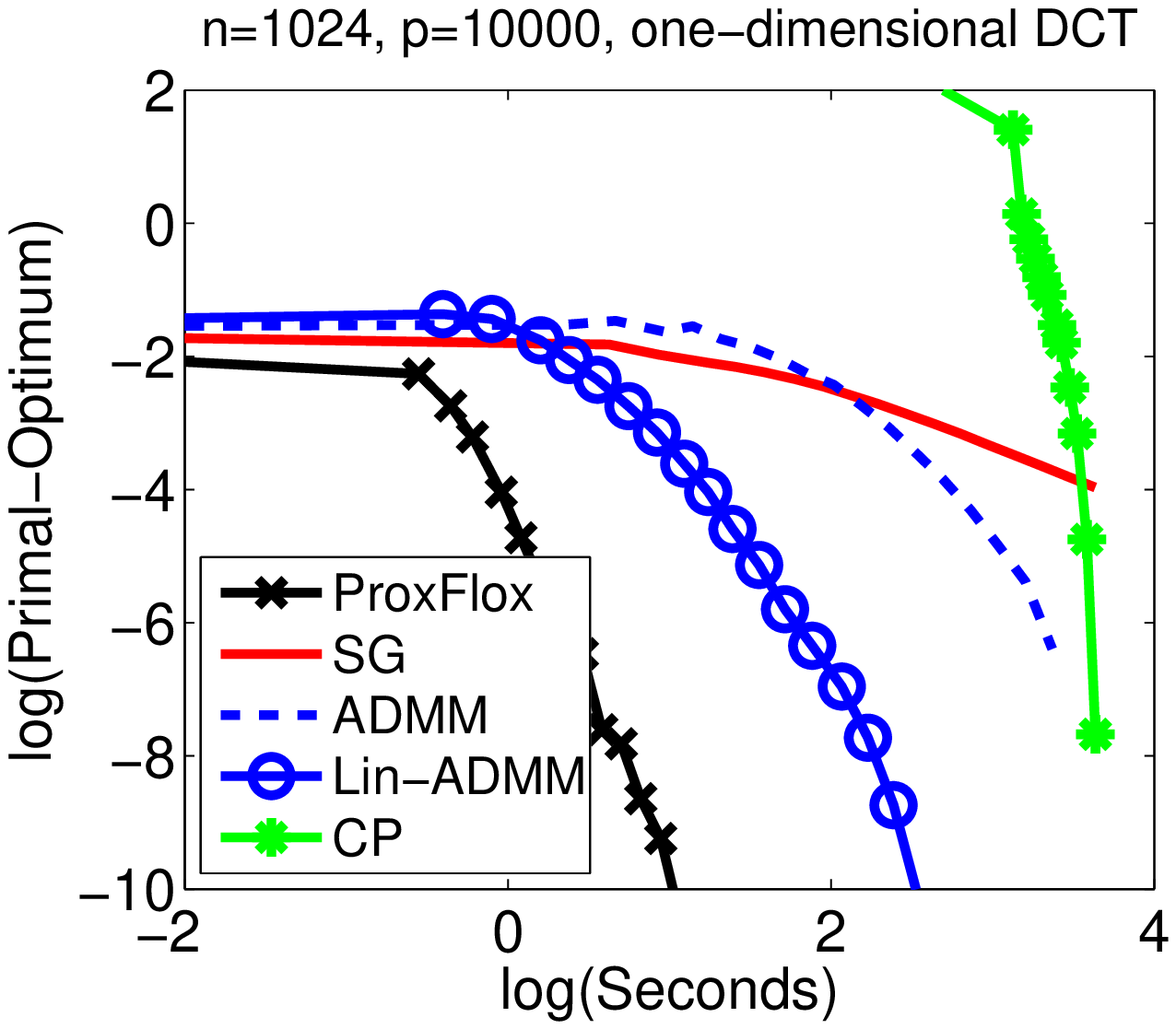} \hfill
   \includegraphics[width=0.32\textwidth]{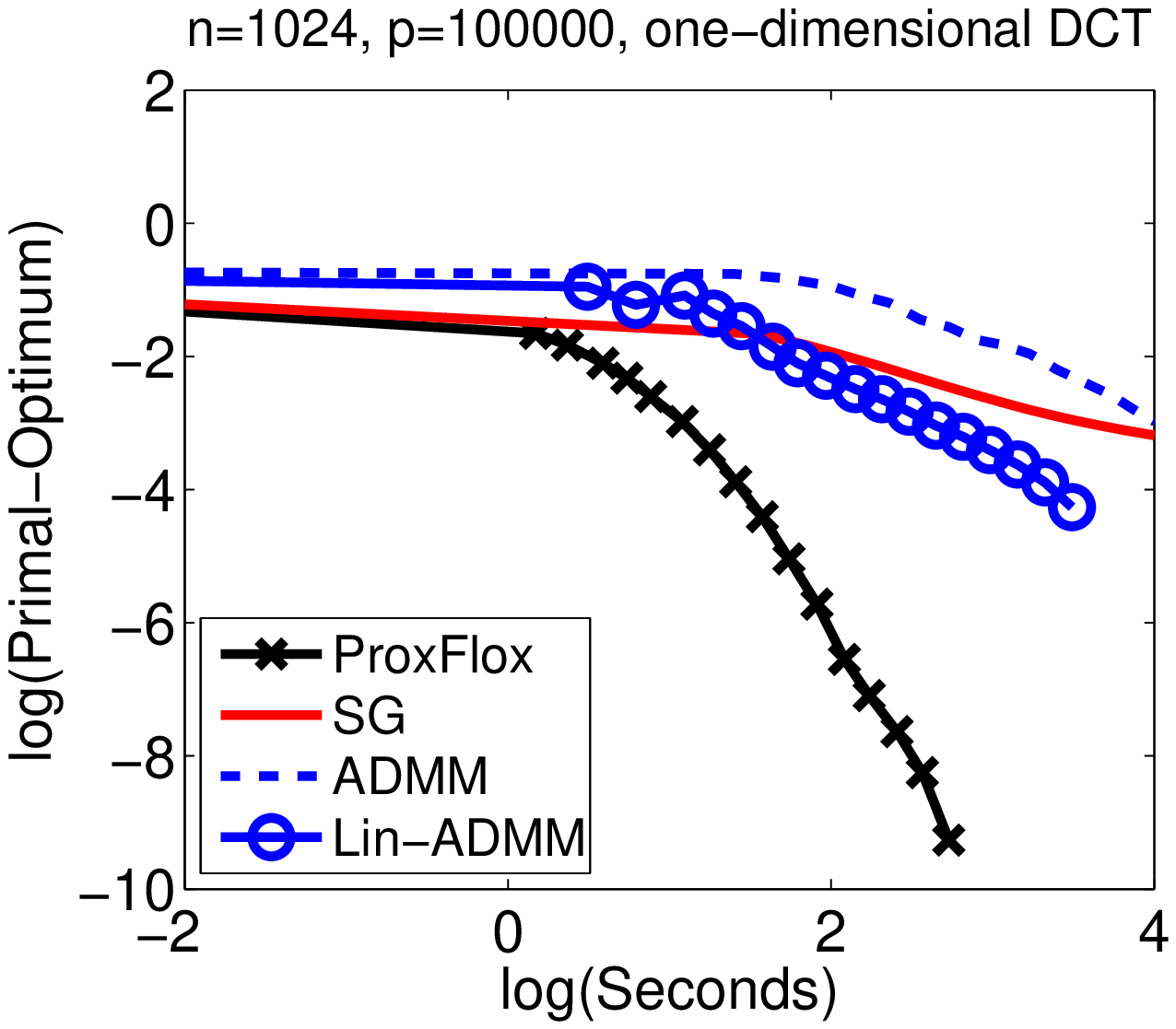} 
   \caption{Speed comparisons: distance to the optimal primal value versus CPU time (log-log scale). Due to the computational burden, QP and CP could not be run on every problem.} 
\label{fig:speed_cmp}
\end{figure}

\subsection{Wavelet Denoising with Structured Sparsity}
We now illustrate the results of Section~\ref{sec:optim_prox}, where a single
large-scale proximal operator ($p \approx 250\,000$) associated to a sum
of~$\ell_\infty$-norms has to be computed. We choose an image denoising task
with an orthonormal wavelet basis, following an experiment similar to one proposed in~\citet{jenatton4}.  Specifically, we consider the following formulation
\begin{equation}
\min_{\w\in\R{p}} \frac{1}{2} \|\y-\X\w\|_2^2+\lambda\Omega(\w),\label{eq:expwav}
\end{equation}
where $\y$ in $\Real^p$ is a noisy input image, $\w$ represents wavelets
coefficients, $\X$ in $\Real^{p \times p}$ is an orthonormal wavelet basis,
$\X\w$ is the estimate of the denoised image, and $\Omega$ is a sparsity-inducing
norm. Since here the basis is orthonormal, solving the decomposition
problem boils down to computing $\w^\star = \text{prox}_{\lambda\Omega}[\X^\top\y]$.
 This makes of Algorithm~\ref{algo:prox} a good candidate to solve it when $\Omega$ is a
sum of $\ell_\infty$-norms. We compare the following candidates for the sparsity-inducing norms
$\Omega$:
\begin{itemize}
\item the $\ell_1$-norm, leading to the wavelet soft-thresholding 
of \citet{donoho2}.
\item a sum of $\ell_\infty$-norms with a hierarchical group structure adapted
to the wavelet coefficients, as proposed in~\citet{jenatton4}. Considering a
natural quad-tree for wavelet coefficients~\citep[see][]{mallat}, this norm
takes the form of Eq.~(\ref{eq:omega}) with one group per wavelet coefficient
that contains the coefficient and all its descendants in the tree. We call this norm $\Omega_{\text{tree}}$.
\item a sum of $\ell_\infty$-norms with overlapping groups representing $2\times2$ spatial neighborhoods in the wavelet domain. This regularization encourages neighboring wavelet coefficients to be set to zero together, which was also exploited in the past in block-thresholding approaches for wavelet denoising~\citep{cai}. We call this norm~$\Omega_{\text{grid}}$.
\end{itemize}
We consider Daubechies3 wavelets~\citep[see][]{mallat} for the matrix $\X$,
use $12$ classical standard test images,\footnote{These images are used in
classical image denoising benchmarks. See~\citet{mairal8}.} and generate
noisy versions of them corrupted by a white Gaussian noise of variance
$\sigma^2$. For each image, we test several values of $\lambda =
2^{\frac{i}{4}}\sigma\sqrt{\log{p}}$, with~$i$ taken in the 
range
$\{-15,-14,\dots,15\}$.
We then keep the parameter $\lambda$
giving the best reconstruction error on average on the $12$ images.
The factor $\sigma\sqrt{\log{p}}$ is a classical heuristic for choosing a reasonable regularization parameter~\citep[see][]{mallat}.
We provide reconstruction results in terms of PSNR in
Table~\ref{table:wave}.\footnote{Denoting by \textrm{MSE} the mean-squared-error
for images whose intensities are between $0$ and $255$, the
\textrm{PSNR} is defined as $\textrm{PSNR}=10\log_{10}( 255^2 /
\textrm{MSE} )$ and is measured in dB. A gain of $1$dB reduces
the \textrm{MSE} by approximately $20\%$.}
Unlike~\citet{jenatton4}, who set all the weights $\eta_g$ in $\Omega$ equal to one, 
we tried exponential weights of the form $\eta_g= \rho^k$, with $k$ being the depth of the group in the wavelet tree, and~$\rho$ is taken in $\{0.25,0.5,1,2,4\}$. As for $\lambda$, the value providing the best reconstruction is kept.
The wavelet transforms in our experiments are computed with the
matlabPyrTools
software.\footnote{\texttt{http://www.cns.nyu.edu/$\sim$eero/steerpyr/}.}
Interestingly, we observe in Table~\ref{table:wave} that the results obtained with $\Omega_{\text{grid}}$
are significantly better than those obtained with $\Omega_{\text{tree}}$, meaning
that encouraging spatial consistency in wavelet coefficients is more effective
than using a hierarchical coding.
\begin{table}[btp]
\label{table:denoise}
\centering
\begin{tabular}{|c|c|c|c|c|c|c|}
\hline
    &     \multicolumn{3}{c|}{PSNR}  &    \multicolumn{3}{c|}{IPSNR vs. $\ell_1$}   \\
\hline
$\sigma$ & $\ell_1$ & $\Omega_{\text{tree}}$ & $\Omega_{\text{grid}}$  & $\ell_1$ & $\Omega_{\text{tree}}$ & $\Omega_{\text{grid}}$ \\
\hline
$5$ & 35.67 & 35.98 & \textbf{36.15} & $ 0.00\pm .0 $ & $0.31 \pm .18$ & $\mathbf{0.48 \pm .25}$ \\
\hline
$10$ & 31.00 & 31.60 & \textbf{31.88}  & $ 0.00\pm .0 $ & $0.61 \pm .28$  & $\mathbf{0.88 \pm .28}$ \\
\hline
$25$ & 25.68 & 26.77 & \textbf{27.07}  & $ 0.00\pm .0 $ & $1.09 \pm .32$  & $\mathbf{1.38 \pm .26}$ \\
\hline
$50$ & 22.37 & 23.84 & \textbf{24.06}  & $ 0.00\pm .0 $ & $1.47 \pm .34$  & $\mathbf{1.68 \pm .41}$ \\
\hline
$100$ & 19.64 & 21.49 & \textbf{21.56}  & $ 0.00\pm .0 $ & $1.85 \pm .28$  & $\mathbf{1.92 \pm .29}$ \\
\hline
\end{tabular}
\caption{PSNR measured for the denoising of $12$
standard images when the regularization function is the $\ell_1$-norm,
the tree-structured norm $\Omega_{\text{tree}}$, 
and the structured norm $\Omega_{\text{grid}}$, and improvement in PSNR compared to the $\ell_1$-norm (IPSNR). 
Best results for each level of noise and each wavelet type are in bold. 
The reported values are averaged over $5$ runs with different noise realizations. 
}
\label{table:wave}
\end{table}
We also note that our approach is relatively fast, despite the high dimension
of the problem.  Solving exactly the proximal problem with $\Omega_{\text{grid}}$ for
an image with $p=512 \times 512 = 262\,144$ pixels (and therefore approximately
the same number of groups) takes approximately $\approx 4-6$ seconds on a single core of
a 3.07GHz CPU. 

\subsection{CUR-like Matrix Factorization}\label{subsec:cur}
In this experiment, we show how our tools can be used to perform the so-called CUR matrix decomposition~\citep{mahoney2009cur}.
It consists of a low-rank approximation of a data matrix $\X$ in $\Real^{n\times p}$
in the form of a product of three matrices---that is, $\X \approx \Cmat\U\Rmat$.
The particularity of the CUR decomposition lies in the fact that the matrices 
$\Cmat \in \Real^{n\times c}$ and $\Rmat \in \Real^{r\times p}$ are constrained to be respectively a subset of $c$ columns and $r$ rows 
of the original matrix $\X$. The third matrix $\U \in \Real^{c \times r}$ is then given by $\Cmat^+ \X \Rmat^+$, 
where $\A^+$ denotes a Moore-Penrose generalized inverse of the matrix $\A$~\citep{horn1990matrix}.
Such a matrix factorization is particularly appealing when the interpretability of the results matters~\citep{mahoney2009cur}.
For instance, when studying gene-expression datasets, it is easier to gain insight from the selection of 
actual patients and genes, rather than from linear combinations of them.

In~\citet{mahoney2009cur}, CUR decompositions are computed by a sampling procedure based on the singular value decomposition of $\X$. In a recent work,
\citet{bien} have 
shown that \textit{partial} CUR decompositions, i.e., the selection of either rows or columns of $\X$, 
can be obtained by solving a convex program with a group-Lasso penalty.
We propose to extend this approach to the simultaneous selection of both rows and columns of $\X$,
with the following convex problem:
\begin{equation}\label{eq:cur}
 \min_{\W \in \Real^{p\times n}} \frac{1}{2}\|\X - \X\W\X\|_{\text{F}}^2
+ \lambda_{\text{row}} \sum_{i=1}^n \|\W^i\|_\infty
+ \lambda_{\text{col}} \sum_{j=1}^p \|\W_j\|_\infty.
\end{equation}
In this formulation, the two sparsity-inducing penalties controlled by the parameters $\lambda_{\text{row}}$ and $\lambda_{\text{col}}$
set to zero some entire rows and columns of the solutions of problem~(\ref{eq:cur}). 
Now, let us denote by~$\W_{\text{I}\, \text{J}}$ in $\Real^{|\text{I}|\times|\text{J}|}$ 
the submatrix of $\W$ reduced to its nonzero rows and columns, respectively indexed by 
$\text{I} \subseteq\{1,\dots,p\}$ and $\text{J}\subseteq\{1,\dots,n\}$.
We can then readily identify the three components of the CUR decomposition of $\X$, namely
$$
\X\W\X = \Cmat \W_{\text{I}\, \text{J}} \Rmat \approx \X.
$$
Problem~(\ref{eq:cur}) has a smooth convex data-fitting term and brings into play a sparsity-inducing norm 
with overlapping groups of variables (the rows and the columns of $\W$). As a result, 
it is a particular instance of problem~(\ref{eq:formulation}) 
that can therefore be handled with the optimization tools introduced in this paper.
We now compare the performance of the sampling procedure from~\citet{mahoney2009cur} with our proposed sparsity-based approach.
To this end, we consider the four gene-expression datasets \texttt{9\_Tumors}, \texttt{Brain\_Tumors1}, 
\texttt{Leukemia1} and \texttt{SRBCT}, with respective dimensions
$(n,p) \in \{(60,5727),(90,5921),(72,5328),(83,2309)\}$.\footnote{The datasets are freely available at \texttt{http://www.gems-system.org/}.}
In the sequel, the matrix $\X$ is normalized to have unit Frobenius-norm while each of its columns is centered.
To begin with, we run our approach\footnote{More precisely, 
since the penalties in problem~(\ref{eq:cur}) shrink the coefficients of $\W$, we follow a two-step procedure: 
We first run our approach to determine the sets of nonzero rows and columns,
and then compute $\W_{\text{I}\, \text{J}} = \Cmat^+ \X \Rmat^+$.}
over a grid of values for 
$\lambda_{\text{row}}$ and $\lambda_{\text{col}}$ in order to obtain solutions with different sparsity levels, i.e., 
ranging from $|\text{I}|=p$ and $|\text{J}|=n$ down to $|\text{I}|=|\text{J}|=0$.
For each pair of values $[|\text{I}|,|\text{J}|]$, we then apply the sampling procedure from~\citet{mahoney2009cur}.
Finally, the variance explained by the CUR decompositions is reported in Figure~\ref{fig:cur_curve} for both methods.
Since the sampling approach involves some randomness, 
we show the average and standard deviation of the results based on five initializations.
The conclusions we can draw from the experiments match the ones already reported in~\citet{bien} for the partial CUR decomposition.
We can indeed see that both schemes perform similarly. 
However, our approach has the advantage not to be randomized, 
which can be less disconcerting in the practical perspective of analyzing a single run of the algorithm.
It is finally worth being mentioned that the convex approach we develop here is flexible and 
can be extended in different ways. For instance, we can imagine to add further low-rank/sparsity constraints on $\W$
thanks to sparsity-promoting convex regularizations.
 \begin{figure}[h]
    \centering
   \includegraphics[width=0.45\textwidth]{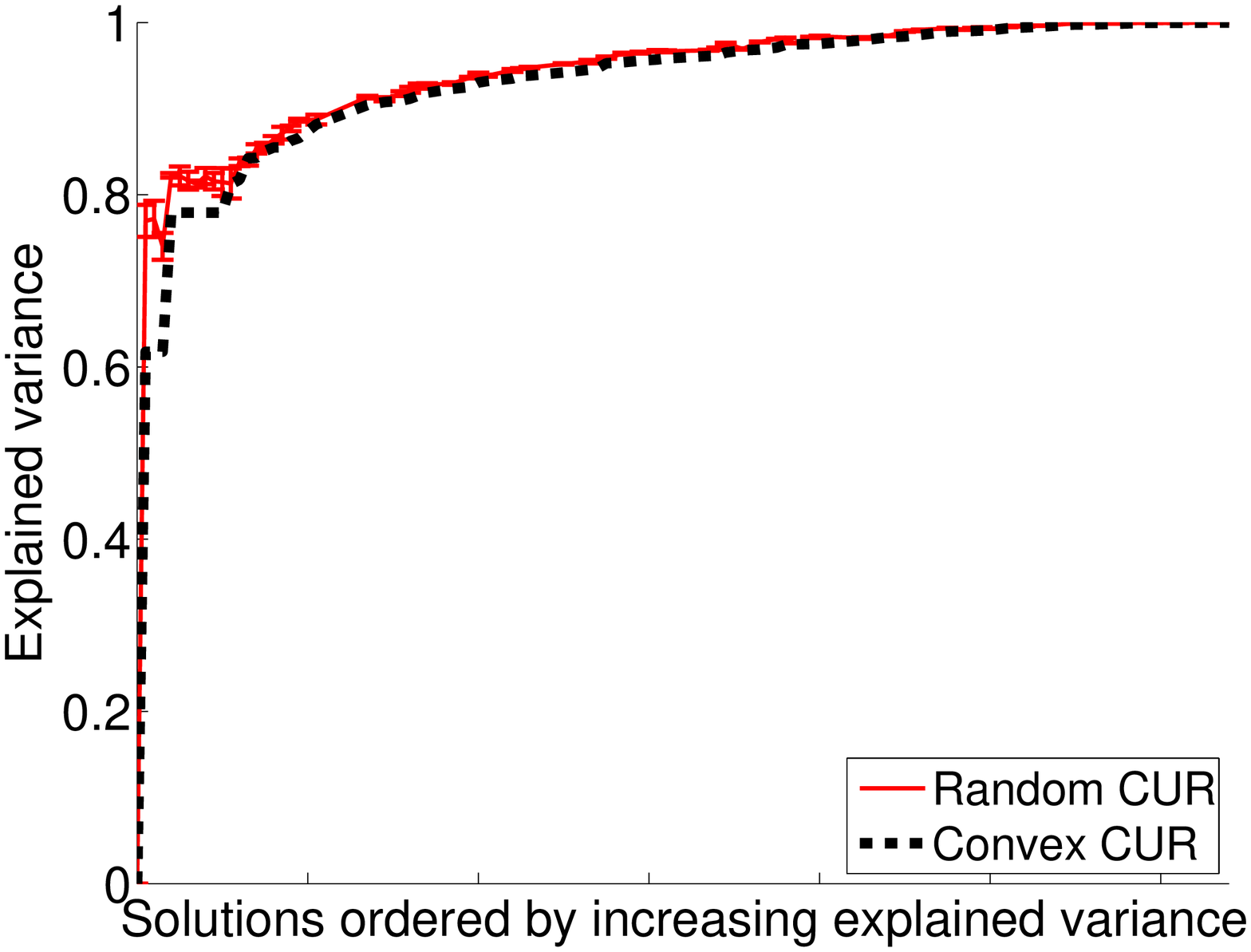} \hfill
   \includegraphics[width=0.45\textwidth]{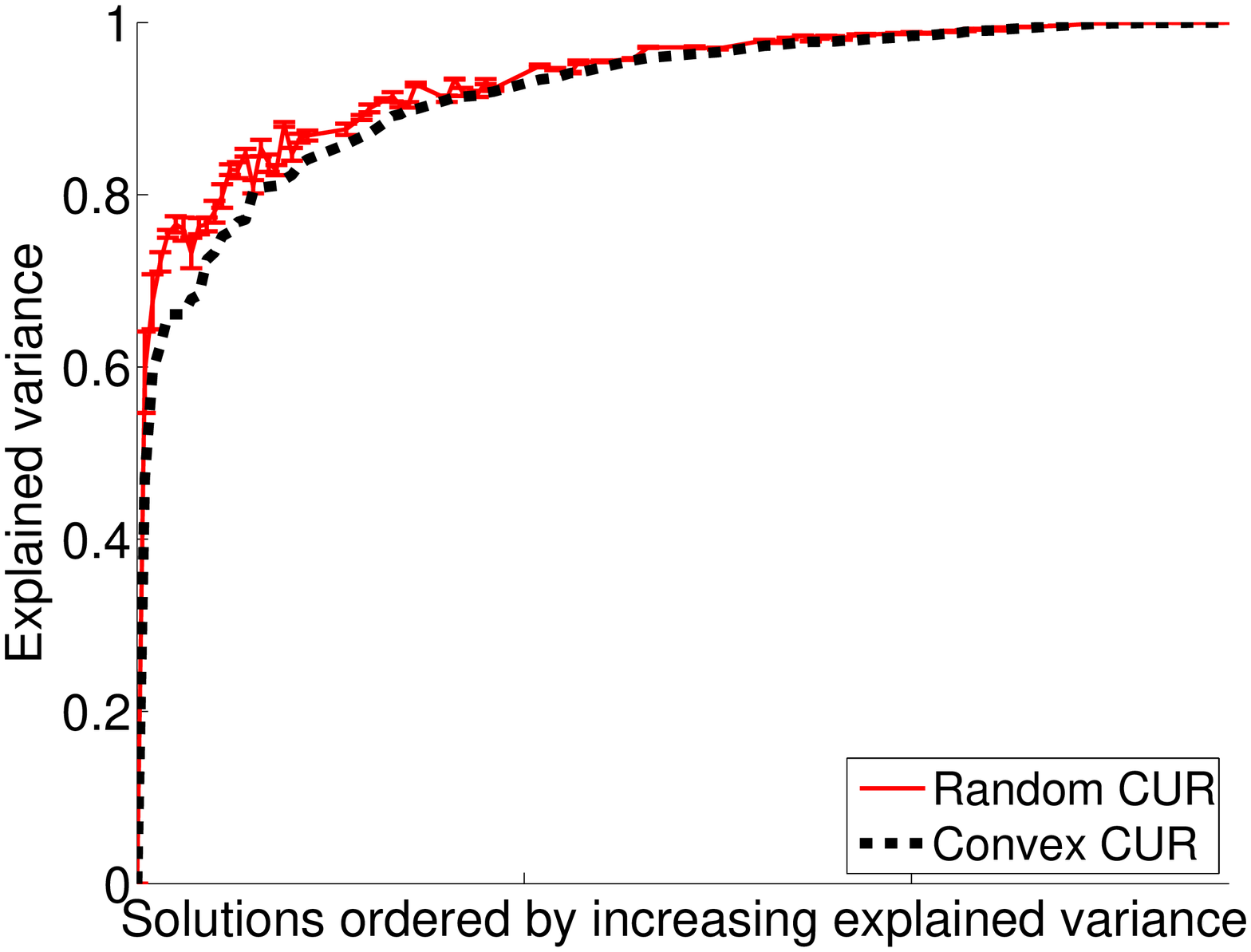}
    \includegraphics[width=0.45\textwidth]{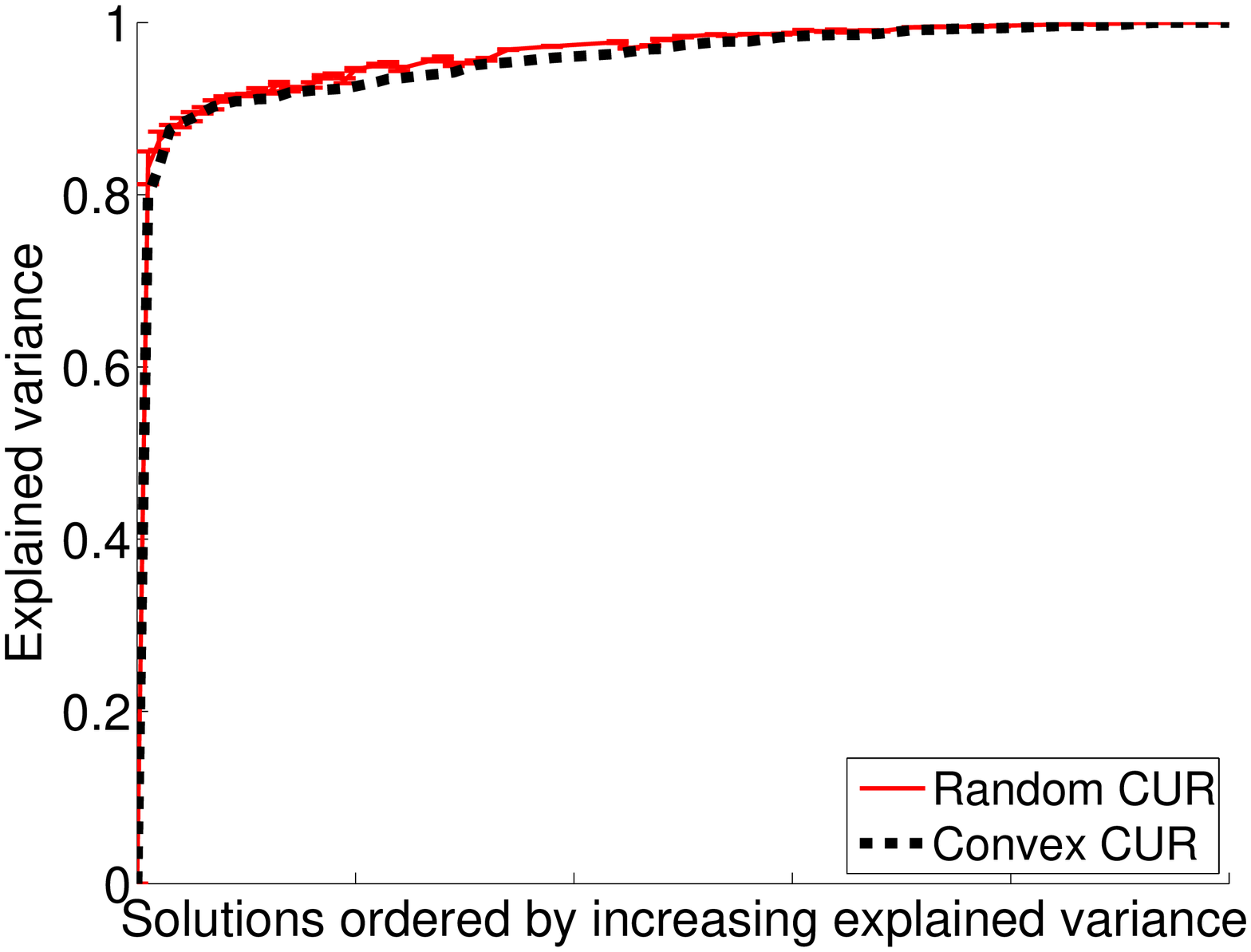} \hfill
   \includegraphics[width=0.45\textwidth]{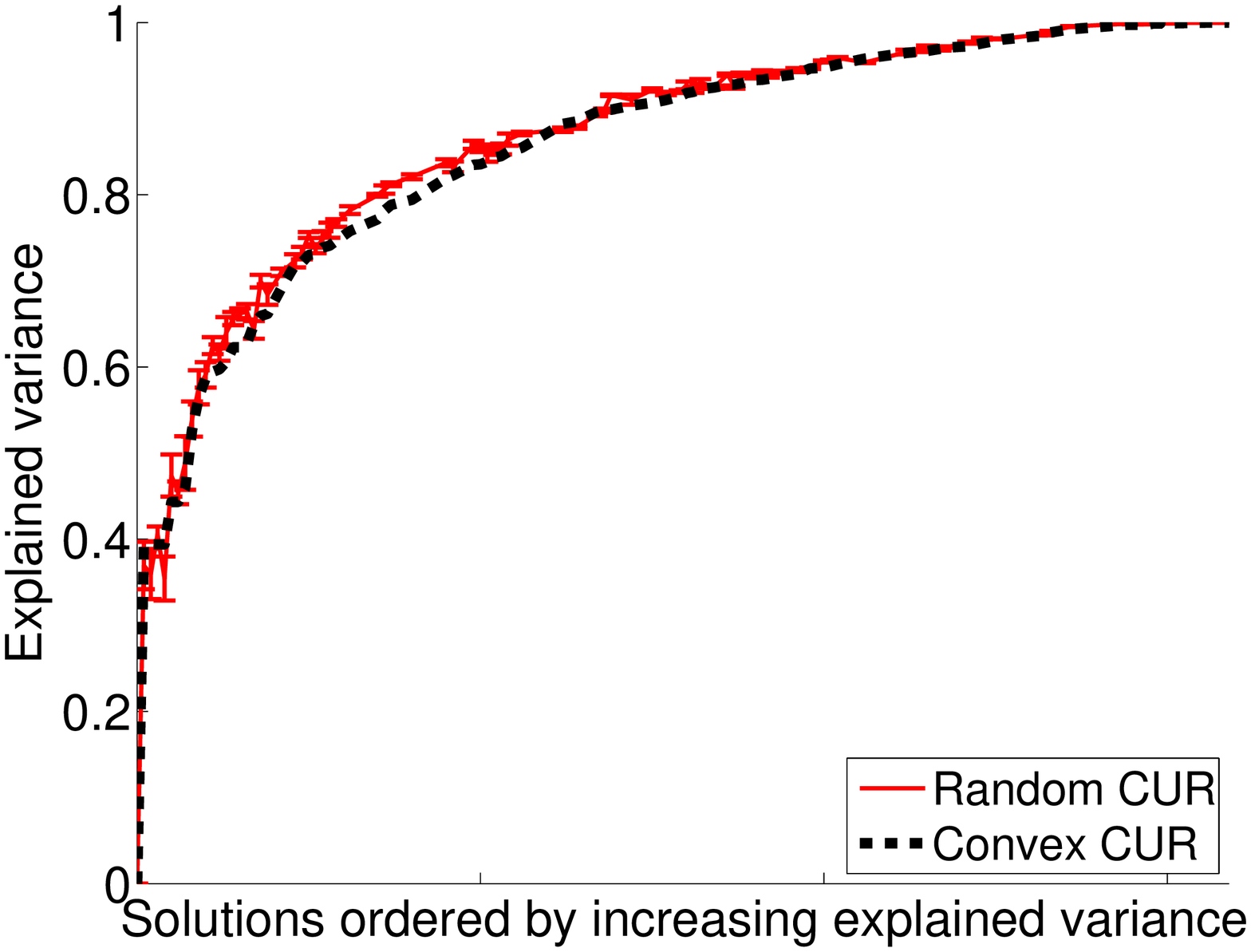}
   \caption{Explained variance of the CUR decompositions obtained for our sparsity-based approach 
   and the sampling scheme from~\citet{mahoney2009cur}. 
   For the latter, we report the average and standard deviation of the results based on five initializations.
   From left to right and top to bottom, the curves correspond to the datasets \texttt{9\_Tumors}, \texttt{Brain\_Tumors1},
   \texttt{Leukemia1} and \texttt{SRBCT}.
   }\label{fig:cur_curve}
 \end{figure}

\subsection{Background Subtraction}
Following~\citet{cehver, huang}, we consider a background subtraction task. Given a sequence of frames from a fixed camera,
we try to segment out foreground objects in a new image. 
If we denote by $\y\in\R{n}$ this image composed of $n$ pixels, 
we model $\y$ as a sparse linear combination of $p$ other images $\X\in\RR{n}{p}$, plus an error term $\e$ in $\R{n}$,
i.e., $\y \approx \X \w + \e$ for some sparse vector $\w$ in $\R{p}$. 
This approach is reminiscent of~\citet{yima} in the context of face recognition, 
where $\e$ is further made sparse to deal with small occlusions.
The term $\X \w$ accounts for \textit{background} parts present in both~$\y$ and~$\X$, while $\e$ contains specific, 
or \textit{foreground}, objects in $\y$.
 The resulting optimization problem is given by
\begin{equation}\label{eq:background_sub}
\min_{\w\in\Real^{p},\e\in\Real^{n}} \frac{1}{2} \|\y \! - \! \X\w \! - \! \e\|_2^2 + \lambda_1\|\w\|_1 + 
\lambda_2 \{ \|\e\|_1 + \Omega(\e) \},\ \text{with}\ \lambda_1,\lambda_2 \geq 0. 
\end{equation}
In this formulation, 
the only $\ell_1$-norm penalty does not take into account the fact that
neighboring pixels in $\y$ are likely to share the same label (background or foreground),
which may lead to scattered pieces of foreground and background regions (Figure~\ref{fig:background_sub}).
We therefore put an additional structured regularization term $\Omega$ on $\e$, 
where the groups in $\G$ are all the overlapping $3\!\times\!3$-squares on the image.
For the sake of comparison, we also consider the regularization~$\tilde{\Omega}$
where the groups are \emph{non-overlapping} $3\!\times\!3$-squares.

This optimization problem can be viewed as an instance of problem~(\ref{eq:formulation}), 
with the particular design matrix $[\X,\ \I]\ \text{in}\ \Real^{n \times (p+n)}$,
defined as the columnwise concatenation of $\X$ and the identity matrix.
As a result, we could directly apply the same procedure as the one used in the other experiments.
Instead, we further exploit the specific structure of problem~(\ref{eq:background_sub}): Notice that
for a fixed vector $\e$, the optimization with respect to $\w$ is a standard Lasso problem 
(with the vector of observations $\y-\e$),\footnote{Since successive frames might not change much,
the columns of $\X$ exhibit strong correlations. Consequently, 
we use the LARS algorithm~\citep{efron} whose complexity is independent of the level of correlation in $\X$.} 
while for $\w$ fixed, 
we simply have a proximal problem associated to the sum of $\Omega$ and the $\ell_1$-norm. 
Alternating between these two simple and computationally inexpensive steps, 
i.e., optimizing with respect to one variable while keeping the other one fixed, 
is guaranteed to converge to a solution of~(\ref{eq:background_sub}).\footnote{More precisely,
the convergence is guaranteed 
since the non-smooth part in~(\ref{eq:background_sub}) is \textit{separable} with respect to $\w$ and~$\e$ \citep{Tseng2001}.
The result from~\citet{bertsekas} may also be applied here, after reformulating~(\ref{eq:background_sub}) 
as a smooth convex problem under separable conic constraints.}
In our simulations, this alternating scheme has led to a significant speed-up compared to the general procedure.

A dataset with hand-segmented images is used to illustrate
the effect of $\Omega$.\footnote{{\scriptsize\texttt{http://research.microsoft.com/en-us/um/people/jckrumm/wallflower/testimages.htm}}}
For simplicity, we use a single regularization parameter, i.e., $\lambda_1=\lambda_2$, chosen to maximize the number of pixels matching the ground truth.
We consider $p=200$ images with $n=57600$ pixels (i.e., a resolution of $120\!\times\!160$, times 3 for the RGB channels).
As shown in Figure~\ref{fig:background_sub}, adding $\Omega$ improves the background subtraction results for the two tested images, 
by removing the scattered artifacts due to the lack of structural constraints of the $\ell_1$-norm, which encodes neither spatial nor color consistency. 
The group sparsity regularization~$\tilde{\Omega}$ also improves upon the~$\ell_1$-norm but introduces block-artefacts corresponding to the non-overlapping group structure.
\begin{figure}[hbtp]
\subfloat[Original frame.]{ \includegraphics[width=0.32\textwidth]{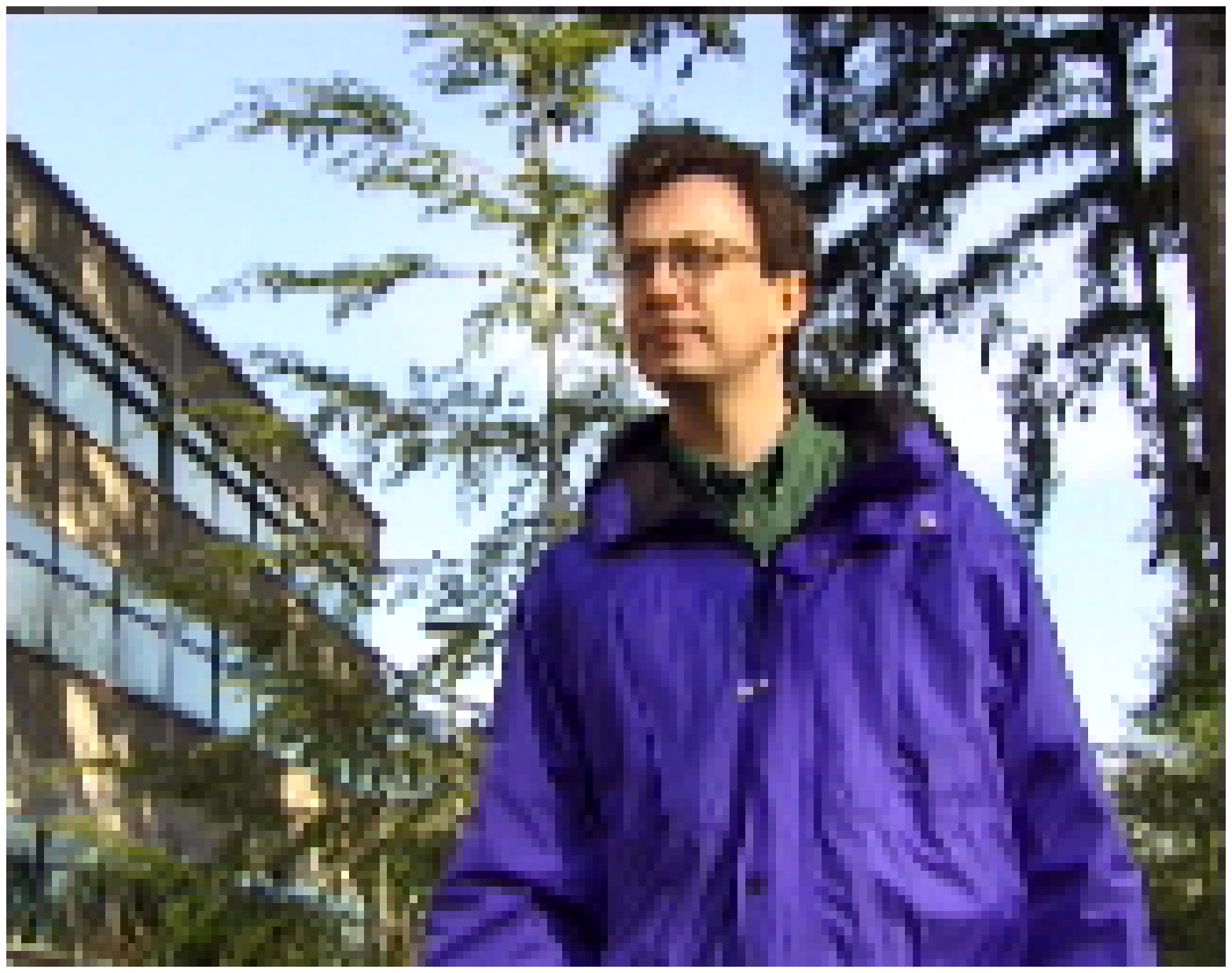} } \hfill
\subfloat[Estimated background with~$\Omega$.]{ \includegraphics[width=0.32\textwidth]{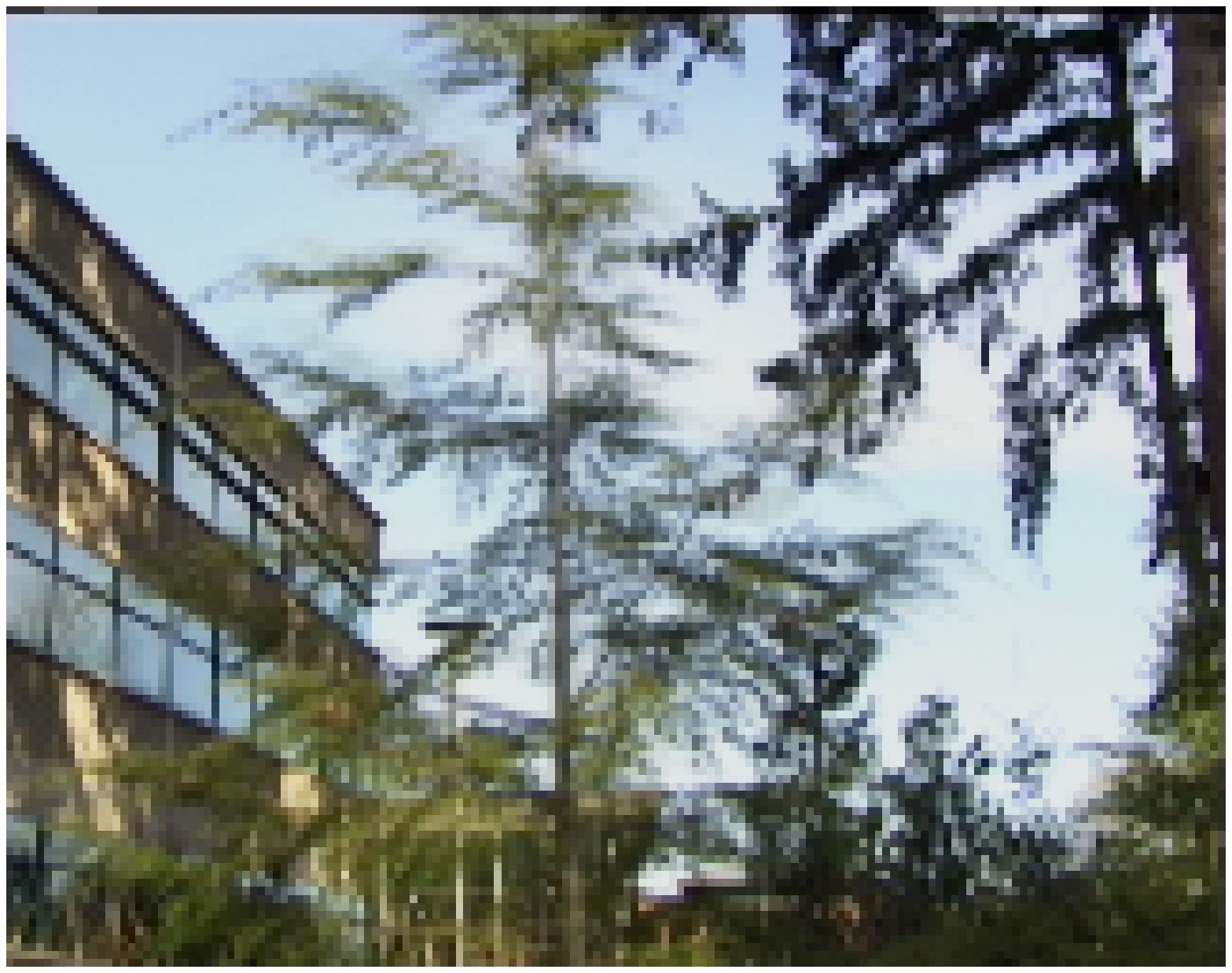} }\hfill
\subfloat[$\ell_1$, $87.1\%$.]{ \includegraphics[width=0.32\textwidth]{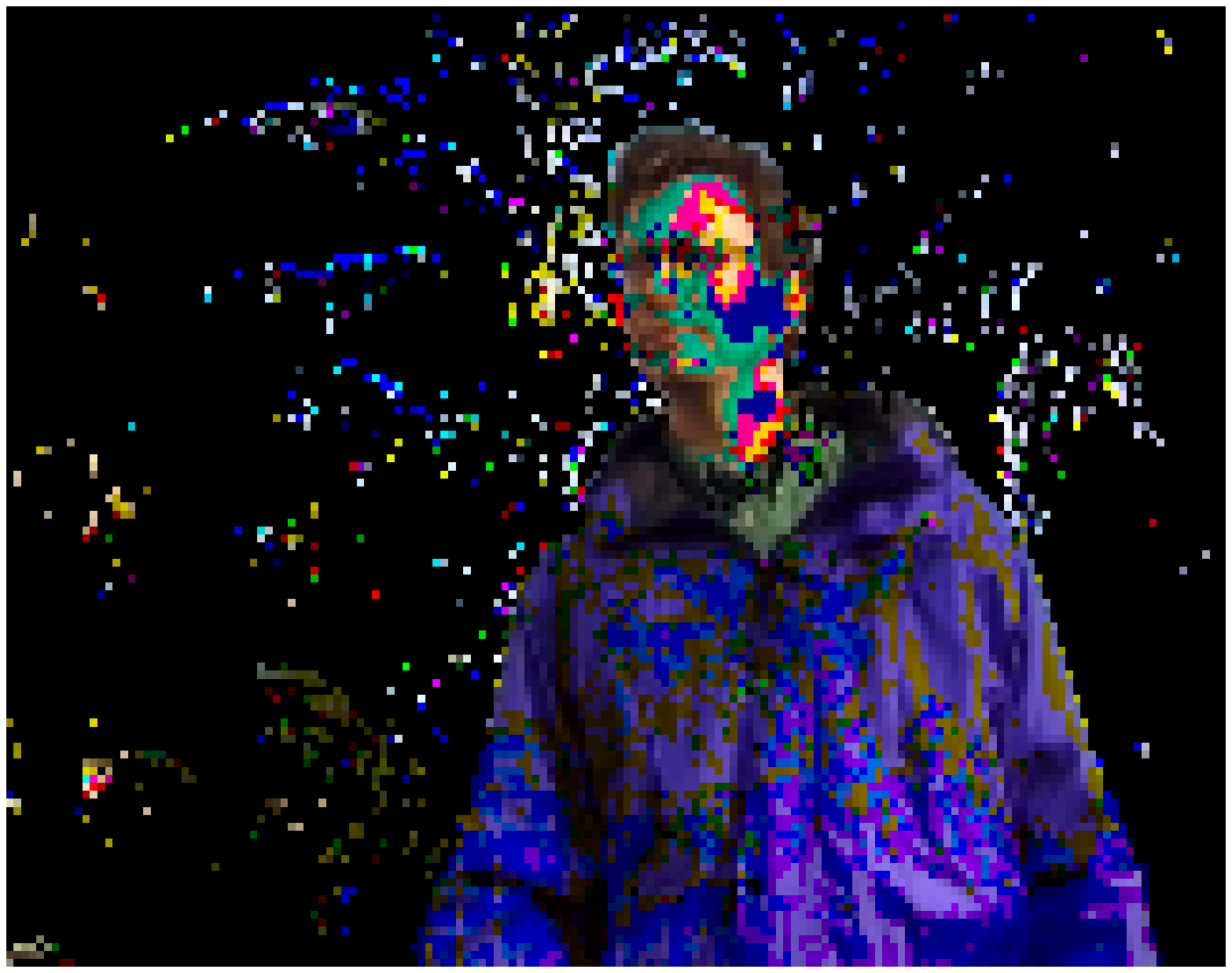} } \\
\subfloat[$\ell_1+\tilde{\Omega}$ (non-overlapping), $96.3\%$.]{ \includegraphics[width=0.32\textwidth]{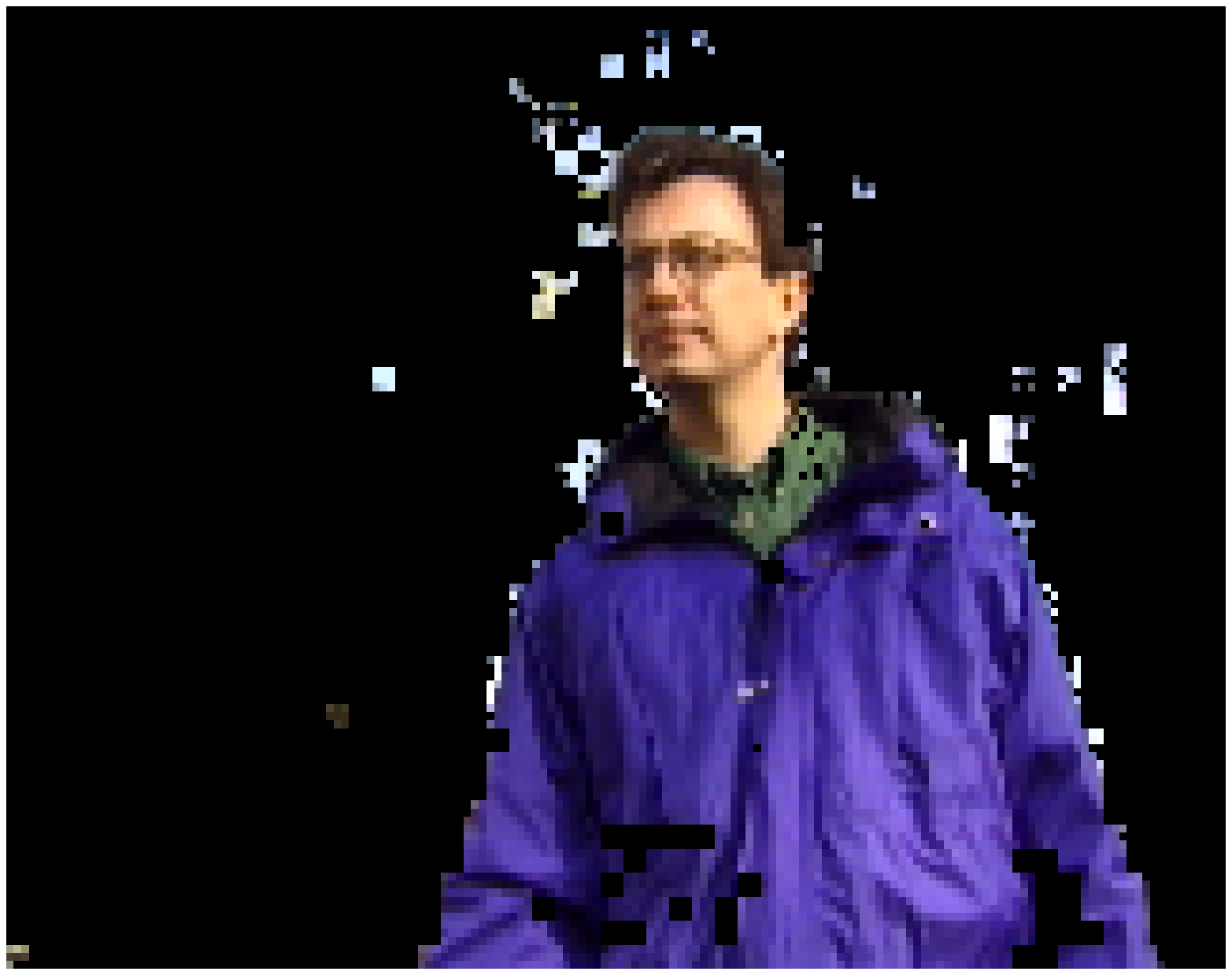} }\hfill
\subfloat[$\ell_1+\Omega$ (overlapping), $98.9\%$.]{ \includegraphics[width=0.32\textwidth]{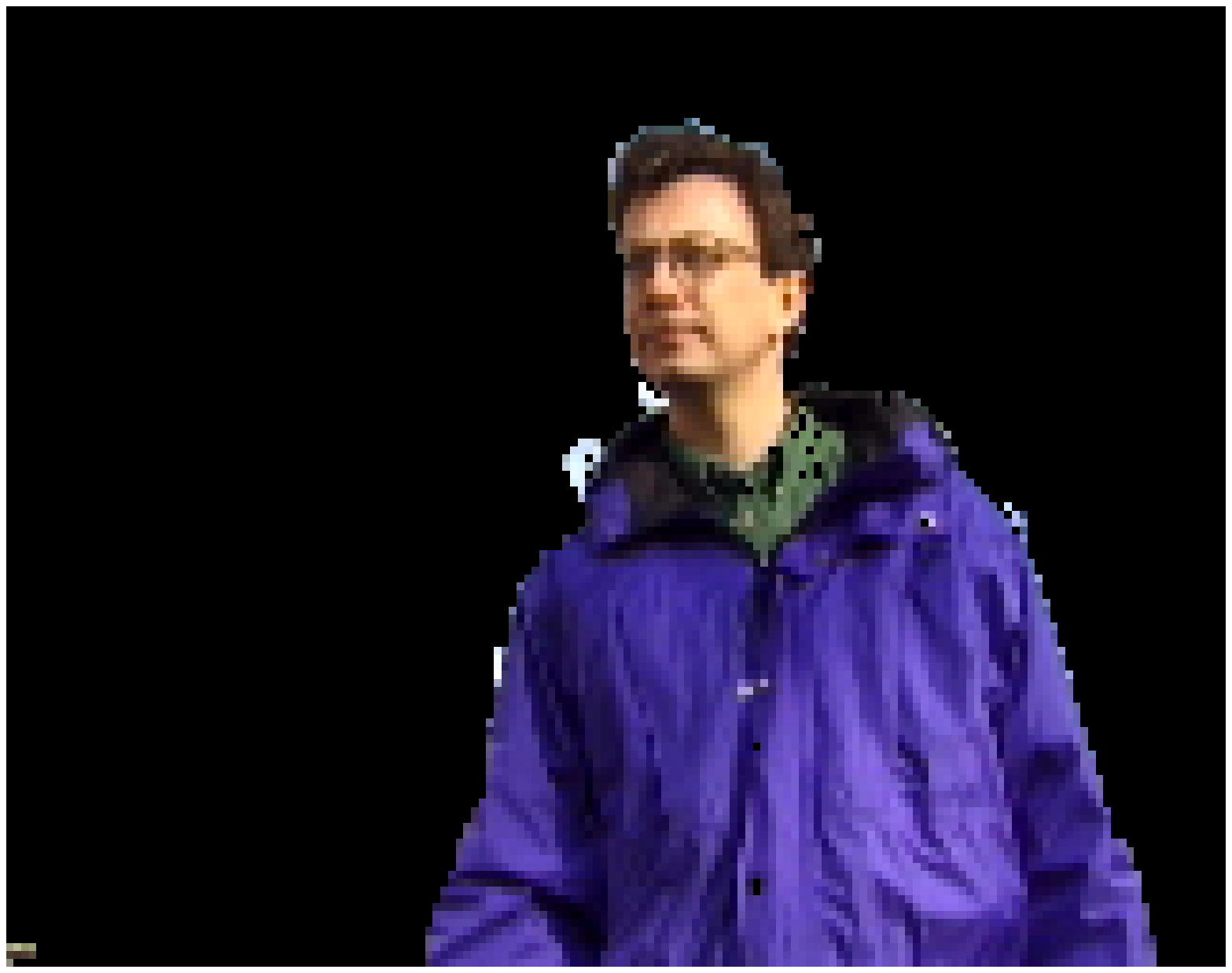} }\hfill
\subfloat[$\Omega$, another frame.]{ \includegraphics[width=0.32\textwidth]{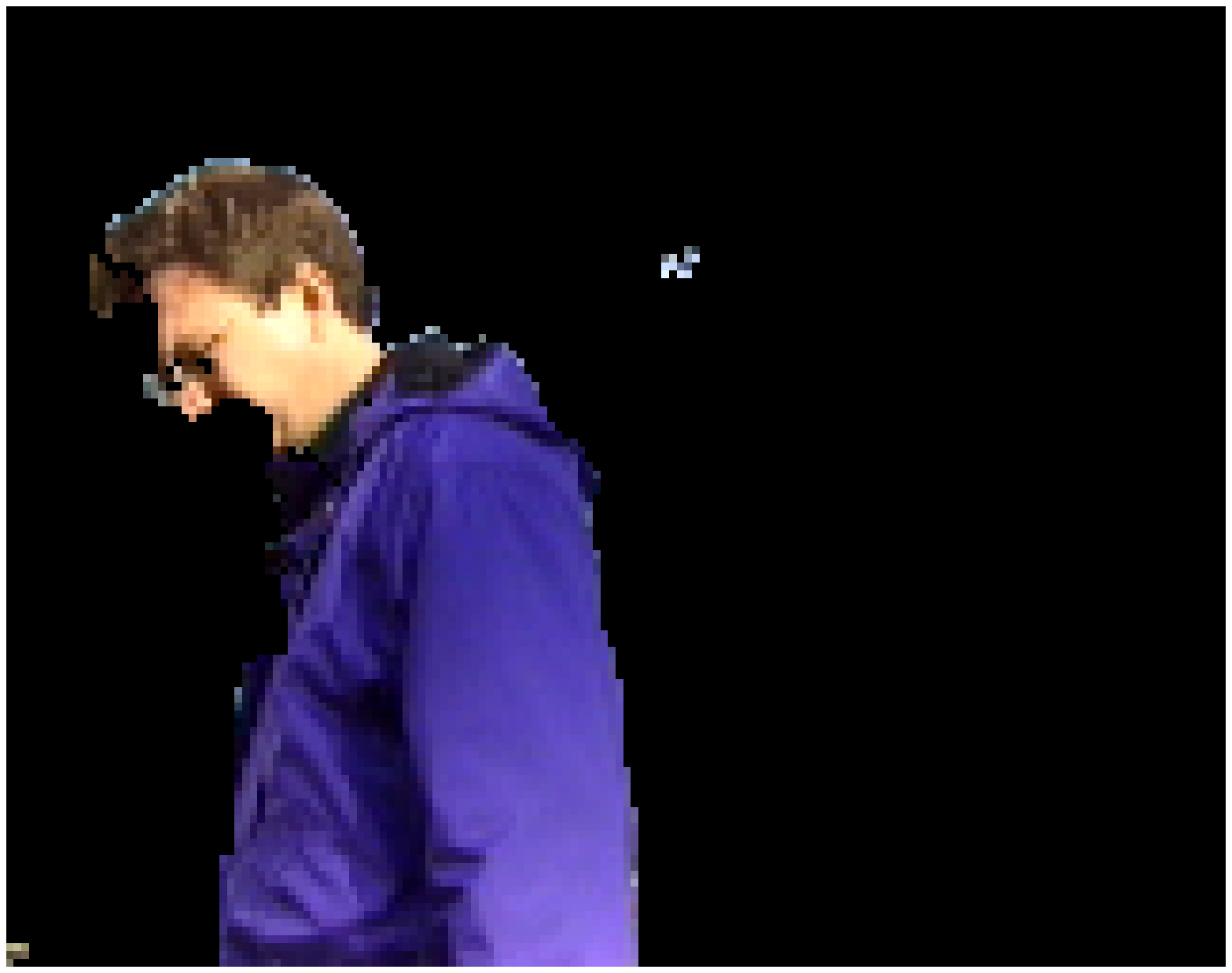} \label{subig:frame1}} \\
\subfloat[Original frame.]{ \includegraphics[width=0.32\textwidth]{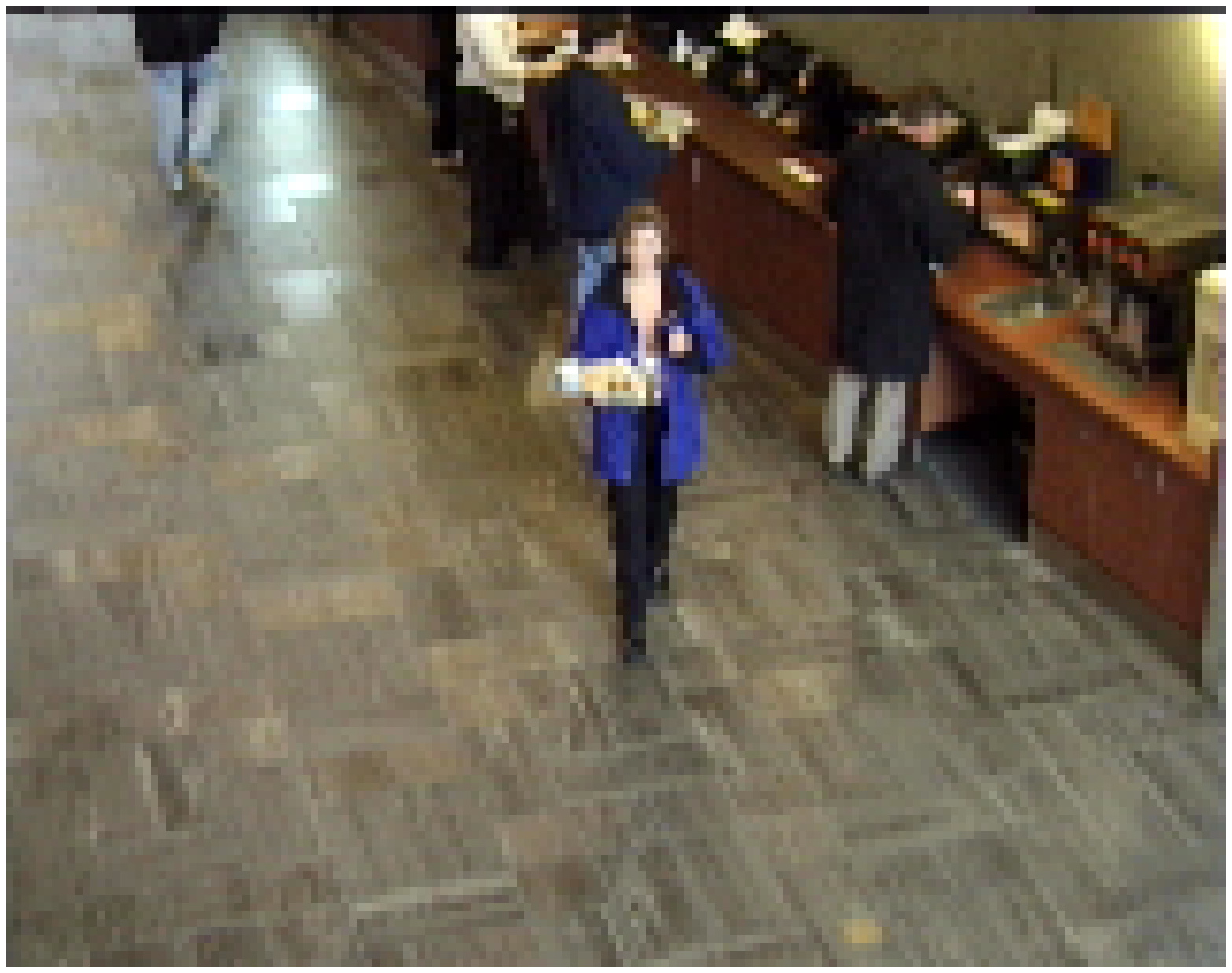} } \hfill
\subfloat[Estimated background with~$\Omega$.]{ \includegraphics[width=0.32\textwidth]{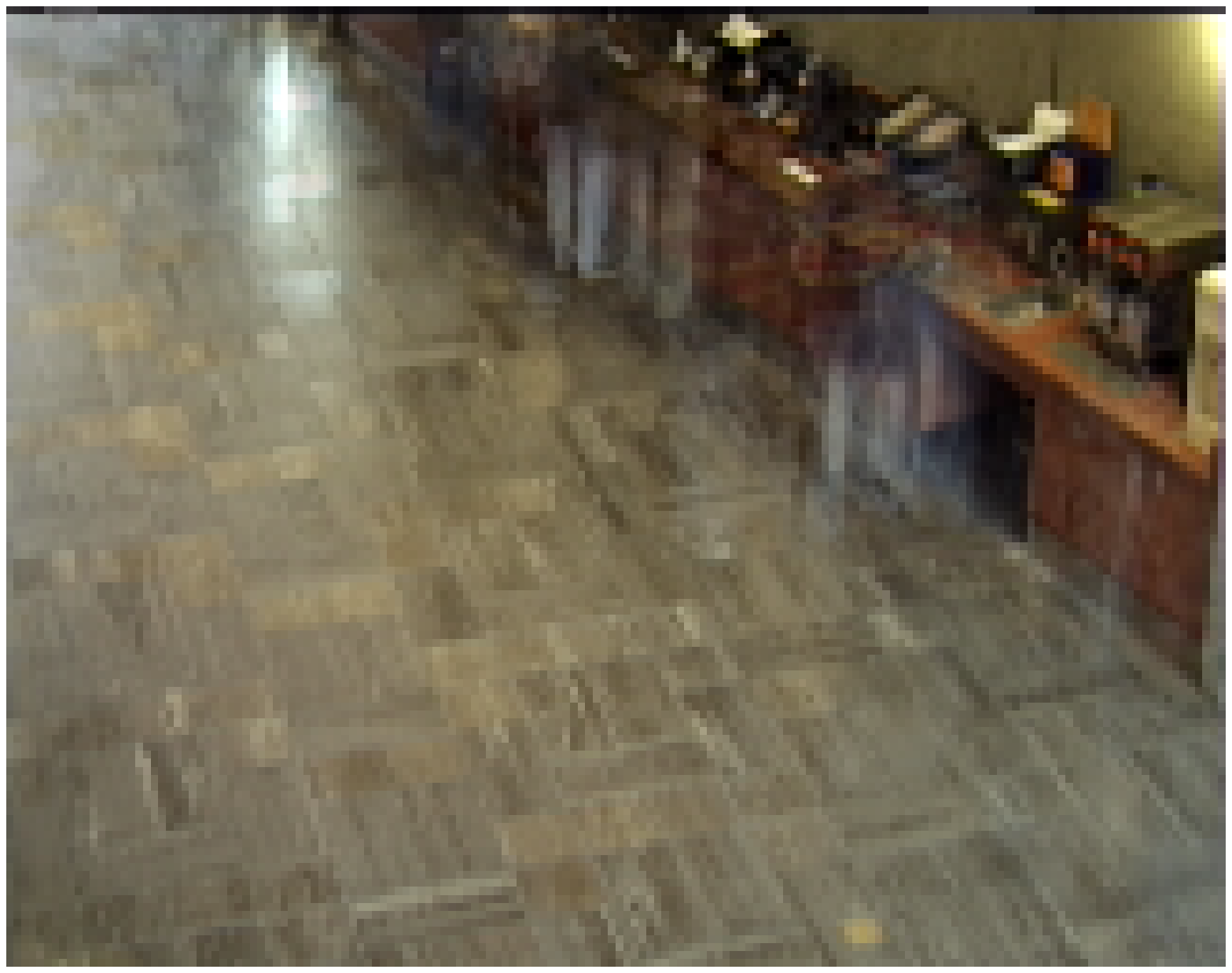} }\hfill
\subfloat[$\ell_1$, $90.5\%$.]{ \includegraphics[width=0.32\textwidth]{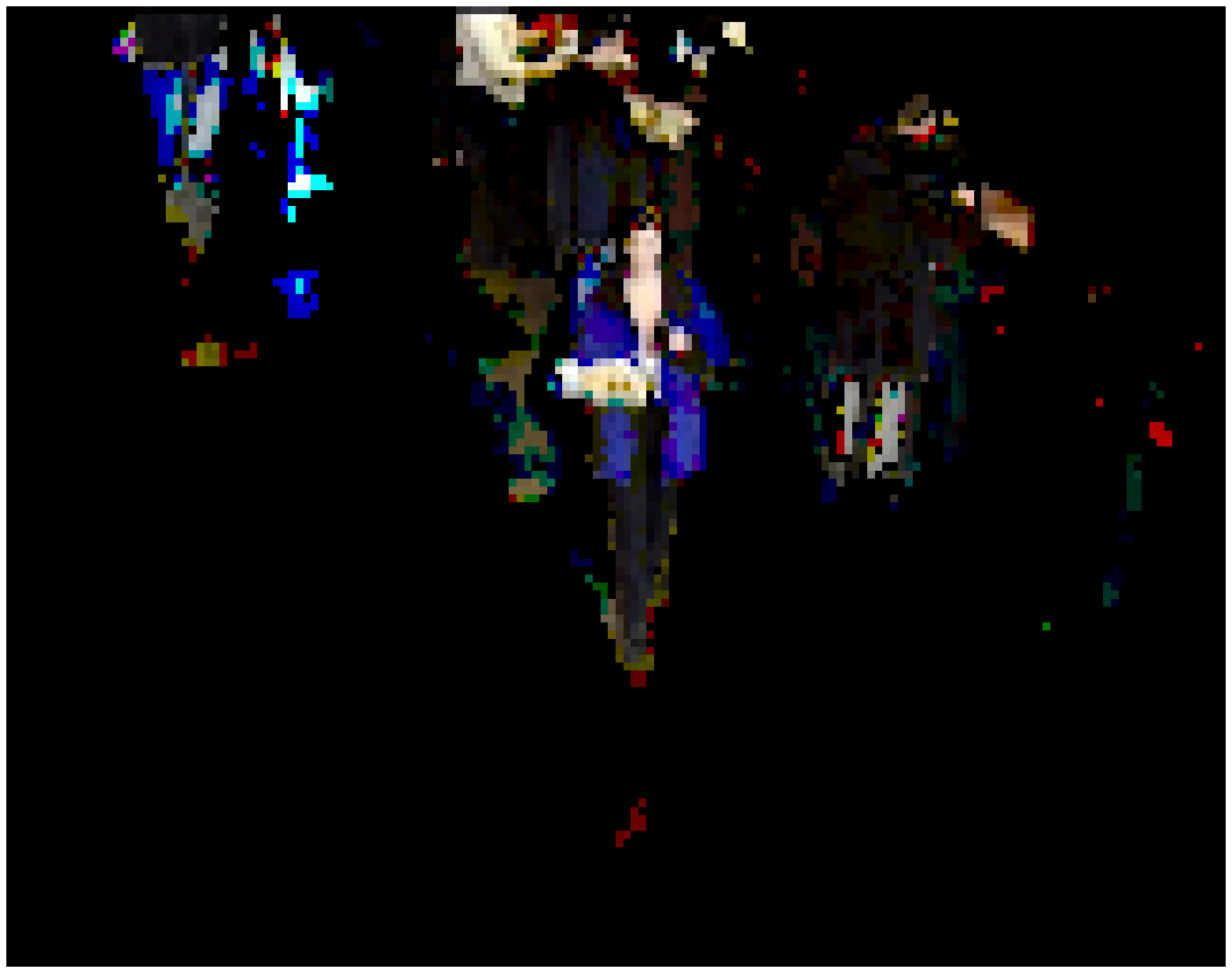} } \\
\subfloat[$\ell_1+\tilde{\Omega}$ (non-overlapping), $92.6\%$.]{ \includegraphics[width=0.32\textwidth]{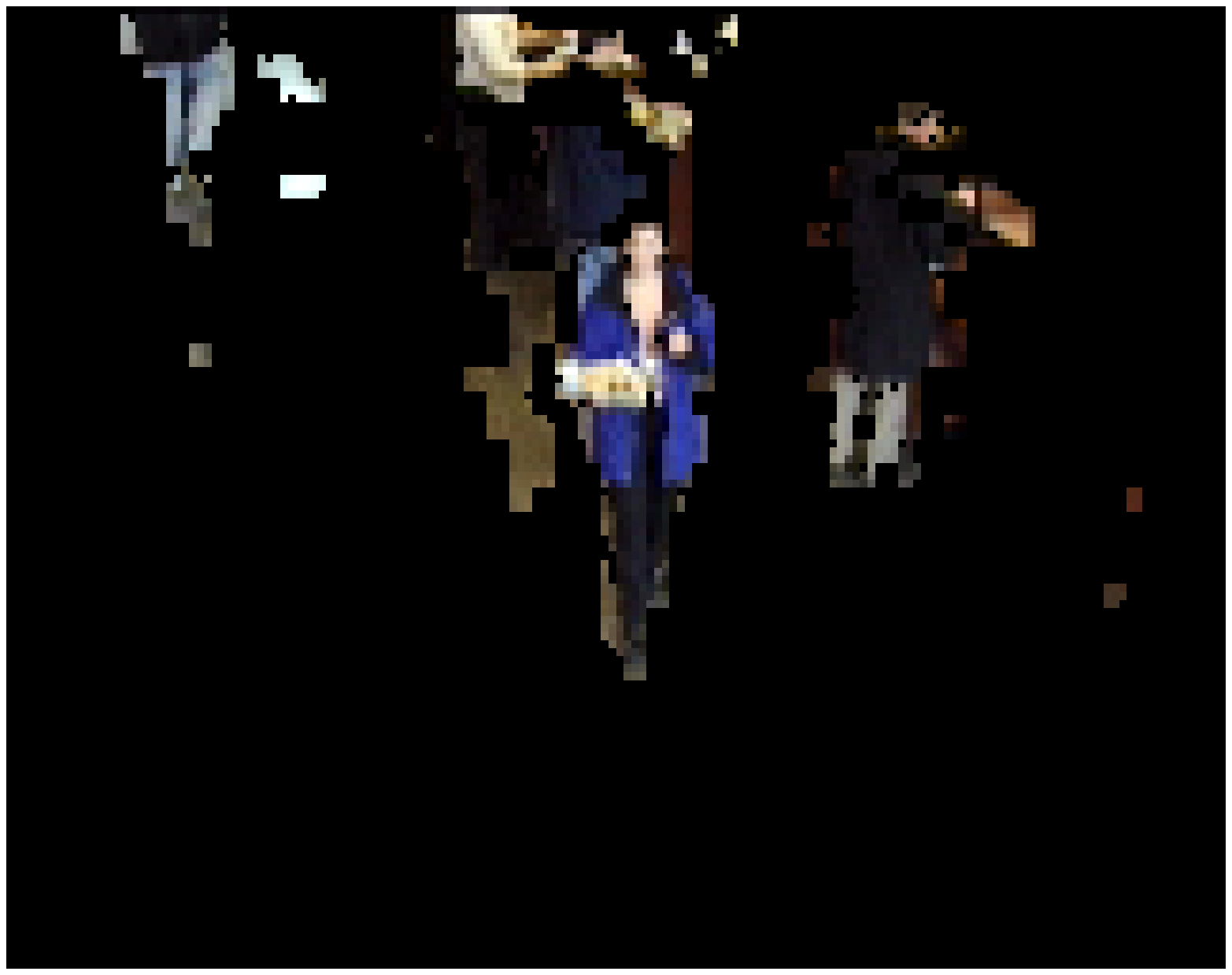} }\hfill
\subfloat[$\ell_1+\Omega$ (overlapping), $93.8\%$.]{ \includegraphics[width=0.32\textwidth]{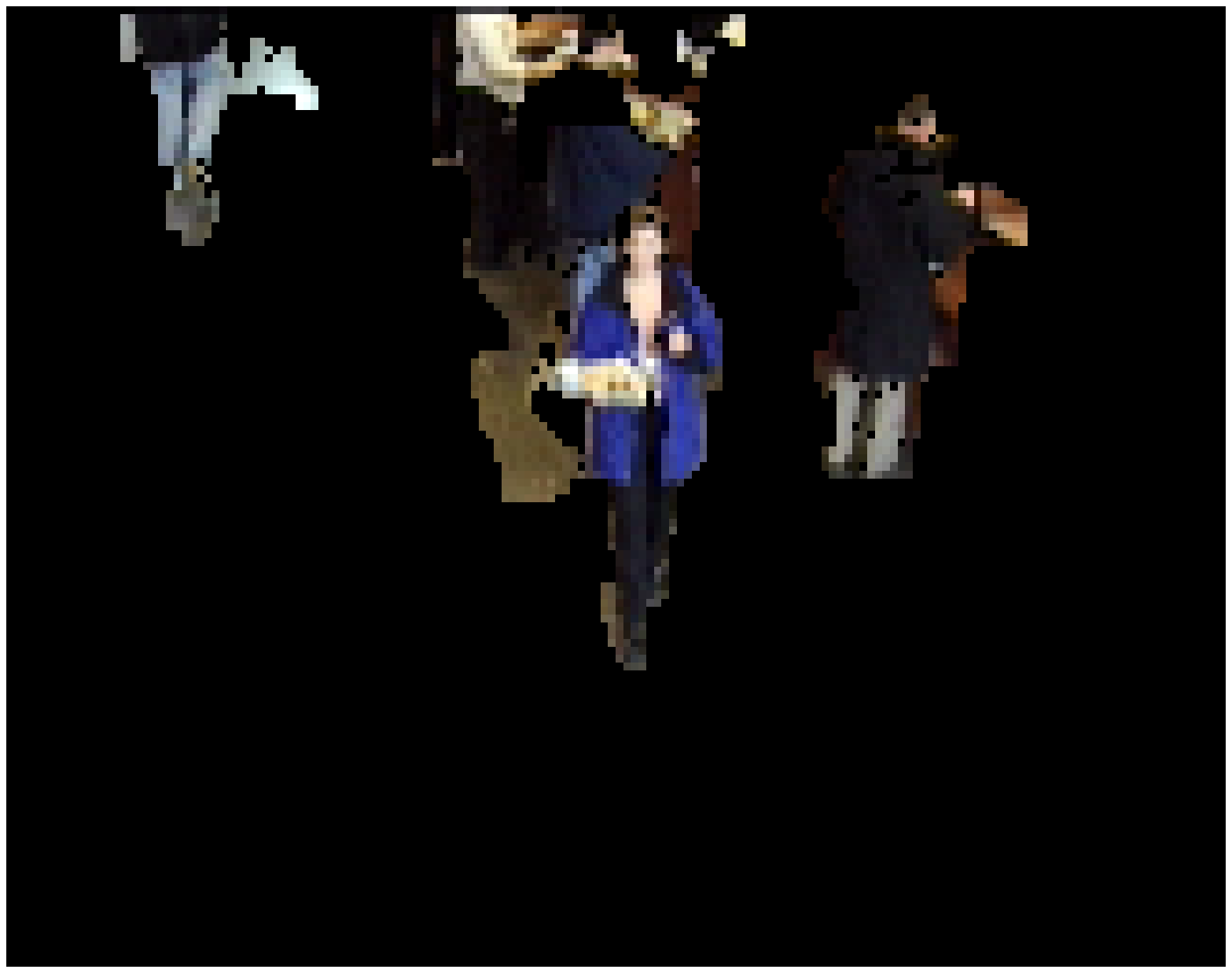} }\hfill
\subfloat[$\Omega$, another frame.]{ \includegraphics[width=0.32\textwidth]{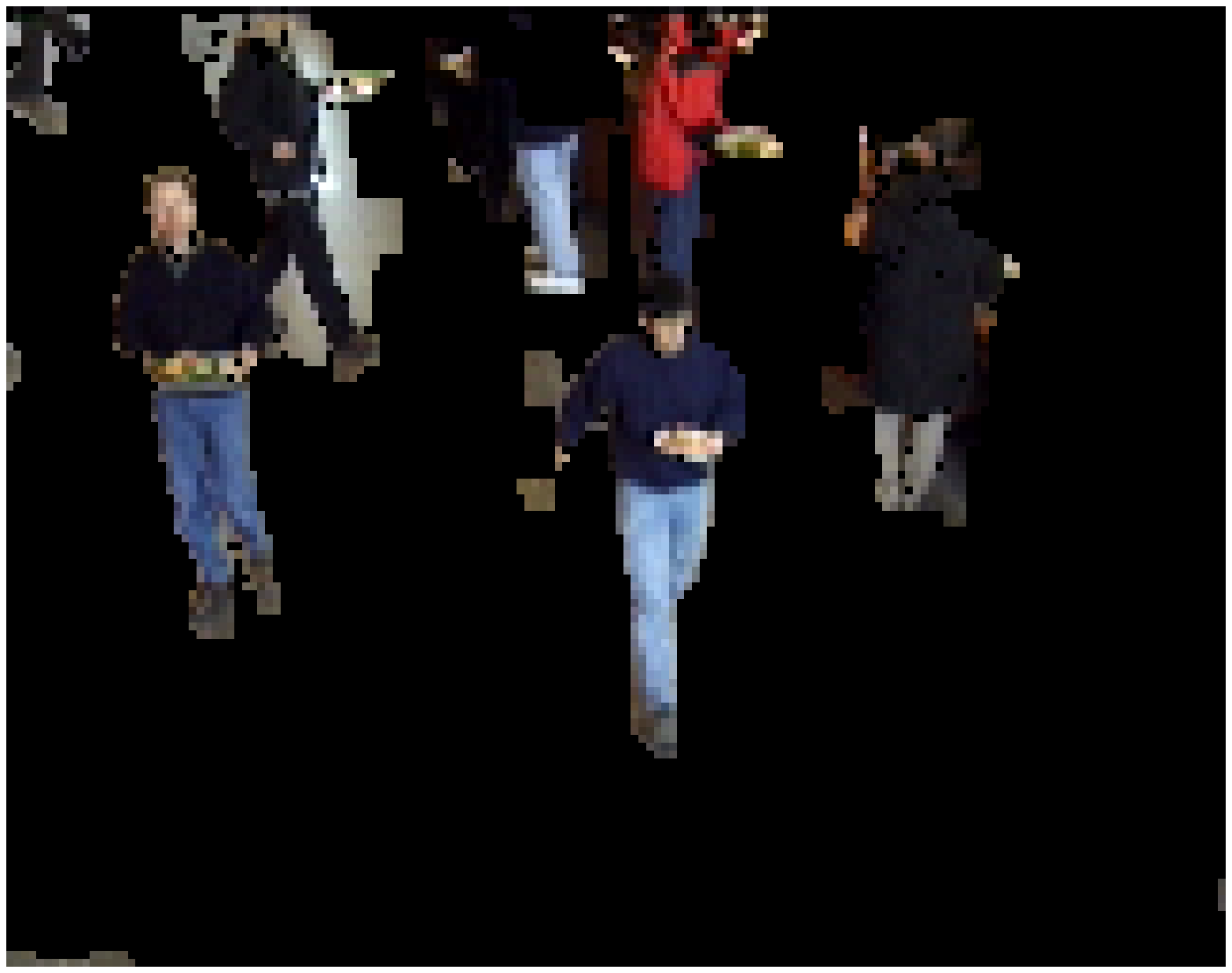} \label{subfig:frame2}} 
\caption{Background subtraction results. For two videos, we present the original image $\y$, the estimated background (i.e., $\X\w$) reconstructed by our method,
and the foreground (i.e., the sparsity pattern of $\e$ as a mask on the original image) detected with $\ell_1$, $\ell_1+\tilde{\Omega}$ (non-overlapping groups) and with $\ell_1+\Omega$.
Figures~\subref{subig:frame1} and~\subref{subfig:frame2} present another foreground found with $\Omega$, on a different image, with the same values of $\lambda_1,\lambda_2$ as for the previous image.
Best seen in color.
}\label{fig:background_sub}
\end{figure}

\subsection{Topographic Dictionary Learning} \label{subsec:topo}
Let us consider a set $ \Y = [\y^1,\dots,\y^n]$ in $\RR{m}{n}$ of $n$ signals
of dimension $m$. The problem of dictionary learning, originally introduced
by~\citet{field2}, is a matrix factorization problem which
aims at representing these signals as linear combinations of 
\textit{dictionary elements} that are the columns of a matrix $\X =
[\x^1,\dots,\x^p]$ in $\RR{m}{p}$.  More precisely, the dictionary $\X$ is
\textit{learned} along with a matrix of \textit{decomposition coefficients} $\W
= [\w^1,\dots,\w^n]$ in~$\RR{p}{n}$, so that $\y^i \approx \X \w^i$ for every
signal~$\y^i$. Typically, $n$ is large compared to $m$ and $p$. In this experiment, 
we consider for instance a database of $n=100\,000$ natural image patches of
size $m=12 \times 12$ pixels, for dictionaries of size $p=400$.
Adapting the dictionary to specific data has proven to be useful in many
applications, including image restoration~\citep{elad,mairal8}, learning image
features in computer vision~\citep{kavukcuoglu2}. The resulting optimization
problem we are interested in can be written
\begin{equation}
    \min_{\X \in \C,\W \in \Real^{p \times n}} \sum_{i=1}^n\frac{1}{2} \|\y^i-\X\w^i\|_2^2 + \lambda \Omega(\w^i),
    \label{eq:dict_topo_learning}
\end{equation}
where $\C$ is a convex set of matrices in $\Real^{m \times p}$ whose columns
have $\ell_2$-norms less than or equal to one,\footnote{Since the quadratic term in
Eq.~(\ref{eq:dict_topo_learning}) is invariant by multiplying $\X$ by a scalar
and $\W$ by its inverse, constraining the norm of~$\X$ has proven to be necessary in
practice to prevent it from being arbitrarily large.}
$\lambda$ is a regularization parameter and $\Omega$ is a sparsity-inducing norm.
When $\Omega$ is the $\ell_1$-norm, we obtain a classical formulation, which is
known to produce dictionary elements that are reminiscent of Gabor-like
functions, when the columns of $\Y$ are whitened natural image
patches~\citep{field2}.

Another line of research tries to put a structure on decomposition coefficients instead
of considering them as independent.  \citet{jenatton3,jenatton4} have for
instance embedded dictionary elements into a tree, by using a
hierarchical norm~\citep{zhao} for $\Omega$.  This model encodes a rule saying
that a dictionary element can be used in the decomposition of a signal only if
its ancestors in the tree are used as well. In the related context of
independent component analysis (ICA), \citet{hyvarinen2} have arranged
independent components (corresponding to dictionary elements) on a
two-dimensional grid, and have modelled spatial dependencies between them. When
learned on whitened natural image patches, this model exhibits ``Gabor-like''
functions which are smoothly organized on the grid, which the authors call a
topographic map.  As shown by~\citet{kavukcuoglu2}, such a result can be
reproduced with a dictionary learning formulation, using a structured norm for
$\Omega$.  Following their formulation, we organize the $p$ dictionary elements
on a $\sqrt{p} \times \sqrt{p}$ grid, and consider $p$
overlapping groups that are $3 \times 3$ or $4 \times 4$ spatial neighborhoods
on the grid (to avoid boundary effects, we assume the grid to be cyclic). We
define $\Omega$ as a sum of $\ell_2$-norms over these groups, since the
$\ell_\infty$-norm has proven to be less adapted for this task.  Another
formulation achieving a similar effect was also proposed by~\citet{garrigues}
in the context of sparse coding with a probabilistic model.

As~\citet{kavukcuoglu2,field2}, we consider a projected stochastic gradient descent
algorithm for learning $\X$---that is, at iteration $t$, we randomly draw one
signal~$\y^t$ from the database $\Y$, compute a sparse code $\w^t = \argmin_{\w
\in \Real^p}\frac{1}{2}\|\y^t-\X\w^t\|_2^2 + \lambda \Omega(\w)$, and
update~$\X$ as follows: $\X \leftarrow \Pi_\C[\X- \rho (\X\w^t-\y^t)\w^{t \top}]$,
where $\rho$ is a fixed learning rate, and $\Pi_\C$ denotes the operator performing
orthogonal projections onto the set~$\C$.  In practice, to further improve the
performance, we use a mini-batch, drawing $500$ signals at eatch iteration
instead of one~\citep[see][]{mairal7}. Our approach mainly differs
from~\citet{kavukcuoglu2} in the way the sparse codes $\w^t$ are obtained.
Whereas~\citet{kavukcuoglu2} uses a subgradient descent algorithm to solve them,
we use the proximal splitting methods presented in Section~\ref{sec:optim_prox2}.
The natural image patches we use are also preprocessed: They are first centered by
removing their mean value (often called DC component), and whitened, as 
often done in the literature~\citep{hyvarinen2,garrigues}.
The parameter~$\lambda$ is chosen such that in average $\|\y^i-\X\w^i\|_2
\approx 0.4\|\y^i\|_2$ for all new patch considered by the algorithm.
Examples of obtained results are shown on Figure~\ref{fig:topodict},
and exhibit similarities with the topographic maps of~\citet{hyvarinen2}.
Note that even though Eq.~(\ref{eq:dict_topo_learning}) is convex with respect
to each variable~$\X$ and~$\W$ when one fixes the other, it is not jointly
convex, and one can not guarantee our method to find a global optimum.
Despite its intrinsic non-convex nature, local minima obtained with
various optimization procedures have been shown to be good enough
for many tasks~\citep{elad,mairal8,kavukcuoglu2}.
\begin{figure}
  \includegraphics[width=0.48\linewidth]{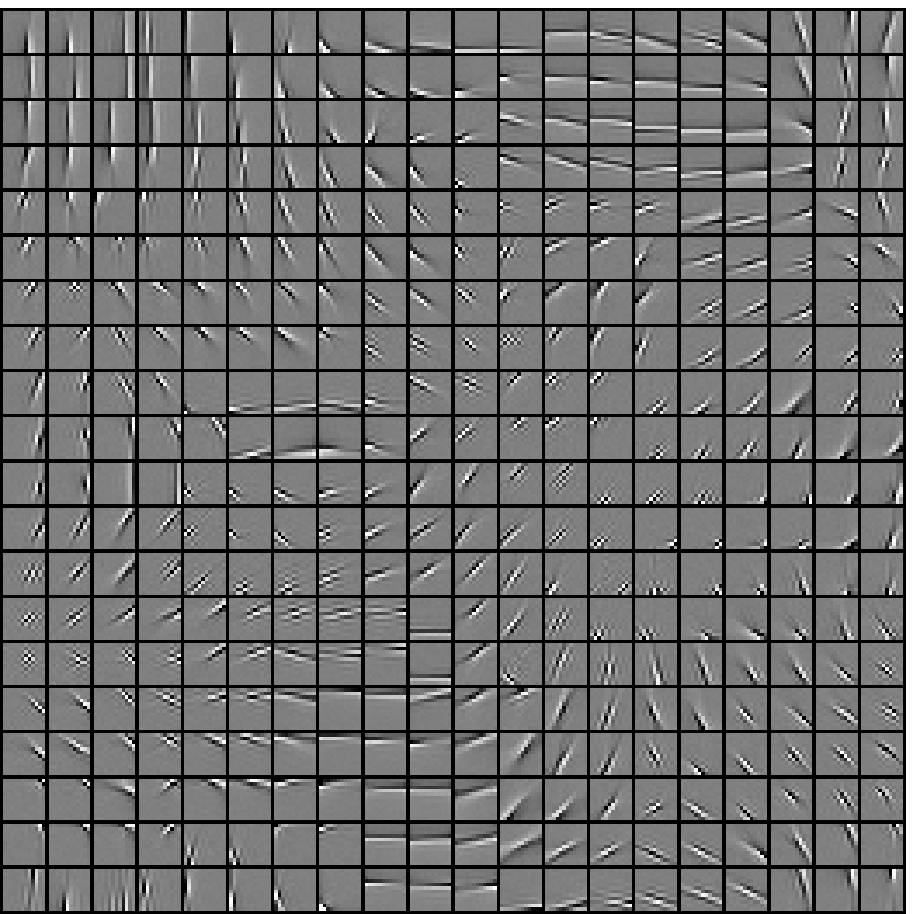} \hfill
  \includegraphics[width=0.48\linewidth]{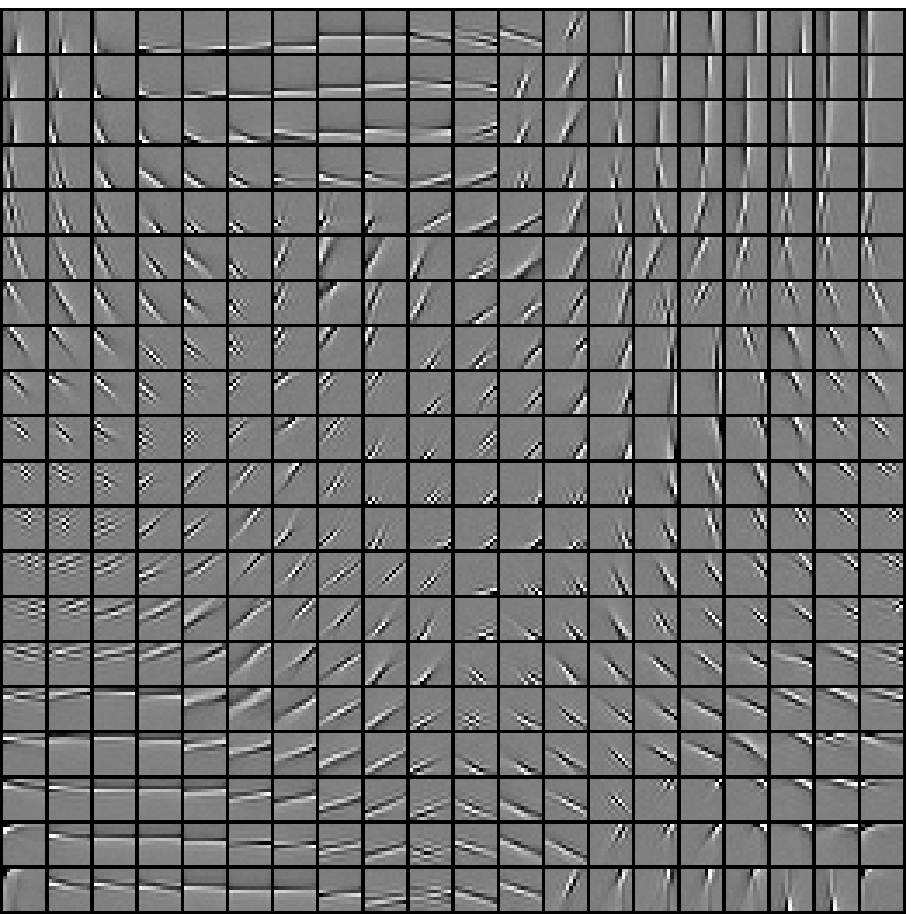} 
\caption{Topographic dictionaries with $400$ elements, learned on a database of $12 \times 12$ whitened natural image patches. Left: with $3 \times 3$ cyclic overlapping groups. Right: with $4 \times 4$ cyclic overlapping groups.}
\label{fig:topodict}
\end{figure}
 
\subsection{Multi-Task Learning of Hierarchical Structures}
As mentioned in the previous section, \citet{jenatton3} have recently proposed to use a
hierarchical structured norm to learn dictionaries of natural image patches.
In~\citet{jenatton3}, the dictionary elements are embedded in a
\textit{predefined} tree $\mathcal{T}$, 
via a particular instance of the
structured norm $\Omega$, which we refer to it as~$\Omega_{\text{tree}}$,
 and call $\GG$ the underlying set of groups.
In this case, using the same notation as in Section~\ref{subsec:topo}, each signal~$\y^i$ admits a sparse decomposition in the
form of a subtree of dictionary elements. 

Inspired by ideas from multi-task learning~\cite{obozinski}, we propose to
learn the tree structure $\mathcal{T}$ by pruning irrelevant parts of a larger
initial tree $\mathcal{T}_0$.  We achieve this by using an additional
regularization term~$\Omega_{\text{joint}}$ across the different decompositions,
so that subtrees of $\mathcal{T}_0$ will \textit{simultaneously} be removed for all signals $\y^i$.
With the notation from Section~\ref{subsec:topo}, the approach of~\citet{jenatton3} is then extended by the following formulation:
\begin{equation}
    \min_{\X\in\mathcal{C}, \W\in\Real^{p \times n}}
    \frac{1}{n}\sum_{i=1}^n\!\Big[\frac{1}{2} \|\y^i-\X\w^i\|_2^2 + \lambda_1 \Omega_{\text{tree}}(\w^i)\Big]\!+\!\lambda_2\Omega_{\text{joint}}(\W), \label{eq:dict_learning}
\end{equation}
where $\W \defin [\w^1,\ldots,\w^n]$ is the matrix of decomposition
coefficients in $\Real^{p \times n}$. The new regularization term operates
on the rows of $\W$ and is defined as $\Omega_{\text{joint}}(\W) \defin
\sum_{g\in\G}\max_{i\in\IntSet{n}}|\w_g^i|$.\footnote{The simplified case where 
$\Omega_{\text{tree}}$ and $\Omega_{\text{joint}}$ are the $\ell_1$-
and mixed $\ell_1/\ell_2$-norms~\cite{yuan} corresponds to~\citet{sprechmann}.}
The overall penalty on $\W$, which results from the combination of $\Omega_{\text{tree}}$ and
$\Omega_{\text{joint}}$, is itself an instance of~$\Omega$ with general overlapping groups, as defined in Eq~(\ref{eq:omega}).

To address problem~(\ref{eq:dict_learning}), we use the same optimization
scheme as \citet{jenatton3}, i.e., alternating between~$\X$ and~$\W$, fixing one
variable while optimizing with respect to the other. 
The task we consider is the denoising of natural image patches, with the same dataset and protocol as~\citet{jenatton3}.
We study whether learning the hierarchy of the dictionary elements improves the denoising performance, 
compared to standard sparse coding (i.e., when $\Omega_{\text{tree}}$ is the $\ell_1$-norm and $\lambda_2=0$) 
and the hierarchical dictionary learning of~\citet{jenatton3} based on predefined trees (i.e., $\lambda_2=0$).
The dimensions of the training set --- $50\,000$ patches of size $8\!\times\!8$ for dictionaries with up to $p=400$ elements --- 
impose to handle extremely large graphs, with $|E|\approx|V|\approx 4.10^7$. 
Since problem~(\ref{eq:dict_learning}) is too large to be solved exactly sufficiently many times to select the regularization parameters $(\lambda_1,\lambda_2)$ rigorously,
we use the following heuristics:
we optimize mostly with the currently pruned tree held fixed (i.e., $\lambda_2=0$), 
and only prune the tree (i.e., $\lambda_2>0$) every few steps on a random subset of $10\,000$ patches.
We consider the same hierarchies as in~\citet{jenatton3}, involving between $30$ and $400$ dictionary elements.
The regularization parameter~$\lambda_1$ is selected on the validation set of $25\,000$ patches, 
for both sparse coding (Flat) and hierarchical dictionary learning (Tree). 
Starting from the tree giving the best performance 
(in this case the largest one, see Figure~\ref{fig:tree}), 
we solve problem~(\ref{eq:dict_learning}) following our heuristics, for increasing values of~$\lambda_2$.
As shown in Figure~\ref{fig:tree}, there is a regime where our approach performs significantly better than the two other compared methods.
The standard deviation of the noise is $0.2$ (the pixels have values in $[0,1]$); no significant improvements were observed for lower levels of noise.
Our experiments use the algorithm of~\citet{beck} based on our proximal operator, with weights $\eta_g$ set to~$1$. We present this algorithm in more details in Appendix~\ref{appendix:fista}.

\begin{figure}[hbtp]
   \centering
   \includegraphics[width=0.5\linewidth]{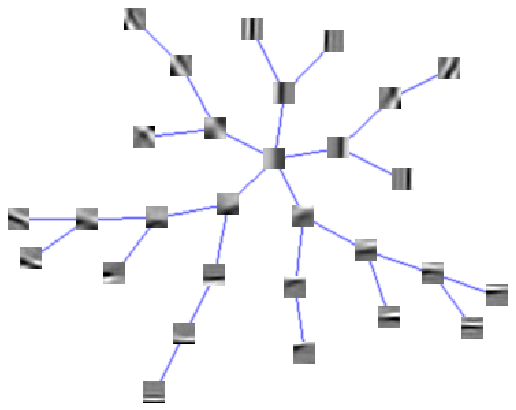}\hfill
   \includegraphics[width=0.5\linewidth]{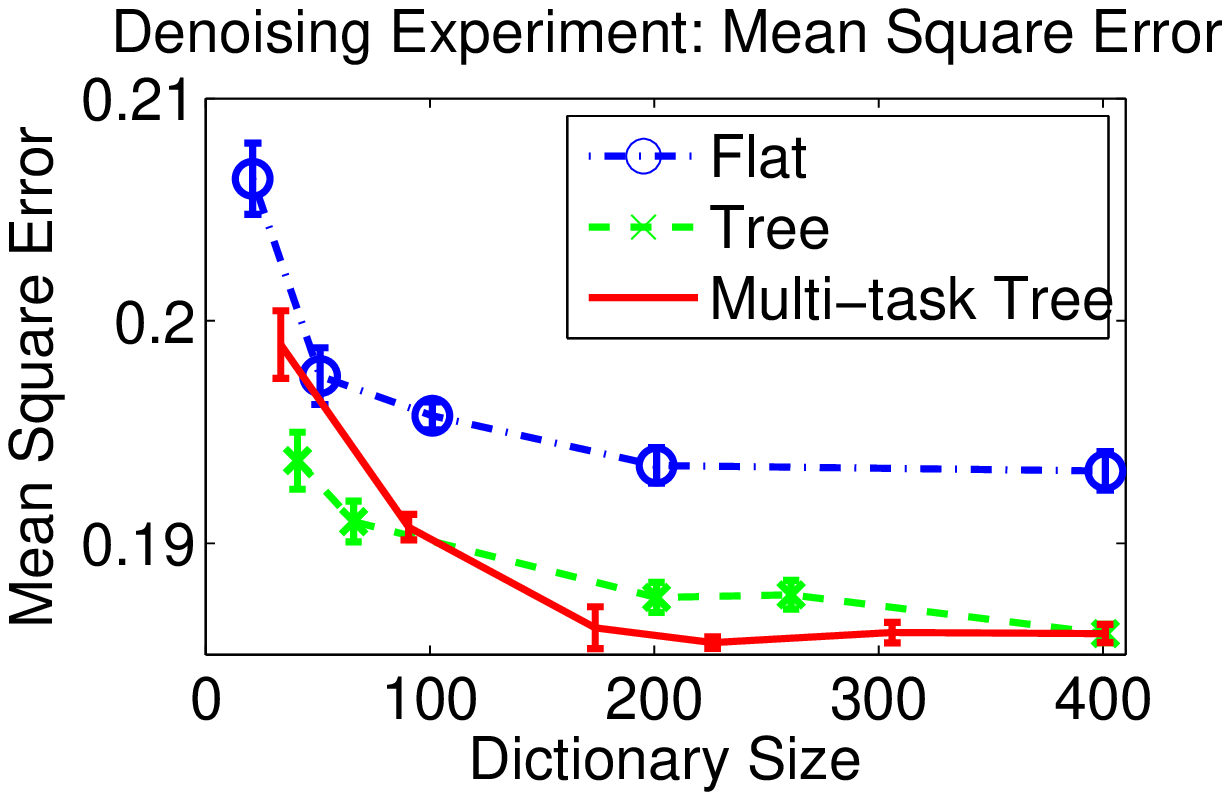} 
   \caption{Left: Hierarchy obtained by pruning a larger tree of $76$ elements. 
   Right: Mean square error versus dictionary size. 
   The error bars represent two standard deviations, based on three runs.} \label{fig:tree}
\end{figure}

\section{Conclusion} \label{sec:ccl}
We have presented new optimization methods for solving sparse structured
problems involving sums of $\ell_2$- or $\ell_\infty$-norms of any
(overlapping) groups of variables.  Interestingly, this sheds new light on
connections between sparse methods and the literature of network flow
optimization.  In particular, the proximal operator for the sum of $\ell_\infty$-norms 
can be cast as a specific form of quadratic min-cost flow problem, for which we proposed
an efficient and simple algorithm.

In addition to making it possible to resort to accelerated gradient methods, 
an efficient computation of the proximal operator offers more generally a certain modularity, 
in that it can be used as a building-block
for other optimization problems. 
A case in point is dictionary learning where
proximal problems come up and have to be solved repeatedly in an inner-loop.
Interesting future work includes the computation of other structured norms such as the one introduced in~\citet{jacob},
or total-variation based penalties, whose proximal operators are also based on minimum cost flow problems~\citep{chambolle}.
Several experiments demonstrate that our algorithm can be applied to
a wide class of learning problems, which have not been addressed before with
convex sparse methods.

\acks{This paper was partially supported by grants from the
Agence Nationale de la Recherche (MGA Project) and
from the European Research Council (SIERRA Project).
In addition, Julien Mairal was supported by the NSF grant SES-0835531 and NSF award CCF-0939370.
The authors would like to thank Jean Ponce for interesting discussions and
suggestions for improving this manuscript, Jean-Christophe Pesquet and Patrick-Louis Combettes for pointing us to the literature of proximal splitting methods, Ryota Tomioka for his advice on 
using augmented Lagrangian methods.}

\appendix

\section{Equivalence to Canonical Graphs}\label{appendix:equivalent}
Formally, the notion of equivalence between graphs can be summarized by
the following lemma:
\begin{lemma}[Equivalence to canonical graphs.] \label{lemma:equivalent}~\\
Let $G=(V,E,s,t)$ be the canonical graph corresponding to a group structure $\GG$.
Let $G'=(V,E',s,t)$ be a graph sharing the same set of
vertices, source and sink as $G$, but with a different arc set $E'$.
We say that $G'$ is equivalent to $G$ if and only if the following conditions hold:
\begin{itemize}
\item Arcs of $E'$ outgoing from the source are the same as in $E$, with the same costs and capacities.
\item Arcs of $E'$ going to the sink are the same as in $E$, with the same costs and capacities. 
\item For every arc $(g,j)$ in $E$, with $(g,j)$ in $V_{gr} \times V_u$, there exists
a unique path in $E'$ from $g$ to $j$ with zero costs and infinite capacities on every
arc of the path.
\item Conversely, if there exists a path in $E'$ between a vertex $g$ in $V_{gr}$
and a vertex $j$ in $V_u$, then there exists an arc $(g,j)$ in $E$.
\end{itemize}
Then, the cost of the optimal min-cost flow on $G$ and $G'$ are the same.
Moreover, the values of the optimal flow on the arcs $(j,t)$, $j$ in $V_u$, are
the same on $G$ and $G'$.
\end{lemma}
\begin{proof}
We first notice that on both $G$ and $G'$, the cost of a flow on the graph only
depends on the flow on the arcs $(j,t)$, $j$ in $V_u$, which we have denoted by
$\xibbar$ in $E$.

We will prove that finding a feasible flow $\pi$ on $G$ with a cost $c(\pi)$ is
equivalent to finding a feasible flow $\pi'$ on $G'$ with the same cost
$c(\pi)=c(\pi')$.  We now use the concept of \emph{path flow}, which is a flow
vector in $G$ carrying the same positive value on every arc of a directed path
between two nodes of $G$.  It intuitively corresponds to sending a positive
amount of flow along a path of the graph.

According to the definition of graph equivalence introduced in the Lemma, it is
easy to show that there is a bijection between the arcs in $E$, and the paths
in $E'$ with positive capacities on every arc.  Given now a feasible flow $\pi$
in $G$, we build a feasible flow $\pi'$ on $G'$ which is a \textit{sum} of path
flows.  More precisely, for every arc $a$ in $E$, we consider its equivalent
path in $E'$, with a path flow carrying the same amount of flow as $a$.
Therefore, each arc $a'$ in $E'$ has a total amount of flow that is equal to
the sum of the flows carried by the path flows going over $a'$.  It is also
easy to show that this construction builds a flow on $G'$ (capacity and
conservation constraints are satisfied) and that this flow $\pi'$ has the same
cost as $\pi$, that is, $c(\pi)=c(\pi')$.

Conversely, given a flow $\pi'$ on $G'$, we use a classical path flow
decomposition~\citep[see][Proposition 1.1]{bertsekas2}, saying that there
exists a decomposition of $\pi'$ as a sum of path flows in $E'$. Using the
bijection described above, we know that each path in the previous sums
corresponds to a unique arc in $E$. We now build a flow $\pi$ in $G$, by
associating to each path flow in the decomposition of $\pi'$, an arc in $E$
carrying the same amount of flow.  The flow of every other arc in $E$ is set to
zero.  It is also easy to show that this builds a valid flow in $G$ that has
the same cost as $\pi'$.
\end{proof}

\section{Convergence Analysis}\label{appendix:convergence}
We show in this section the correctness of Algorithm~\ref{algo:prox} for computing
the proximal operator, and of Algorithm~\ref{algo:dual_norm} for computing the dual
norm $\Omega^\star$.
\subsection{Computation of the Proximal Operator}
We first prove that our algorithm converges and that it finds the optimal solution of
the proximal problem. This requires that we introduce the optimality conditions for problem~(\ref{eq:dual_problem}) derived from~\citet{jenatton3,jenatton4} since our convergence proof essentially checks that these conditions are satisfied upon termination of the algorithm.

\begin{lemma}[Optimality conditions of the problem~(\ref{eq:dual_problem}) from \citealt{jenatton3,jenatton4}]\label{lemma:opt}~\newline
The primal-dual variables $(\w,\xib)$ are respectively solutions of the primal~(\ref{eq:prox_problem})
and dual problems~(\ref{eq:dual_problem}) if and only if 
the dual variable $\xib$ is feasible for the problem~(\ref{eq:dual_problem}) and
\begin{displaymath}
\begin{split}
    & {\textstyle \w = \u -\sum_{g \in \G} \xib^g},  \\
  &  \forall g \in \G,~~ \left\{ \begin{array}{l}
       \w_g^\top \xib_g^g = \|\w_g\|_\infty \|\xib^g\|_1  ~~\text{and}~~ \|\xib^g\|_1=\lambda \eta_g,  \\
       \text{or}~~ \w_g = 0. 
    \end{array} \right. 
 \end{split}
 \end{displaymath}
\end{lemma}

Note that these optimality conditions provide an intuitive view of our
min-cost flow problem. Solving the min-cost flow problem is equivalent to
sending the maximum amount of flow in the graph under the capacity constraints,
while respecting the rule that \emph{the flow coming from a group $g$ should
always be directed to the variables $\u_j$ with maximum residual $\u_j-\sum_{g
\in \G}\xib^g_j$}.
This point can be more formaly seen by noticing that one of the optimality conditions above 
corresponds to the case of equality in the $\ell_1/\ell_\infty$ H\"older inequality.

Before proving the convergence and correctness of our algorithm, we also recall
classical properties of the min capacity cuts, which we
intensively use in the proofs of this paper.  The procedure
\texttt{computeFlow} of our algorithm finds a minimum $(s,t)$-cut of a graph
$G=(V,E,s,t)$, dividing the set $V$ into two disjoint parts $V^+$ and $V^-$.
$V^+$ is by construction the sets of nodes in $V$ such that there exists a
non-saturating path from $s$ to $V$, while all the paths from $s$ to $V^-$ are
saturated. Conversely, arcs from $V^+$ to $t$ are all saturated, whereas there
can be non-saturated arcs from $V^-$ to $t$. Moreover, the following properties,
which are illustrated on Figure~\ref{fig:graph2}, hold
\begin{itemize}
\item There is no arc going from $V^+$ to $V^-$. Otherwise the value of the cut would
be infinite (arcs inside $V$ have infinite capacity by construction of our graph). 
\item There is no flow going from $V^-$ to $V^+$~\citep[see][]{bertsekas2}.
\item The cut goes through all arcs going from $V^+$ to $t$, and all arcs going from $s$ to $V^-$.
\end{itemize}
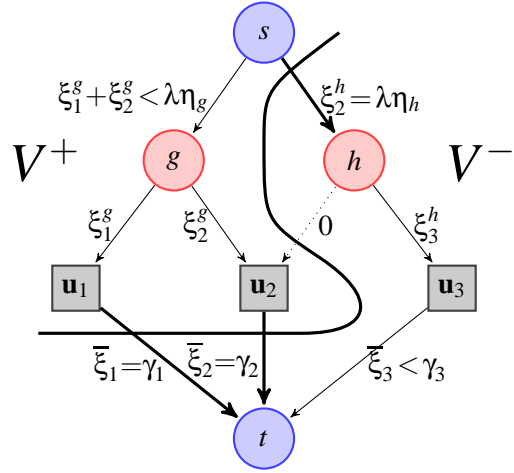
\begin{figure}[hbtp]
\tikzstyle{source}=[circle,thick,draw=blue!75,fill=blue!20,minimum size=8mm]
\tikzstyle{sink}=[circle,thick,draw=blue!75,fill=blue!20,minimum size=8mm]
\tikzstyle{group}=[place,thick,draw=red!75,fill=red!20, minimum size=8mm]
\tikzstyle{var}=[rectangle,thick,draw=black!75,fill=black!20,minimum size=6mm]
\tikzstyle{part}=[minimum size=6mm]
\tikzstyle{every label}=[red]
   \begin{center}
      \begin{tikzpicture}[node distance=1.7cm,>=stealth',bend angle=45,auto]
         \begin{scope}
            \node [source]   (s)                                    {$s$};
            \node [group]    (g)  [below of=s,xshift=-12mm]                      {$g$}
            edge  [pre] node[left] {$\xib^g_1 \!+ \!\xib^g_2 \! < \!\lambda \eta_g$} (s);
            \node [part] (p) [left of=g] {{\huge $V^+$}};
            \node [group]    (h)  [below of=s,xshift=12mm]                      {$h$}
            edge  [pre,very thick] node[right] {$\xib^h_2 \! = \!\lambda \eta_h$} (s);
            \node [part] (p2) [right of=h] {{\huge $V^-$}};
            \node [var] (u2) [below of=g,xshift=12mm]                    {$\u_2$}
            edge  [pre,dotted] node[above, right] {$0$} (h)
            edge  [pre] node[above, left] {$\xib^{g}_2$} (g);
            \node [var] (u1)  [left of=u2, node distance=2.5cm] {$\u_1$}
            edge  [pre] node[above, left] {$\xib^{g}_1$} (g);
            \node [var] (u3) [right of=u2, node distance=2.5cm] {$\u_3$}
            edge  [pre] node[above, right] {$\xib^{h}_3$} (h);
            \node [sink] (si) [below of=u2,node distance=2cm] {$t$}
            edge [pre,very thick] node[above,left,xshift=1mm] {$\xibbar_1 \!\! =\!\! \gammab_1$} (u1)
            edge [pre,very thick] node[above,left,xshift=1mm] {$\xibbar_2\!\! = \!\! \gammab_2$} (u2)
            edge [pre] node[above,right] {$\xibbar_3 \! < \! \gammab_3$} (u3);
            \draw [very thick] (1,0) .. controls (0,-0.75) .. (0,-1.5) .. controls (0,-2.5) .. (0.75,-3) .. controls (1.5,-3.5) and (1.5,-4) .. (0.5,-4) -- (-3,-4);
         \end{scope}
      \end{tikzpicture}
   \end{center}
   \caption{Cut computed by our algorithm. $V^+\! = \! V_u^+ \cup V_{gr}^+$, with $V_{gr}^+\! =\! \{g\}$, $V_{u}^+\! =\! \{ 1,2 \}$,
and $V^- \! =\! V_u^- \cup V_{gr}^-$, with $V_{gr}^-\! =\! \{h\}$, $V_{u}^- \! =\! \{ 3 \}$. Arcs going from $s$ to $V^-$ are saturated, as well as arcs going from $V^+$ to $t$. Saturated arcs are in bold. Arcs with zero flow are dotted.} \label{fig:graph2}
\end{figure}

Recall that we assume (cf. Section~\ref{subsec:graph})
that the scalars $\u_j$ are all non negative, and that we add non-negativity
constraints on~$\xib$.  With the optimality conditions of Lemma~\ref{lemma:opt}
in hand, we can show our first convergence result.
\begin{proposition}[Convergence of Algorithm \ref{algo:prox}] \label{prop:convergence}~\\
Algorithm~\ref{algo:prox} converges in a finite and polynomial number of operations.
\end{proposition}
\begin{proof}
Our algorithm splits recursively the graph into disjoints parts and processes
each part recursively.  The processing of one part requires an orthogonal
projection onto an $\ell_1$-ball and a max-flow algorithm, which can both be
computed in polynomial time.  To prove that the procedure converges, it is
sufficient to show that when the procedure \texttt{computeFlow} is called
for a graph $(V,E,s,t)$ and computes a cut $(V^+, V^-)$, then the components
$V^+$ and $V^-$ are both non-empty.

Suppose for instance that $V^- \!\!= \emptyset$.  In this case, the capacity of
the min-cut is equal to $\sum_{j \in V_u} \gammab_j$, and the value of the
max-flow is $\sum_{j \in V_u} \xibbar_j$. Using the classical max-flow/min-cut
theorem~\cite{ford}, we have equality between these two terms.  Since, by
definition of both $\gammab$ and $\xibbar$, we have for all~$j$ in~$V_u$,
$\xibbar_j \leq \gammab_j$, we obtain a contradiction with the existence of $j$
in $V_u$ such that $\xibbar_j \neq \gammab_j$. 

Conversely, suppose now that $V^+ \!\!= \emptyset$.
Then, the value of the max-flow is still $\sum_{j \in V_u} \xibbar_j$, and the value
of the min-cut is $\lambda\sum_{g \in V_{gr}} \eta_g$. Using again the
max-flow/min-cut theorem, we have that $\sum_{j \in V_u} \xibbar_j = \lambda\sum_{g
\in V_{gr}} \eta_g$.  Moreover, by definition of $\gammab$, we also have
$\sum_{j \in V_u} \xibbar_j \leq \sum_{j \in V_u} \gammab_j \leq \lambda\sum_{g
\in V_{gr}} \eta_g$, leading to a contradiction with the existence of $j$ in
$V_u$ satisfying $\xibbar_j \neq \gammab_j$. 
We remind the reader of the fact that such a $j\in V_u$ exists 
since the cut is only computed when the current estimate $\xibbar$ is not optimal yet.
This proof holds for any graph that is equivalent to the canonical one.
\end{proof}
After proving the convergence, we prove that the algorithm is correct with the next proposition.

\begin{proposition}[Correctness of Algorithm \ref{algo:prox}] ~\\
 Algorithm~\ref{algo:prox} solves the proximal problem of Eq.~(\ref{eq:prox_problem}).
\end{proposition}
\begin{proof}
For a group structure $\GG$, we first prove the correctness of our algorithm if the graph used is its associated canonical graph that we denote $G_0=(V_0,E_0,s,t)$.
We proceed by induction on the number of nodes of the graph.
The induction hypothesis $\H(k)$ is the following:~\vspace*{0.2cm}\\
\textit{For all canonical graphs $G=(V = V_u \cup V_{gr},E,s,t)$ associated with a group structure $\GG_V$ with weights $(\eta_g)_{g \in \GG_V}$ such that $|V|\leq k$, \texttt{computeFlow}$(V,E)$
solves the following optimization problem:}
\begin{equation}
\min_{(\xib_j^g)_{j \in V_u,g \in V_{gr}}} \sum_{j \in V_u} \frac{1}{2} (\u_j - \sum_{g \in V_{gr}}
\xib_j^g)^2 \st \forall g \in V_{gr},~ \sum_{j \in V_u} \xib_j^g \leq \lambda \eta_g
~~\text{and}~~ \xib_j^g=0,~\forall j \notin g. \label{eq:dual2}
\end{equation}
Since $\GG_{V_0}=\GG$, it is sufficient to show that $\H(|V_0|)$ to prove the proposition.

We initialize the induction by $\H(2)$, corresponding to the simplest canonical graph, for which $|V_{gr}|=|V_u|=1$).
Simple algebra shows that $\H(2)$ is indeed correct.

We now suppose that $\H(k')$ is true for all $k' < k$ and consider a graph
$G=(V,E,s,t)$, $|V|=k$.  The first step of the algorithm computes the variable
$(\gammab_j)_{j\in V_u}$ by a projection on the $\ell_1$-ball. This is itself
an instance of the dual formulation of Eq.~(\ref{eq:dual_problem}) in a
simple case, with one group containing all variables.  We can therefore use
Lemma~\ref{lemma:opt} to characterize the optimality of
$(\gammab_j)_{j\in V_u}$, which yields 
\begin{equation}
    \!\!\left\{ \begin{array}{l}
       \sum_{j\in V_u} (\u_j-\gammab_j)\gammab_j = \big(\max_{j \in V_u} |\u_j-\gammab_j|\big)
  \sum_{j\in V_u} \gammab_j ~\text{and}~ \sum_{j\in V_u} \gammab_j=\lambda \sum_{g \in V_{gr}} \eta_g,\!\!\!\!
\\
       \text{or}~~ \u_j - \gammab_j = 0,~\forall j \in V_u. \label{eq:opt_gamma}
    \end{array} \right. 
\end{equation}
The algorithm then computes a max-flow, using the scalars $\gammab_j$ as
capacities, and we now have two possible situations:\\
\begin{enumerate}
\item If $\xibbar_j = \gammab_j$ for
all $j$ in $V_u$, the algorithm stops; we write $\w_j = \u_j-\xibbar_j$ for $j$ in $V_u$, and using Eq.~(\ref{eq:opt_gamma}), we obtain
\begin{equation}
    \left\{ \begin{array}{l}
       \sum_{j\in V_u} \w_j\xibbar_j = (\max_{j \in V_u} |\w_j|)
  \sum_{j\in V_u} \xibbar_j ~~\text{and}~~ \sum_{j\in V_u} \xibbar_j=\lambda \sum_{g \in V_{gr}} \eta_g,
\\
       \text{or}~~ \w_j = 0,~\forall j \in V_u.
    \end{array} \right. 
\end{equation}
We can rewrite the condition above as 
$$
\sum_{g\in V_{gr}}\sum_{j\in g}\! \w_j\xib_j^g = \sum_{g\in V_{gr}}\! ( \max_{j \in V_u} |\w_j| )\!\!\! \sum_{j\in V_u}\xib_j^g.
$$
Since all the quantities in the previous sum are positive, this can only hold if for all $g\in V_{gr}$,
$$
\sum_{j\in V_u}\w_j\xib_j^g = ( \max_{j \in V_u} |\w_j| )\! \sum_{j\in V_u}\xib_j^g.
$$
Moreover, by definition of the max flow and the optimality conditions, we have
$$
\forall g\in V_{gr},\ \sum_{j\in V_u}\! \xib^g_j \leq \lambda \eta_g,\ \text{and}\ \sum_{j\in V_u} \xibbar_j=\lambda \sum_{g \in V_{gr}} \eta_g,
$$
which leads to
$$
\forall g\in V_{gr}, \sum_{j\in V_u}\! \xib^g_j  = \lambda \eta_g.
$$
By Lemma~\ref{lemma:opt}, we have shown that the problem~(\ref{eq:dual2}) is solved.

\item Let us now consider the case where there exists $j$ in $V_u$ such that
$\xibbar_j \neq \gammab_j$. The algorithm splits the vertex set $V$
into two parts $V^+$ and $V^-$, which we have proven to be non-empty in the proof
of Proposition~\ref{prop:convergence}. The next step of the algorithm removes all edges between $V^+$ and $V^-$ (see Figure~\ref{fig:graph2}).
Processing $(V^+,E^+)$ and $(V^-,E^-)$ independently, it
updates the value of the flow matrix $\xib^g_j,\ j\in V_u,\ g\in V_{gr}$, 
and the corresponding flow vector~$\xibbar_j,\ j\in V_u$.
As for $V$, we denote by $V^+_u \defin V^+ \cap V_u$, $V^-_u \defin V^- \cap V_u$ and
$V^+_{gr} \defin V^+ \cap V_{gr}$, $V^-_{gr} \defin V^- \cap V_{gr}$.

Then, we notice that $(V^+,E^+,s,t)$ and $(V^-,E^-,s,t)$ are respective canonical graphs for the group
structures $\GG_{V^+} \defin \{ g \cap V_u^+ \mid g \in V_{gr} \}$, and $\GG_{V^-} \defin \{ g \cap V_u^- \mid g \in V_{gr} \}$.

Writing $\w_j = \u_j-\xibbar_j$ for $j$ in $V_u$, and using the induction hypotheses
$\H(|V^+|)$ and $\H(|V^-|)$, we now have the following optimality conditions deriving
from Lemma~\ref{lemma:opt} applied on Eq.~(\ref{eq:dual2}) respectively for the graphs
$(V^+,E^+)$ and $(V^-,E^-)$:
\begin{equation}
 \forall g \in V_{gr}^+,  g' \defin g \cap V_u^+,~
    \left\{ \begin{array}{l}
       \w_{g'}^\top \xib_{g'}^g = \|\w_{g'}\|_\infty 
 \sum_{j\in g'}\! \xib^g_j ~~\text{and}~~ \sum_{j\in g'}\! \xib^g_j=\lambda \eta_g,
\\
       \text{or}~~ \w_{g'} = 0, \\
    \end{array} \right. \label{eq:hplus}
\end{equation}
and
\begin{equation}
\forall g \in V_{gr}^-, g' \defin g \cap V_u^-,  
    \left\{ \begin{array}{l}
       \w_{g'}^\top \xib_{g'}^g = \|\w_{g'}\|_\infty 
\sum_{j\in g'}\! \xib^g_j  ~~\text{and}~~ \sum_{j\in g'}\! \xib^g_j=\lambda \eta_g,
\\
       \text{or}~~ \w_{g'} = 0. \\
    \end{array} \right. 
\label{eq:hmoins}
\end{equation}
We will now combine Eq.~(\ref{eq:hplus}) and Eq.~(\ref{eq:hmoins}) into optimality conditions
for Eq.~(\ref{eq:dual2}).  We first notice that $g \cap V_u^+ = g$ since there
are no arcs between $V^+$ and $V^-$ in $E$ (see the properties of the cuts
discussed before this proposition).
It is therefore possible to replace $g'$ by $g$ in Eq.~(\ref{eq:hplus}). 
We will show that it is possible to do the same in Eq.~(\ref{eq:hmoins}), so that combining these two equations yield the optimality conditions of Eq.~(\ref{eq:dual2}).

More precisely, we will show that for all $g \in V_{gr}^-$ and $j \in
g \cap V_u^+$, $|\w_j| \leq \max_{l \in g \cap V_u^-} |\w_l|$, in which case $g'$ can be replaced by $g$ in Eq.~(\ref{eq:hmoins}).
This result is relatively intuitive: $(s,V^+)$ and $(V^-,t)$ being an $(s,t)$-cut,
all arcs between $s$ and $V^-$ are saturated, while there are unsaturated arcs
between $s$ and $V^+$; one therefore expects the residuals $\u_j -\xibbar_j$
to decrease on the $V^+$ side, while increasing on the $V^-$ side.
The proof is nonetheless a bit technical.

Let us show first that for all $g$ in $V_{gr}^+$, $\NormInf{\w_g} \leq \max_{j\in V_u}|\u_j-\gammab_j|$.
We split the set $V^+$ into disjoint parts:
\begin{displaymath}
\begin{split}
V_{gr}^{++} &\defin \{ g \in V_{gr}^+ \st \NormInf{\w_g} \leq \max_{j\in V_u}|\u_j-\gammab_j| \}, \\
V_{u}^{++} &\defin \{ j \in V_{u}^+ \st \exists g \in V_{gr}^{++},~ j \in g \}, \\
V_{gr}^{+-} &\defin V_{gr}^+ \setminus V_{gr}^{++} = \{ g \in V_{gr}^+ \st \NormInf{\w_g} > \max_{j\in V_u}|\u_j-\gammab_j| \}, \\
V_{u}^{+-} &\defin V_{u}^+ \setminus V_{u}^{++}. \\
\end{split}
\end{displaymath}
As previously, we denote $V^{+-}\! \defin V_u^{+-}\! \cup V_{gr}^{+-}$ and $V^{++} \defin\!\!
V_u^{++} \cup V_{gr}^{++}$.  We want to show that $V_{gr}^{+-}$ is necessarily empty.
We reason by contradiction and assume that $V_{gr}^{+-} \neq \varnothing$.

According to the definition of the different sets above, we observe that no
arcs are going from $V^{++}$ to $V^{+-}$, that is, for all $g$ in
$V_{gr}^{++}$, $g \cap V_{u}^{+-}=\varnothing$.  We observe as well that the
flow from $V_{gr}^{+-}$ to $V_u^{++}$ is the null flow, because optimality
conditions (\ref{eq:hplus}) imply that for a group $g$ only nodes $j \in g$
such that $\w_j=\|\w_g\|_\infty$ receive some flow, which excludes nodes in
$V_u^{++}$ provided $V_{gr}^{+-} \neq \varnothing$;
Combining this fact and the inequality 
$
\sum_{g \in V_{gr}^{+}} \lambda \eta_g
\geq \sum_{j \in V_u^{+}} \gammab_j
$ (which is a direct consequence of the minimum $(s,t)$-cut),
we have as well 
$$
\sum_{g \in V_{gr}^{+-}} \lambda \eta_g \geq \sum_{j \in V_u^{+-}} \gammab_j.
$$

Let $j \in V_u^{+-}$, if $\xibbar_j \neq 0$ then for some $g \in V_{gr}^{+-}$ such that $j$ receives
some flow from $g$, which from the optimality conditions~(\ref{eq:hplus}) implies $\w_j=\|\w_g\|_\infty$;
by definition of $V_{gr}^{+-}$, $\|\w_g\|_\infty > \u_j-\gammab_j$. But since at the optimum, $\w_j=\u_j-\xibbar_j$,
this implies that $\xibbar_j < \gammab_j$, and in turn that $\sum_{j\in V_u^{+-}} \xibbar_j=\lambda \sum_{g \in V_{gr}^{+-}} \eta_g$.
Finally, $$\lambda \sum_{g \in V_{gr}^{+-}} \eta_g=\sum_{j\in V_u^{+-},\, \xibbar_j \neq 0} \xibbar_j < \sum_{j\in V_u^{+-}} \gammab_j$$ and this is a contradiction.

We now have that for all $g$ in $V_{gr}^+$, $\NormInf{\w_g} \leq \max_{j\in V_u}|\u_j-\gammab_j| $.
The proof showing that for all $g$ in $V_{gr}^-$, 
$
\NormInf{\w_g} \geq \max_{j\in V_u}|\u_j-\gammab_j|,
$ 
uses the same kind of decomposition for $V^-$, and follows along similar arguments. We will therefore not detail it.

To summarize, we have shown that for all $g \in V_{gr}^-$ and $j \in
g \cap V_u^+$, $|\w_j| \leq \max_{l \in g \cap V_u^-} |\w_l|$.
Since there is no flow from $V^-$ to $V^+$, 
i.e., $\xib_j^g=0$ for $g$ in $V_{gr}^-$ and $j$ in $V_{u}^+$,
we can now replace the definition of $g'$ in Eq.~(\ref{eq:hmoins}) by $g'
\defin g \cap V_u$, the combination of Eq.~(\ref{eq:hplus}) and
Eq.~(\ref{eq:hmoins}) gives us optimality conditions for Eq.~(\ref{eq:dual2}).
\end{enumerate}

The proposition being proved for the canonical graph, we extend it now
for an equivalent graph in the sense of Lemma~\ref{lemma:equivalent}.
First, we observe that the algorithm gives the same values of $\gammab$
for two equivalent graphs. Then, it is easy to see that the value $\xibbar$
given by the max-flow, and the chosen $(s,t)$-cut is the same, which is
enough to conclude that the algorithm performs the same steps for
two equivalent graphs.
\end{proof}

\subsection{Computation of the Dual Norm $\Omega^\star$}
As for the proximal operator, the computation of dual norm $\Omega^*$ can itself be shown to solve another network flow problem, based on the following variational formulation, which extends a previous result from \citet{jenatton}:
\begin{lemma}[Dual formulation of the dual-norm $\Omega^\star$.] ~\\ Let $\kappab \in \R{p}$. We have
\begin{displaymath}
\Omega^*(\kappab) = \!\!\! \min_{\xib\in\RR{p}{|\G|},\tau \in \R{}} \!\!\! \tau \quad \text{s.t.}\quad \sum_{g\in\G}\xib^g=\kappab,\ \text{and}\ \forall g\in\G,\ \|\xib^g\|_1 \leq \tau\eta_g ~~\text{with}~~\
\xib_j^g=0\ \text{if}\ j\notin g. 
\end{displaymath}
\end{lemma}

\begin{proof} 
By definition of $\Omega^*(\kappab)$, we have 
$$\Omega^*(\kappab)\defin\max_{\Omega(\z)\leq 1}\z^\top\kappab.$$
By introducing the primal variables $(\alpha_g)_{g\in\GG} \in \R{|\GG|}$, we can rewrite the previous maximization problem as
$$
 \Omega^*(\kappab) = \max_{ \sum_{g\in\GG} \!\eta_g\alpha_g \leq 1 } \kappab^\top \z,\quad \mbox{ s.t. } \quad \forall\ g\in\GG,\ \|\z_g\|_\infty \leq \alpha_g,
$$
with the additional $|\GG|$ conic constraints $ \|\z_g\|_\infty \leq \alpha_g$. 
This primal problem is convex and satisfies Slater's conditions for generalized conic inequalities, which implies that strong duality holds \citep{boyd}.
We now consider the Lagrangian $\mathcal{L}$ defined as
$$
\mathcal{L}(\z,\alpha_g,\tau, \gamma_g,\xib) 
= 
\kappab^\top \z + \tau (1 \! - \!\! \sum_{g\in\GG} \! \eta_g \alpha_g ) + 
\sum_{g\in\GG} \binom{\alpha_g}{\z_g}^\top \! \binom{\gamma_g}{\xib^g_g}, 
$$ 
with the dual variables 
$\{\tau, (\gamma_g)_{g\in\GG},\xib\} \in \R{}_+ \! \times \! \R{|\GG|} \! \times \!\RR{p}{|\GG|}$
such that for all $g\in\GG$, 
$\xib^g_j = 0 \, \mbox{ if } \, j \notin g$
and 
$ \|\xib^g\|_1 \leq \gamma_g $.
The dual function is obtained by taking the derivatives of $\mathcal{L}$ with respect to the primal variables $\z$ and $(\alpha_g)_{g\in\GG}$ and equating them to zero,
which leads to 
\begin{eqnarray*}
	\forall j\in \{1,\dots,p\},&      \kappab_j + \!\! \sum_{ g \in \GG } \xib^g_j & =  0 \\
	\forall g \in \GG,& \tau \eta_g -  \gamma_g & = 0.
\end{eqnarray*}
After simplifying the Lagrangian and flipping the sign of $\xib$, the dual problem then reduces to
$$
	\min_{ \xib\in\RR{p}{|\G|},\tau \in \R{} }  \tau  \quad \mbox{ s.t. }
	\begin{cases}
	 \forall j \in \{1,\dots,p\}, \kappab_j = \!\! \sum_{ g \in \GG } \xib^g_j\,\mbox{ and }\, \xib^g_j = 0 \, \mbox{ if } \, j \notin g,\\
	 \forall g \in \GG, \|\xib^g\|_1 \leq \tau\eta_g,
	\end{cases}
$$
which is the desired result.
\end{proof}
We now prove that Algorithm~\ref{algo:dual_norm} is correct.
\begin{proposition}[Convergence and correctness of Algorithm \ref{algo:dual_norm}] ~\\
 Algorithm~\ref{algo:dual_norm} computes the value of the dual norm of Eq.~(\ref{eq:dual_norm}) in a finite and polynomial number of operations.
\end{proposition}
\begin{proof}
The convergence of the algorithm only requires to show that the
cardinality of $V$ in the different calls of the function
\texttt{computeFlow} strictly decreases. Similar arguments
to those used in the proof of Proposition~\ref{prop:convergence} can
show that each part of the cuts $(V^+,V^-)$ are
both non-empty. The algorithm thus requires a finite number of
calls to a max-flow algorithm and converges in a finite
and polynomial number of operations.

Let us now prove that the algorithm is correct for a canonical graph. We
proceed again by induction on the number of nodes of the graph.  
More precisely, we consider the
induction hypothesis $\H'(k)$
defined as:~\vspace*{0.2cm}\\
\textit{for all canonical graphs} $G=(V,E,s,t)$ associated with a group structure $\GG_V$ and \textit{such that} $|V| \leq k$, \texttt{dualNormAux}$(V=V_u \cup V_{gr},E)$
\textit{solves the following optimization problem:}
\begin{equation}
 \min_{\xib,\tau} \tau \quad \text{s.t.}\quad \forall j \in V_u, \kappab_j = \sum_{g\in V_{gr}}\xib_j^g,\ \text{and}\ \forall g\in V_{gr},\ \sum_{j \in V_u} \xib^g_j \leq \tau\eta_g ~~\text{with}~~\
\xib_j^g=0\ \text{if}\ j\notin g. \label{eq:Hprime}
\end{equation}
We first initialize the induction by $\H(2)$ (i.e., with the simplest canonical graph, such that $|V_{gr}|=|V_u|=1$).
Simple algebra shows that $\H(2)$ is indeed correct.

We next consider a canonical graph $G=(V,E,s,t)$ such that $|V|=k$, and suppose that $\H'(k-1)$ is true.
After the max-flow step, we have two possible cases to discuss:
\begin{enumerate}
\item If $\xibbar_j = \gammab_j$ for all $j$ in $V_u$, the algorithm stops.
We know that any scalar $\tau$ such that the constraints of
Eq.~(\ref{eq:Hprime}) are all satisfied necessarily verifies $\sum_{g \in
V_{gr}} \tau \eta_g \geq \sum_{j \in V_{u}}\kappab_j$. We have indeed that
$\sum_{g \in V_{gr}} \tau \eta_g$ is the value of an $(s,t)$-cut in the graph,
and $\sum_{j \in V_{u}}\kappab_j$ is the value of the max-flow, and the
inequality follows from the max-flow/min-cut theorem~\cite{ford}. This gives a
lower-bound on $\tau$. Since this bound is reached, $\tau$ is necessarily
optimal.
\item We now consider the case where there exists $j$ in $V_u$ such that
$\xibbar_j \neq \kappab_j$, meaning that for the given value of $\tau$,
the constraint set of Eq.~(\ref{eq:Hprime}) is not feasible for $\xib$,
and that the value of $\tau$ should necessarily increase.
The algorithm splits the vertex set $V$ into two non-empty parts $V^+$ and
$V^-$ and we remark that there are no arcs going from $V^+$ to $V^-$, and no
flow going from $V^-$ to $V^+$. 
 Since the arcs going from $s$ to $V^-$ are 
 saturated, we have that $\sum_{g \in V_{gr}^-} \tau \eta_g \leq \sum_{j
 \in V_{u}^-}\kappab_j$.
 Let us now consider $\tau^\star$ the solution of Eq.~(\ref{eq:Hprime}).
 Using the induction hypothesis $\H'(|V^-|)$, the algorithm computes a new value
 $\tau'$ that solves Eq.~(\ref{eq:Hprime}) when replacing $V$ by $V^-$ and
 this new value satisfies the following inequality $\sum_{g \in V_{gr}^-} \tau'
 \eta_g \geq \sum_{j \in V_{u}^-}\kappab_j$. The value of $\tau'$ has therefore
 increased and the updated flow $\xib$ now satisfies the constraints of
 Eq.~(\ref{eq:Hprime}) and therefore $\tau'\geq \tau^\star$.
 Since there are no arcs going from $V^+$ to $V^-$, $\tau^\star$ is feasible
for Eq.~(\ref{eq:Hprime}) when replacing $V$ by $V^-$ and we have that
$\tau^\star \geq  \tau'$ and then $\tau'=\tau^\star$.
\end{enumerate}
To prove that the result holds for any equivalent graph, similar arguments to those used in
the proof of Proposition~\ref{prop:convergence} can be exploited, 
showing that the algorithm
computes the same values of $\tau$ and same $(s,t)$-cuts at each step.
\end{proof}

\section{Algorithm FISTA with duality gap} \label{appendix:fista}

In this section, we describe in details the algorithm FISTA \cite{beck} when
applied to solve problem~(\ref{eq:formulation}), with a duality gap as the stopping
criterion. The algorithm, as implemented in the experiments, is summarized in~Algorithm~\ref{alg:fista}.

Without loss of generality, let us assume we are looking for models of the form
$\X\w$, for some matrix $\X \in \RR{n}{p}$ (typically, a linear model where $\X$
is the design matrix composed of $n$ observations in $\Real^p$).  
Thus, we can consider the following primal problem 
\begin{equation}
   \min_{\w \in \Real^p} f(\X\w) + \lambda \Omega(\w),\label{eq:formulation_with_X}
\end{equation}
in place of problem~(\ref{eq:formulation}).  Based on Fenchel duality arguments
\cite{borwein},
$$
 f(\X\w)+ \lambda \Omega(\w) + f^*(-\kappab),\ \text{for}\ \w\in\R{p},\kappab\in\R{n}\ \text{and}\
\Omega^*(\X^\top\kappab) \leq \lambda,
$$
is a duality gap for problem~(\ref{eq:formulation_with_X}), where
$f^*(\kappab)\defin\sup_{\z} [\z^\top\kappab - f(\z)]$ is the Fenchel conjugate
of $f$~\citep{borwein}.
Given a primal variable $\w$, a good dual candidate $\kappab$ can be
obtained by looking at the conditions that have to be satisfied by the pair
$(\w,\kappab)$ at optimality~\cite{borwein}.  In particular, the dual variable
$\kappab$ is chosen to be  
$$
\kappab = -\rho^{-1} \nabla\! f(\X\w),\ \text{with}\ \rho\defin \max\big\{\lambda^{-1}\Omega^*(\X^\top\! \nabla\! f(\X\w)),1\big\}.
$$
Consequently, computing the duality gap requires evaluating the dual norm $\Omega^*$, 
as explained in Algorithm~\ref{algo:dual_norm}. 
We sum up the computation of the duality gap in~Algorithm~\ref{alg:fista}.
Moreover, we refer to the proximal operator associated with $\lambda \Omega$ as
$\text{prox}_{\lambda\Omega}$.\footnote{As a brief reminder, it is defined as the
function that maps the vector~$\u$ in~$\R{p}$ to the (unique, by strong
convexity) solution of Eq.~(\ref{eq:prox_problem}).}
\begin{algorithm}[hbtp]
\caption{FISTA procedure to solve problem~(\ref{eq:formulation_with_X}).}\label{alg:fista}
\begin{algorithmic}[1]
\STATE \textbf{Inputs}: initial $\w_{(0)} \in \R{p}$, $\Omega$, $\lambda>0$, $\varepsilon_{\text{gap}}>0$ (precision for the duality gap).
\STATE \textbf{Parameters}: $\nu > 1$, $L_{0} > 0$.
\STATE \textbf{Outputs}: solution $\w$.
\STATE \textbf{Initialization}: $\y_{(1)}=\w_{(0)}$, $t_1=1$, $k=1$.
\WHILE{ $\big\{$ \texttt{computeDualityGap}$\big(\w_{(k-1)}\big) > \varepsilon_{\text{gap}} \big\}$ }
\STATE Find the smallest integer $s_k\! \geq\! 0$ such that

\STATE $\quad f(\text{prox}_{[\lambda\Omega]}(\y_{(k)})) \leq f(\y_{(k)}) + \Delta_{(k)}^\top \nabla f(\y_{(k)}) + \frac{\tilde{L}}{2} \| \Delta_{(k)} \|_2^2,$
\STATE $\quad$with $\tilde{L}\defin L_{k}\nu^{s_k}$ and $\Delta_{(k)}\defin \y_{(k)}\!-\!\text{prox}_{[\lambda\Omega]}(\y_{(k)})$.
\STATE $L_{k}    \leftarrow  L_{k-1} \nu^{s_k}$.
\STATE $\w_{(k)}  \leftarrow \text{prox}_{[\lambda\Omega]}(\y_{(k)})$.
\STATE $t_{k+1}    \leftarrow (1+\sqrt{1+t_k^2})/2$.
\STATE $\y_{(k+1)} \leftarrow \w_{(k)} + \frac{t_k-1}{t_{k+1}} (\w_{(k)} - \w_{(k-1)})$.
\STATE $k \leftarrow k + 1$.
\ENDWHILE
\STATE \textbf{Return}: $\w \leftarrow  \w_{(k-1)}$.
\end{algorithmic}
\vspace*{0.2cm}
{\bf Procedure} \texttt{computeDualityGap}($\w$)
\begin{algorithmic}[1]
 \STATE $\kappab \leftarrow  -\rho^{-1} \nabla\! f(\X\w),\ \text{with}\ \rho\defin \max\big\{\lambda^{-1}\Omega^*(\X^\top\! \nabla\! f(\X\w)),1\big\}$.
 \STATE \textbf{Return}: $f(\X\w)+ \lambda \Omega(\w) + f^*(-\kappab)$.
\end{algorithmic}
\end{algorithm}

In our experiment, we choose the line-search parameter $\nu$ to be equal to $1.5$.
\section{Speed comparison of Algorithm \ref{algo:prox} with parametric max-flow algorithms}\label{appendix:exp}
\label{sec:speed_comp}
As shown by~\citet{hochbaum}, min-cost flow problems, and in particular, the
dual problem of~(\ref{eq:prox_problem}), can be reduced to a specific
\emph{parametric max-flow} problem.  We thus compare our approach (ProxFlow)
with the efficient parametric max-flow algorithm proposed by \citet{gallo} and
a simplified version of the latter proposed by~\citet{babenko}.  We refer to
these two algorithms as GGT and SIMP respectively.  The benchmark is
established on the same datasets as those already used in the experimental
section of the paper, namely: (1) three datasets built from overcomplete bases
of discrete cosine transforms (DCT), with respectively $10^4,\ 10^5$ and $10^6$
variables, and (2) images used for the background subtraction task, composed of
57600 pixels.  For GGT and SIMP, we use the \texttt{paraF} software which is a
\texttt{C++} parametric max-flow implementation available at
\texttt{http://www.avglab.com/andrew/soft.html}. Experiments were conducted on
a single-core 2.33 Ghz.
We report in the following table the average execution time in seconds of each
algorithm for $5$ runs, as well as the statistics of the corresponding problems:\\
\begin{center}
\begin{tabular}{|c||c|c|c|c|}
\hline
Number of variables $p$ & $10\,000$ & $100\,000$ & $1\,000\,000$ & $57\,600$ \\
\hline\hline
$|V|$ & $20\,000$ & $200\,000$ & $2\,000\,000$ & $57\,600$ \\
\hline
$|E|$ & $110\,000$ & $500\,000$ & $11\,000\,000$ & $579\,632$ \\ 
\hline
ProxFlow (in sec.)  & $\mathbf{0.4}$ & $\mathbf{3.1}$ & $\mathbf{113.0}$ & $\mathbf{1.7}$ \\
\hline
GGT (in sec.) & $2.4$ & $26.0$ & $525.0$ & $16.7$ \\
 \hline
SIMP (in sec.) & $1.2$ & $13.1$ & $284.0$ & $8.31$ \\ 
\hline
\end{tabular} 
\end{center}
Although we provide the speed comparison for a single value of $\lambda$ (the
one used in the corresponding experiments of the paper), we observed that our
approach consistently outperforms GGT and SIMP for values of $\lambda$
corresponding to different regularization regimes.

\bibliography{mairal11a}

\end{document}